\documentclass[letterpaper,10pt]{article}
\pdfoutput=1
\usepackage[margin=1in]{geometry}
\usepackage{graphicx}
\usepackage{pdflscape}

\usepackage[style=alphabetic,%
sorting=anyt,%
maxnames=4,maxalphanames=4,%
backref=true,%
backend=biber,bibencoding=auto]{biblatex}%
\usepackage[utf8]{inputenc}
\DeclareUnicodeCharacter{012D}{\u{\i}}

\usepackage[noDcommand,leqno]{kpfonts}
\usepackage{amsxtra}
\usepackage{amsthm}
\usepackage{microtype}

\usepackage[all,pdf,tips,2cell]{xy}
\SelectTips{cm}{}
\UseAllTwocells

\usepackage{comment}
\usepackage{xspace} 
\usepackage{xcolor}
\definecolor{theblue}{rgb}{0.2,0.04,0.7}%
\definecolor{thered}{rgb}{0.8,0.04,0.07}%
\definecolor{thegreen}{rgb}{0.06,0.44,0.08}%
\definecolor{thegrey}{gray}{0.5}%
\definecolor{theshade}{gray}{0.92}%
\usepackage[%
colorlinks=true,
citecolor=thegreen,
linkcolor=theblue,
urlcolor=thered,
plainpages=false,
pdfpagelabels,bookmarks]{hyperref}

\swapnumbers
\newtheorem*{theorem*}{Theorem}
\newtheorem{theorem}{Theorem}[section]
\newtheorem{proposition}[theorem]{Proposition}
\newtheorem*{proposition*}{Proposition}
\newtheorem{lemma}[theorem]{Lemma}
\newtheorem{corollary}[theorem]{Corollary}
\newtheorem{lemma-def}[theorem]{Lemma-Definition}

\theoremstyle{definition}
\newtheorem{definition}[theorem]{Definition}
\newtheorem{example}[theorem]{Example}
\newtheorem*{example*}{Example}

\newtheorem*{construction*}{Construction}
\newtheorem{remark}[theorem]{Remark}
\newtheorem*{remark*}{Remark}
\newtheorem{remarks}[theorem]{Remarks}

\newtheorem*{variant*}{Variant}

\newtheorem*{notation*}{Notation}

\numberwithin{equation}{section}%
\newcommand{\strong}[1]{\textbf{#1}}

\newcommand{\bfl}{butterfly\xspace}

\newcommand{\bfls}{butterflies\xspace}
\newcommand{\Bfls}{Butterflies\xspace}

\newcommand{\iso}{\simeq}
\newcommand{\lto}{\longrightarrow}
\newcommand{\To}{\Rightarrow}

\newcommand{\isoto}{\overset{\sim}{\to}}
\newcommand{\lisoto}{\overset{\sim}{\lto}}

\newcommand{\cat}[1]{\mathbf{#1}}
\newcommand{\C}{\cat{C}} 
\newcommand{\T}{\cat{T}} 
\newcommand{\one}{\cat{1}} 
\newcommand{\cN}{\cat{N}}\newcommand{\NN}{\cN}
\newcommand{\cS}{\cat{S}}\renewcommand{\SS}{\cS} 
\newcommand{\bicat}[1]{\mathfrak{#1}}
\newcommand{\bB}{\bicat{B}}
\newcommand{\bC}{\bicat{C}}
\newcommand{\bD}{\bicat{D}}

\newcommand{\stack}[1]{\mathscr{#1}}

\newcommand{\stG}{\stack{G}}
\newcommand{\stH}{\stack{H}}
\newcommand{\stK}{\stack{K}}
\newcommand{\stM}{\stack{M}}

\newcommand{\stR}{\stack{R}}

\newcommand{\Cat}{\bicat{Cat}}

\newcommand{\xmod}{\bicat{BXMod}}
\newcommand{\grst}{\bicat{BGrSt}}
\newcommand{\sxmod}{\bicat{SXMod}}
\newcommand{\sgrst}{\bicat{SGrSt}}
\DeclareMathOperator{\tors}{Tors}

\DeclareMathOperator{\sttors}{\textsc{Tors}}

\DeclareMathOperator{\grdHom}{\cat{Hom}}
\DeclareMathOperator{\stHom}{\textsc{Hom}}

\DeclareMathOperator{\grdbiext}{\cat{Biext}}
\DeclareMathOperator{\stbiext}{\textsc{Biext}}

\DeclareMathOperator{\grdbiadd}{\cat{Hom}}
\DeclareMathOperator{\stbiadd}{\textsc{Hom}}
\DeclareMathOperator{\grdnadd}{\cat{Hom}}
\DeclareMathOperator{\stnadd}{\textsc{Hom}}
\DeclareMathOperator{\grdbfl}{\cat{Biext}}
\DeclareMathOperator{\grdnbfl}{\cat{MExt}}
\DeclareMathOperator{\stbfl}{\textsc{Biext}}
\DeclareMathOperator{\stnbfl}{\textsc{MExt}}

\newcommand{\HML}{\mathit{HML}}

\newcommand{\del}{\partial}
\renewcommand{\epsilon}{\varepsilon}
\renewcommand{\phi}{\varphi}

\newcommand{\braces}[2]{\lbrace #1,#2\rbrace}
\newcommand{\abs}[1]{\lvert#1\rvert}

\newcommand{\binmid}{\mathbin{|}}

\newcommand{\field}[1]{\ensuremath{\mathbb{#1}}}
\newcommand{\ZZ}{\field{Z}}

\newcommand{\GG}{\field{G}}

\DeclareMathOperator{\Aut}{Aut}
\DeclareMathOperator{\B}{\mathit{B}} 

\DeclareMathOperator{\coker}{coker}

\DeclareMathOperator{\Hom}{Hom}

\DeclareMathOperator{\Id}{Id}
\DeclareMathOperator{\id}{id}
\DeclareMathOperator{\im}{Im}
\renewcommand{\Im}{\im}
\DeclareMathOperator{\Ker}{Ker} 
\DeclareMathOperator{\M}{M\mspace{-1.5mu}}
\DeclareMathOperator{\N}{\mathit{N}} 

\author{Ettore Aldrovandi\footnote{\url{aldrovandi@math.fsu.edu}}\\
  \small Department of Mathematics, Florida State University}

\title{Biextensions, bimonoidal functors, multilinear functor calculus, and categorical rings}
\date{}

\begin{document}
\maketitle
\begin{abstract}
  We associate to a bimonoidal functor, i.e. a bifunctor which is monoidal in each variable, a nonabelian version of a biextension. We show that such a biextension satisfies additional triviality conditions which make it a bilinear analog of the kind of spans known as butterflies and, conversely, these data determine a bimonoidal functor. We extend this result to $n$-variables, and prove that, in a manner analogous to that of butterflies, these multi-extensions can be composed. This is phrased in terms of a multilinear functor calculus in a bicategory.  As an application, we study a bimonoidal category or stack, treating the multiplicative structure as a bimonoidal functor with respect to the additive one. In the context of the multilinear functor calculus, we view the bimonoidal structure as an instance of the general notion of pseudo-monoid. We show that when the structure is ring-like, i.e. the pseudo-monoid is a stack whose fibers are categorical rings, we can recover the classification by the third Mac~Lane cohomology of a ring with values in a bimodule.
\end{abstract}

\phantomsection  
\addcontentsline{toc}{section}{Introduction}  
\section*{Introduction}
\label{sec:introduction}

Let $\stH$ and $\stG$ be monoidal stacks in a topos $\T$ and $F$ a monoidal functor $F\colon \stH\to \stG$. It is well known that $F$ can be represented by a special kind of span. More precisely, if the monoidal laws are group-like, or if we restrict to the invertible objects (restricting to the invertible objects does not affect the characteristic class), we can find presentations $H_\bullet$ for $\stH$ and $G_\bullet$ for $\stG$ by crossed modules of $\T$ so that $F$ is represented by a diagram of group-objects of the form
\begin{equation*}
  \xymatrix@-1pc{%
    H_1 \ar[dd] \ar[dr] && G_1 \ar[dd] \ar[dl] \\ & E \ar[dl] \ar[dr]\\ H_0 && G_0
  }
\end{equation*}
where the salient feature is that the sequence $G_1\to E\to H_0$ of objects of $\T$ is an extension of $H_0$ by $G_1\to G_0$: an exact sequence in which the conjugation action of $E$ on $G_1$ is compatible with that of $G_0$ \cite{ButterfliesI}. Such an extension is conveniently described in geometric terms as a $G_1$-bitorsor $E$ over $H_0$, with the property that one of the actions is obtained by way of the crossed module structure of $G_\bullet$~\cite{MR92m:18019}. A converse of this correspondence is also available, establishing an equivalence between the groupoid of monoidal functors from $\stH$ to $\stG$ and that of diagrams like the one above (a morphism between two diagrams with the same wings $H_\bullet$ and $G_\bullet$ is a homomorphism of $E\to E'$ compatible with all the maps).

One of the main results of this paper is an analogous result for \emph{bi}monoidal functors $F\colon \stH\times \stK\to \stG$, that is, bifunctors which are monoidal in each variable.

The concept of bimonoidal functor requires $\stG$ to be braided.  In addition, $F$ must satisfy a compatibility condition requiring that applying the monoidal condition on both of its variables in the two possible orders lead to the same result. (We can think of $F$ as an obvious generalization of a bilinear map, in which case this condition is trivially satisfied.)  This condition is formally equal to the compatibility one between the two partial multiplication laws of a biextension \cite{MR0257089,MR0354656-VII}. Recall that for abelian groups $A$, $B$, $C$ of $\T$, a biextension of $B$ and $C$ by $A$ consists of an $A$-torsor $E$ over $B\times C$ equipped with two partial commutative and associative product laws, making $E$, for each generalized point $b\in B$, an extension of $C$ by $A$, and, similarly, for each $c\in C$, an extension of $B$ by $A$. The coincidence between these conditions is not purely formal: the bilinear point of view is to regard a biextension as providing a bifunctor $\phi_E\colon B\times C\to \tors(A)$, which is then monoidal (i.e.\ a homomorphism) in each variable. Notable examples are certain duality pairings where, in particular,  $C=B^\vee$ and $A=\GG_m$ \cite[see, e.g.\ ][]{MR2734333,MR2600696}.

Our result begins from a generalization of these ideas, starting with that of biextension.  Thus, for two groups $H$ and $K$ of $\T$, and a braided crossed module $(G_1,G_0)$, we define a biextension of $H,K$ by $(G_1,G_0)$ as a $G_1$-bitorsor (or, more appropriately, using the terminology of \cite{MR92m:18019}, a $(G_1,G_0)$-torsor) $E$ over $H\times K$ such that, for each point $x\in H$ (resp.\ $y\in K$), $E$ is an extension of $K$ (resp.\ of $H$) by the crossed module $G_1\to G_0$.  The braiding is required by the compatibility between the two partial product laws of $E$. The relative diagram is formally the same as in the classical situation, with some differences due to the fact that none of the groups involved is assumed abelian—although the crossed module $G_1\to G_0$ and its associated stack are braided.

Now, let $F$ be a bimonoidal functor $F\colon \stH\times \stK\to \stG$, where each of $\stK,\stH$, and $\stG$ has a presentation by a crossed module. We show that $F$ determines a biextension of $H_0,K_0$ by the crossed module $G_1\to G_0$, equipped with a pair of compatible trivializations for the two pullbacks along the maps $\del\times\id\colon H_1\times K_0\to H_0\times K_0$ and $\id\times \del\colon H_0\times K_1\to H_0\times K_0$.  This is what we call, in this context, a \bfl—the bilinear version of the notion introduced in~\cites()(over a point){ButterfliesI}[earlier][]{Noohi:weakmaps}.  Viceversa, given a biextension $E$ of $(H_0,K_0)$ by $G_1\to G_0$, we obtain a bimonoidal functor
\begin{equation*}
  \phi_E \colon H_0\times K_0 \lto \sttors(G_1,G_0),
\end{equation*}
where the right hand side denotes the stack of $(G_1,G_0)$-torsors equivalent to $\stG$, and where $H_0$ and $K_0$, are interpreted as discrete monoidal stacks.  The additional data provided by the two compatible trivializations allow to conclude that $\phi_E$ is compatible with descent along the presentations $\xymatrix@1@C-1pc{H_0\ltimes H_1 \ar@<0.5ex>[r]\ar@<-0.4ex>[r]& H_0} \to \stH$ and $\xymatrix@1@C-1pc{K_0\ltimes K_1 \ar@<0.5ex>[r]\ar@<-0.4ex>[r]& K_0} \to \stK$ by the groupoids determined by their respective crossed modules, so it determines a bimonoidal functor $F_E\colon \stH\times \stK\to \stG$.  With the obvious notion of morphism, we obtain:
\begin{theorem*}[Theorem~\ref{thm:1} in the main text]
    There is an equivalence of (pointed) groupoids
  \begin{equation*}
    u\colon \grdbfl(H_\bullet,K_\bullet;G_\bullet) \lisoto
    \grdbiadd(\stH,\stK;\stG),
  \end{equation*}
  where the right hand side denotes the groupoid of bimonoidal functors, and left hand side that of biextensions equipped with the aforementioned trivializations.
\end{theorem*}
This result is actually valid over a variable object $S$ of $\T$, hence we have a similar statement where the left and right side above are replaced by the corresponding stacks, which one obtains by letting $S\to \grdbfl(H_\bullet\rvert_S,K_\bullet\rvert_S;G_\bullet\rvert_S)$, and similarly for the right hand side.

Bimonoidal categories or stacks provide examples of bimonoidal morphisms.  If $\stR$ is bimonoidal, it has two monoidal structures, say $\boxplus$ and $\boxtimes$, satisfying an appropriate set of axioms \cite{MR0335598,MR0335599}; the distributivity one, in  particular, says that $\boxtimes$ is a bimonoidal functor $\boxtimes\colon \stR\times \stR\to \stR$ with respect to the other structure, $\boxplus$. Now, assuming $\boxplus$ to be group-like (see \cite{MR3010546} for passing from a merely additive monoidal to a full group-like one, i.e. from ``rig'' to ``ring'') and $\stR$ to have a presentation of the form $\xymatrix@1@C-1pc{R_0\ltimes R_1 \ar@<0.5ex>[r]\ar@<-0.4ex>[r]& R_0} \to \stR$ as above, the bifunctor $\boxtimes$ can be described by a biextension $E_\boxtimes$ of $(R_0,R_0)$ by $R_1\to R_0$.  We are interested in extracting informations about $\stR$, especially of a cohomological nature such as the characteristic class, from the biextension $E_\boxtimes$. Informations of this kind,  which reduces to (usually complicated) cocycle calculations, ordinarily come from the coherence diagrams of $\boxtimes$, among others.  These diagrams involve morphisms like $\boxtimes\circ (\boxtimes \times \id)$ and $\boxtimes\circ (\id \times \boxtimes)$, which are monoidal in each of the three variables, and higher iterations of compositions involving $\boxtimes$ and $\id$. Ideally, each of those multi-functors corresponds to some kind of iteration of $E_\boxtimes$. So we need to generalize the representation of bimonoidal functors by biextensions to an arbitrary number of variables.

Extending the concept of bimonoidal functor to $n$ variables is immediate, the only difference being the same compatibility condition we have in the $n=2$ case must hold for each pair of variables. Therefore, for the multivariable analog of the right hand side in the above statement, the only substantial difference is in the bookeping aspect, and we immediately see that monoidal functors in $n$ variables can be composed—provided we restrict ourselves to considering braided objects. While this is easy to verify, it has the far-reaching consequence that braided monoidal categories, or more generally stacks, comprise the 2-categorical analog of a multicategory \cite{MR2770076,Shulman2010Talk} we denote $\M\grst$, where $\grst$ is the 2-category they form relative to ordinary unary monoidal functors.

As for the left hand side of the equivalence in the theorem, we must extend the concept of biextension to $n$ variables and investigate whether such objects admit a composition law. We define a multi-extension, or an $n$-extension, by $G_1\to G_0$ as a $(G_1,G_0)$-torsor (same as for biextensions) over an $n$-fold product, this time equipped with $n$ partial product laws which are required to be pairwise compatible in the same manner as those in a biextension. (The generalization of a biextension to $n$ variables, in the fully abelian context, is outlined in a remark in~\cite{MR0354656-VII}.) Given $n$ crossed modules ${H}_{i,\bullet}=(\del\colon H_{i,1}\to H_{i,0})$, $i=1,\dots, n$, we are interested in the $n$-extensions of $H_{1,0}\times \dots \times H_{n,0}$ by $G_1\to G_0$ which are also equipped with $n$ compatible trivializations of their pullbacks along each of the $n$ morphisms $\del_i=\id\times \dots \times \del\times\dots \times \id$.  We call them $n$-\bfls.  Now, a direct extension of the above theorem (cf.\ Theorem~\ref{thm:4} below) provides an equivalence 
\begin{equation*}
  \grdnbfl(H_{1,\bullet},\dots, H_{n,\bullet};G_\bullet) \lisoto
  \grdnadd(\stH_1,\dots,\stH_n;\stG)
\end{equation*}
of pointed groupoids, where each $H_{i,\bullet}$ is a presentation of the corresponding $\stH_i$, and $G_\bullet$ of $\stG$, the left hand side is the groupoid of $n$-extensions, and the right hand side denotes the groupoid of $n$-monoidal functors $\stH_1\times\dots \times \stH_n\to \stG$.

These multi-\bfls can be composed in a way reminiscent of the unary case discussed in \cite{ButterfliesI,Noohi:weakmaps}, that is by ``wing juxtaposition,''  which turns out to be associative up to coherent isomorphism. Wing juxtaposition (we informally use this term thanks to the shape of the composition diagrams in the unary case) is an operation which allows to associate to \bfls $F_1,\dots,F_n,E$, where $E$ is an $n$-\bfl, a new one, denoted $E(F_1,\dots,F_n)$, whose underlying bitorsor is a quotient of the pullback of $E$ to the product $F_1\times\dots \times F_n$—each $F_i$ is equipped with an equivariant section with values in the $i^{\mathrm{th}}$ factor in the base of $E$; the pullback is along the product of these maps.  The upshot is that braided crossed modules of $\T$ now form a bi-multicategory (i.e. \emph{bi}categorical analog of a multi-category) of their own if we use these multi-\bfls as morphisms. In addition, the composition is compatible with the previous equivalence. More precisely, we have:
\begin{theorem*}[Theorem~\ref{thm:5}, Propositions~\ref{prop:10} and~\ref{prop:11}, and Theorem~\ref{thm:6}]
  Braided crossed modules of $\T$, equipped with the groupoids $\grdnbfl(H_{1,\bullet},\dots, H_{n,\bullet};G_\bullet)$ as $\Hom$-categories, form a bi-multicategory $\M\xmod$.  Further, there is an equivalence of bi-multicategories
  \begin{equation*}
    \M\xmod \lisoto \M\grst,
  \end{equation*}
  induced by the associated stack functor, carrying each $\grdnbfl(H_{1,\bullet},\dots, H_{n,\bullet};G_\bullet)$ to $\grdnadd(\stH_1,\dots,\stH_n;\stG)$.
\end{theorem*}
This is the multi-categorical analog of the unary case treated in \cite{ButterfliesI}\footnote{In fact the unary version proved therein holds without any commutativity assumption.}, which states the bicategory $\xmod$ of braided crossed modules, equipped with groupoids of spans as categories of morphisms, is equivalent (as a bicategory) to $\grst$. The proof, in particular that of Theorem~\ref{thm:5}, is geometrical, unlike that of corresponding result in~\cite{ButterfliesI,Noohi:weakmaps}. It requires showing first that the pullback of the bitorsor $E$ to the product $F_1\times\dots \times F_n$ is equivariant under the left and right actions of $H_{1,1}\times \dots \times H_{n,1}$, and that it descends to the base of $F_1\times\dots \times F_n$; second, that this descended bitorsor has the structure of a multi-\bfl.

We can now discuss bimonoidal categories from the broader and more convenient perspective afforded by the idea of a bi-multicategory. If $\M\bC$ is a bi-multicategory, we informally look into monoids with respect to the multi-composition structure of $\M\bC$. Formally, we say that an object $X$ of $\M\bC$ is a (pseudo-)monoid of $\bC$, the underlying bicategory of $\M\bC$, if it is a pseudo-algebra over the \emph{club} $\NN$, where $\NN$, after~\cite{MR0340371}, denotes the natural numbers equipped with the category structure described in section~\ref{sec:formalism}; a club is a multicategory with one object.\footnote{For the sake of simplicity we are ignoring the ``extraordinary'' structure given by the action of permutations.}  The monoidal structure carried by $X$ is of the ``unbiased'' kind \cite{MR2094071}, namely we have an $n$-ary operation $m_n\colon X\times \dots \times X\to X$ for each object $n\in\NN$, as opposed to a privileged binary one, with the various coherence conditions being taken care of by the pseudo-algebra structure.

We apply the previous observation to $\bC=\grst$ or $\bC=\xmod$. We say that $\stR$ is weakly ring-like (for want of a better name) if it is a pseudo-monoid of $\grst$.  The ``weakly'' adverb refers to the fact that the underlying categorical group structure of $\stR$ is only braided commutative, whereas we usually define categorical rings as having underlying braided \emph{symmetric} categorical groups \cite[see e.g.\ ][]{MR2369166}. Thus ``weak'' does not mean ``not strict,'' but it signals the commutative law of the additive structure of $\stR$ is weaker than usually required.

We make a similar definition for a pseudo monoid in $\xmod$: this is a novel object comprised of a braided crossed module $R_\bullet$ and multi-extensions $E_n$ of $(R_0,\dots,R_0)$ ($n$ factors) by $R_1\to R_0$ for each $n\in \NN$, subject to the conditions dictated by the pseudo-algebra structure.  We show that even in the realm of a genuine bicategory such as $\xmod$, we obtain coherence conditions resembling a weak version of Mac~Lane's pentagon (see sections~\ref{sec:formalism-monoid} and~\ref{sec:pentagons}, in particular diagram~\eqref{eq:74}).

Since pseudo-monoids are transported across equivalences, pseudo-monoids in $\grst$ are presented by those in $\xmod$. We have:
\begin{proposition*}[Proposition~\ref{prop:13}, Proposition~\ref{prop:14}, and Corollary~\ref{cor:1}]
  Let $\stR$ be (weakly) ring-like. Let $\xymatrix@1@C-1pc{R_0\ltimes R_1 \ar@<0.5ex>[r]\ar@<-0.4ex>[r]& R_0} \to \stR$ be a presentation by a braided crossed module. Then $R_\bullet$ is equipped with biextensions $E_n$, $n\in\NN$ as above so that $(R_\bullet, E_\bullet)$ is a pseudo-monoid in $\xmod$. Furthermore, if $A=\pi_0(\stR)\iso \coker\del$ and $M=\pi_1(\stR)\iso \ker\del$, then $A$ is a (possibly non-unital) ring of $\T$ and $M$ is an $A$-bimodule. The converse also holds.
\end{proposition*}
We may ask whether the weakly ring-like structure is strong enough to force the underlying braiding to become symmetric.  The answer is affirmative, at least in the more interesting (to us) case where the multiplicative structure has a unit object. As we shall see, this is tied to the seemingly unrelated problem of computing a characteristic class for weakly categorical rings.

Our procedure is a bit nonstandard in two aspects. The first is that to compute the class corresponding to $\stR$ we only choose local data for the underlying braided group-like structure of $\stR$, or which is the same thing, for the presentation. We compute the rest of the cocycle from the pseudo-monoidal structure carried by the presentation, by examining the biextension's behavior with respect to the chosen local data.

The second aspect is that we choose to initially dispense with the symmetry condition, therefore we start with a choice of local data for the underlying braided group-like structure of $\stR$. This yields an invariant in the Eilenberg Mac~Lane group $H^4(K(A,2),M)$, lifting to $H^5(K(A,3),M)$ if braiding is symmetric \cite{MR1250465,MR1702420}.\footnote{In a general topos we should interpret them as hypercohomology groups. See the discussion in \cites[\S 2]{MR0258842}[\S 7]{MR95m:18006}[\S 6]{MR1702420}.} In fact the initial choice for the local data steers the computation into being based on the Eilenberg Mac~Lane iterated bar construction \cite{MR0056295,MR0065162}, in place of the more customary Mac~Lane's $Q$-construction. (In the stack context, the iterated bar construction emerges as part of the decomposition of $\stR$ as a gerbe over $A$, together with its additional braiding structure \cite{MR95m:18006,MR1702420}.)

Over a point, a standard categorical ring gives rise to an element in the third Mac~Lane cohomology group $\HML^3(A,M)$ \cites()(for general definitions){MR2369166}[a gap was filled in][]{MR3085798}[see also][chap.\ 13]{MR1600246}, so the question is whether
our procedure yields anything standard, and in particular the expected invariant in $\HML^3(A,M)$.

That it does, at least in the unital case, is due to the fact that, by \cite[\S 11]{MR0098773}, the third Mac~Lane cohomology can also be computed by employing the  infinite bar construction $B^\infty(A) \coloneqq B(A,1) \subseteq B(A,2)\subseteq \dotsm$ instead of the cubical complex $Q(A)$.\footnote{Recall that classically Mac\~Lane cohomology is defined as the Hochschild cohomology of the complex $Q(A)$. The latter computes the stable homology of the Eilenberg-Mac\ Lane spaces, that is, $H_q(Q(A))\iso H_{n+q}(K(A,n))$, $n>q$.} One needs a product structure on $B^\infty(A)$, explicitly up to degrees corresponding to the subcomplex $B(A,2)$, to use in the multiplicative bar construction.  It follows that we can plug one of $B(A,2)$, $B(A,3)$, or $B^\infty(A)$ to calculate a cohomology of $A$ with values in the bimodule $M$. In the unital case, all these choices give rise to isomorphic cohomologies.  The product structure on $B(A,2)$, which is unusual, is explicitly given in Mac~Lane's work. We recover that product, and hence the rest of the structure, directly from the pseudo-monoidal structure attached to the presentation of $\stR$, namely the biextension plus the various compatibility conditions it satisfies.

We do not obtain the representing cocycle for the class of $\stR$ in the form corresponding to class in $\HML^3(A,M)$ right away. Instead, we find that the class of $\stR$ is represented by a twisted cocycle in the multiplicative bar resolution for $B(A,2)$, but the twisting disappears in the unital case, which yields the desired result. In particular this implies that the braiding is necessarily symmetric.  More precisely, we have:
\begin{theorem*}[Theorem~\ref{thm:7}]
  There is a bijective correspondence between equivalence classes of  weak ring-like stacks $\stR$, with $A=\pi_0(\stR)$ and $M=\pi_1(\stR)$, and twisted classes in $\widetilde H^3_2(A,M)$ defined below (cf.\ Definition~\ref{def:8}).  In the unital case, i.e.\ when the external monoidal structure of $\stR$ has a unit element, $A$ is likewise unital and the underlying braiding of $\stR$ is necessarily symmetric. Hence the weak ring-like structure of $\stR$ is fully ring-like, and $[\stR]\in \widetilde H^3_2(A,M)\iso \HML^3(A,M)$.
\end{theorem*}
Hence, unital categorical rings in the sense of the present paper \emph{are} categorical rings tout-court and, by \cite{MR0098773}, we re-obtain their classification in terms of $\HML^3(A,M)$. Had we chosen to work with braided symmetric objects from the beginning there is little doubt the same procedure—i.e.\ a decomposition of the underlying group-like stack followed by an analysis of the attached biextension—would have yielded a Mac~Lane cocycle in the right group right from the start. However, part of the interest was to test to what extent our framework successfully captures the standard theory.

\subsection*{Organization of the material}
\label{sec:org}
We have collected various preliminary items in section~\ref{sec:preliminaries}, including a quick review of standard biextensions. The recapitulation of the relation between group-like stacks and their presentations by crossed modules, though standard, involved a contravariance issue which becomes relevant, therefore we provided a brief outline.

Aside from the preliminary section, we can divide the rest in roughly three parts, plus another containing some appendices.

We develop the generalization of biextensions and their symmetry properties in sections~\ref{sec:biadditive-morphisms} to~\ref{sec:properties}, and the \bfl special kind, including the representation of bimonoidal functors, from section~\ref{sec:butterflies} to~\ref{sec:comm-struct}.

The extension to an arbitrary number of variables takes sections~\ref{sec:auxiliary} to~\ref{sec:bi-mult-braid}.  In particular, the composition of \bfls is discussed in section~\ref{sec:comp-n-bfls}. There, Theorem~\ref{thm:5} establishes the existence of the wing juxtaposition operation.  Its proof is technically involved bookkeeping, but it ultimately is straightforward; it is reproduced in full because, while the juxtaposition product has features similar to the unary case, the proof itself is not an immediate generalization of the corresponding one in \cite{ButterfliesI,Noohi:weakmaps}. Note also that the constructions in the proof are used in the cocycle computations of sect.~\ref{sec:cohom-class}.

The bi-multicategories comprised by braided group-like stacks on one side, and braided crossed modules of $\T$ on the other, are discussed in section~\ref{sec:bi-mult-braid}.

The idea of categorical ring, or ring-like stack, and its presentation as a pseudo-monoids in their respective bicategories is expounded in sections~\ref{sec:bi-mult-braid} and~\ref{sec:pres-categ-ring}, where we also discuss some specific facts about the presentations. In section~\ref{sec:decomp} we discuss the cohomology of rings and the computation of the characteristic class. We do this in slightly simplified form by omitting the additional computations which arise because the choices for the various trivializations required for the cocycle computations are necessarily local. The complete computations are deferred to Appendix~\ref{sec:hyper}.

To avoid unnecessary and possibly long detours, parts of the material have been placed in a number of appendices. Some of this material is necessary for self-consistency but known to the experts. Thus, bi-multicategories are discussed in appendices~\ref{sec:formalism} and~\ref{sec:formalism-monoid}, the analog of the pentagon in a bicategory in~\ref{sec:pentagons}, and  appendix~\ref{sec:lemmas} contains just some technical lemmas pertaining to section~\ref{sec:properties}. Finally, the resolution of a simplicial object by hypercovers and the complete hyper-cohomology computations necessary for a full calculation of the characteristic class are contained in Appendix~\ref{sec:hyper}.

\subsection*{What to read}
\label{sec:what-read}

One possibility is to only read Part~\ref{part:part1} followed by Part~\ref{part:part3}, if the reader is willing to only skim through Part~\ref{part:part2}.  For the latter, it is possible to just read the statements, in particular for the $n$-fold composition of multi-extensions theorem proved in sect.~\ref{sec:comp-n-bfls}, only referring to the proof of~\ref{thm:5} when needed in sect.~\ref{sec:cohom-class}. Alternatively, one can read Part~\ref{part:part3}, in particular section~\ref{sec:decomp}, only skimming through the previous two, and omitting most of section~\ref{sec:decomp}. Parts~\ref{part:part2} and~\ref{part:part3} depend on the multivariable functor calculus in a bicategory, an account of which is contained in the first two appendices, which may only be referred to when needed. Most their content should be known to the experts.  An account of what the pentagonal axiom would look like in a bicategory is contained in Appendix~\ref{sec:pentagons}, and in general a reader will need only equation~\eqref{eq:73} and diagram~\eqref{eq:74}. While in section ~\ref{sec:cohom-class}, the reader can refer to Appendix~\ref{sec:hyper} for the full hyper-cohomology arguments.

\subsection*{Notations and conventions}
\label{sec:notat-conv}

The convention we use is to equate ``additive'' with ``monoidal,'' therefore the term ``biadditive'' means ``bimonoidal'' in the sense used above in the introduction.  By extension, ``$n$-additive'' means ``$n$-monoidal,'' i.e.\ monoidal in each of the $n$-variables. 

A notation of the form $G_\bullet$ or $\mathsf{G}$ denotes a crossed module $(G_1,G_0,\del)$, where $\del \colon G_1\to G_0$ is the homomorphism, and the (right) action is denoted element-wise by $(x,g)\mapsto g^x$, but is not otherwise labeled.  Very often we will simply write $(G_1,G_0)$.

For ease of notation we will often use the convention: $(G_1,G_0)\cong (G,\Pi)$ and, later in the paper, $(R_1,R_0)\cong (R,\Lambda)$ when we discuss ring-like structures.

Up to section~\ref{sec:bimon-struct}, all monoidal structures are notated multiplicatively, with $I$ denoting the identity object.  We switch to an additive notation for one of the two monoidal structures in a bimonoidal situation. In that case, the identity object is $0$. If this is the structure for which we construct a presentation by a crossed module, then we use an additive notation for the groups in the crossed module, even though they are by no means assumed commutative.

The notation $G_S$ denotes the base-change of $G$ to $S$, that is, $S\times G \to S$. If $G$ is a group object, then $G_S$ is group over $S$, with group law $(s,g)(s,g')= (s,gg')$, in set-theoretic notation.

\subsection*{Acknowledgements}
\label{sec:ack}
I would like to thank Sandra Mantovani, Giuseppe Metere, Cristiana Bertolin, and Enrico Vitale for interesting discussions and the kind hospitality extended to me by their institutions in Milan, Turin, and Louvain-la-Neuve, where parts of this work were presented.

Special thanks are also due to the anonymous referee, whose pointed criticisms and inquiries led to a much improved version of this work.

\tableofcontents

\part{}\label{part:part1}
\section{Preliminaries}
\label{sec:preliminaries}

In the following, let $\T$ be a topos which is assumed to be $\mathrm{Sh}(\C)$ for a site $\C$. For the reader's convenience we recall some well know facts on monoidal stacks and their relations with crossed modules of $\T$.

\subsection{Group-like stacks, crossed modules, and bitorsors}
\label{sec:cat-groups-cross}

We follow \cites{MR92m:18019}{ButterfliesI}, to which we refer for further details.
A \strong{crossed module} of $\T$ consists of a group homomorphism $\del \colon G\to\Pi$ of $\T$ and a right action of $\Pi$ on $G$ such that: (1) $\del$ is a morphism of right $\Pi$-objects, with $\Pi$ acting on itself by conjugation; (2) the action of $G$ on itself induced by the $\Pi$-action via $\del$ coincides with conjugation. If $g,h$ are (generalized) points of $G$ and $x$ of $\Pi$ we have the familiar conditions:
\begin{align*}
  \del (g^x) &= x^{-1}\,\del(g)\, x \\
  g^{\del (h)} &= h^{-1}\, g\, h,
\end{align*}
where the exponent notation denotes the action.

The crossed module $\mathsf{G}$ gives rise to a groupoid $\Gamma \colon \xymatrix@1@C-1pc{\Pi\ltimes G \ar@<0.5ex>[r]\ar@<-0.4ex>[r]& \Pi}$, and hence, via the nerve construction, to a simplicial group $\N \mathsf{G}$ whose object in degree $p$ is $(\dotsm (\Pi\ltimes G)\ltimes \dots \ltimes G)$ ($p$ factors). The groups $\pi_1=\ker\del$ and $\pi_0=\coker\del$ are the homotopy (sheaves of) groups of $G_\bullet$ are in fact the only two nonzero homotopy sheaves of $\N \mathsf{G}$. It is easily verified that $\pi_1$ is abelian and central in $G$ and that $\pi_0$ is the sheaf of connected components of the groupoid corresponding to $\mathsf{G}$ (and hence of $\N\mathsf{G}$).

The group laws of $G$ and $\Pi\ltimes G$ equip the groupoid $\Gamma$ with a structure of strict categorical group, that is, a morphism
\begin{equation*}
  m\colon \Gamma\times \Gamma \lto \Gamma
\end{equation*}
of groupoids satisfying the usual axioms of a group-like, strict monoidal category \cite[see also][]{MR1250465}.

The associated stack construction performed on $\Gamma$ yields a stack $\stG$ of $\T$, which inherits the—now lax—group-like structure. (``The'' associate stack is only unique up to equivalence. The specific model we consider is the one whose fibered categories consist of descent data for $\Gamma$ and their morphisms—see \cite[Lemme 3.2]{MR1771927}. It is easy to see that in the case at hand these descent data are cocycles with values in $\Gamma$—see\cite[\S 3.3-4]{ButterfliesI}.) Thus, $\stG$ is a monoidal stack, in the sense that there exists a stack morphism
\begin{equation*}
  m\colon \stG\times \stG\lto \stG,
\end{equation*}
and an associator $\mu\colon m\circ (m\times 1)\Rightarrow m\circ (1\times m)$ 
satisfying the axioms of a monoidal category. We specify the rest of the group-like structure by requiring the existence of a unit object $I$, which we can identify with a morphism $I\colon e\to \stG$ (where $e$ is the terminal object of $\T$), and that left and right-multiplications by a fixed object $x$ induce functorial equivalences
\begin{equation*}
  m(x,-)\colon \stG\lto \stG\,\quad m(-,x)\colon \stG\lto \stG.
\end{equation*}
The group-like stack obtained in this way from a crossed module will be denoted $[\del\colon G\to\Pi]\sptilde$. 

There is a compelling geometric picture for the associated stack obtained in this way from a crossed module. Recall the stack $\sttors(G,\Pi)$ is the stack of right $G$-torsors $P$ whose structure group extension $P\wedge^{G} \Pi$ is isomorphic to the trivial torsor \parencite[see][]{MR546620}. Equivalently, an object of this stack is a pair $(P,s)$ where $P$ is a $G$-torsor and $s\colon P\to \Pi$ a $G$-right-equivariant morphism, where $G$ acts on $\Pi$ on the right via multiplication.  Objects of $\sttors(G,\Pi)$ are called $(G,\Pi)$-torsors.  The notion makes sense for each group homomorphism, but if $\del\colon G\to\Pi$ is a crossed module, then $\sttors(G,\Pi)$ becomes group-like. Indeed, each torsor is also equipped with a \emph{left} $G$-action defined by
\begin{equation*}
  g u = u g^{s(u)},
\end{equation*}
where $u\in P$ and $g\in G$, which makes it into a $G$-\emph{bi}\hspace{0.05pc}torsor. The monoidal structure is given by the contracted product
\begin{equation*}
  m((P,s),(Q,t)) = (P\wedge^{G}Q,st),
\end{equation*}
and it is easily proved to be group-like.  Finally, the stack $\sttors(G,\Pi)$ is equivalent, as  a group-like stack, to $[\del\colon G\to \Pi]\sptilde$, obtained via the associated stack construction described above.

Conversely, for any group-like stack $\stG$ there exists a \emph{presentation} by a crossed module, namely there exists a crossed module $\mathsf{G}=(G,\Pi,\del)$ and an equivalence $\stG\iso [\del\colon G\to \Pi]\sptilde$.  For any group-like stack $\stG$ we can define $\pi_1=\Aut_\stG(I)$, and $\pi_0$ as the sheaf associated to the presheaf of connected components. Upon choosing a presentation, its homotopy sheaves are isomorphic to the ones of $\stG$.

In view of the previous geometric interpretation, $\stG$ is also equivalent to $\sttors(G,\Pi)$. In addition, the sequence
\begin{equation}
  \label{eq:1}
  \xymatrix@1@C-0.5pc{G \ar[r]^\del & \Pi \ar[r]^\pi & \stG},
\end{equation}
where each object is regarded as a stack, is exact in a homotopical sense.  If $\stG$ is identified with $\sttors(G,\Pi)$, the projection $\Pi\to \stG$ sends $x\in \Pi$ to the bitorsor $(G,x)$, that is the trivial right $G$-torsor whose left action on itself is given by $g \cdot u = g^x u$, with $u,g\in G$. (This follows from the fact that $x$ is identified with the equivariant section that assigns $x$ to the unit section $e_{G}$.)  In particular, $(G,e_{\Pi})$ can be identified with the unit object of $\stG$.
\begin{remark}
  \label{rem:1}
  Since a morphism $\phi\colon (P,s)\to (Q,t)$ in $\sttors(G,\Pi)$ has the form
  \begin{equation*}
    \xymatrix@-1pc{%
      P \ar[dr]_s \ar[rr]^\phi & & Q \ar[dl]^t \\ & \Pi
    }
  \end{equation*}
  it is immediately seen that the equivalence
  \begin{equation*}
    [\del\colon G\to \Pi]\sptilde \lto \sttors(G,\Pi)
  \end{equation*}
  is \emph{contravariant.}  In particular, if $g\colon x\to x'$ is a morphism in the strict categorical group $\Gamma$, that is, $x'=x\, \del g$, the corresponding morphism of $(G,\Pi)$-torsors is $(G,x') \to (G,x)$.\footnote{There appears to be no good way to get around the issue.  A ``fix'' is to replace the crossed module with a left one. This restores the expected direction of the arrows, at the cost of turning one of the monoidal laws into the opposite one \cite[see e.g.][after Théorème 4.5]{MR92m:18019}.}
\end{remark}
A \strong{morphism} of group-like stacks is a stack morphism $F\colon \stH\to \stG$ preserving the monoidal structure.  A morphism of crossed modules determines one between the associated group-like stacks in the obvious way. In the converse direction, a morphism $F\colon \stH\to \stG$ only determines a morphism in the \emph{homotopy category} between corresponding crossed modules. This morphism can be represented by a \strong{butterfly,} namely a diagram of group objects of $\T$ of the form:
\begin{equation*}
  \xymatrix@-1pc{%
    H \ar[dd]_\del \ar[dr]^\sigma && G \ar[dd]^\del \ar[dl]_\imath \\
    & E \ar[dl]_\pi \ar[dr]^\jmath \\
    \Sigma && \Pi
  }
\end{equation*}
(where the vertical arrows are crossed modules) from which a fraction representing the morphism $F\colon \stH\to \stG$ can be obtained (see loc.~cit.\ and below for more details).

\subsection{Braidings}
\label{sec:braid-cross-mod}

A \strong{braided crossed module,} \cite[see][]{MR1250465,ButterfliesI}, is a crossed module $\mathsf{G}=(G,\Pi,\del)$ equipped with a bracket
\begin{equation*}
  \braces - -  \colon \Pi\times \Pi\lto G
\end{equation*}
such that $\del \braces x y = y^{-1}x^{-1} y x$, for all points $x,y$ of $\Pi$. The bracket satisfies:
\begin{equation}
  \label{eq:2}
\begin{aligned}
  \braces x {yz} &= {\braces x y}^z \braces x z,& \braces x {\del h} & = h^{-1}h^x, \\
  \braces {xy} z &= {\braces y z} {\braces x z}^y,& \braces {\del g} y &= g^{-y}g,
\end{aligned}
\end{equation}
for all $x,y,z\in \Pi$ and $g,h\in G$. These properties all arise from the observation that the bracket $\braces - -$ corresponds to a braiding in the usual sense for the categorical group determined by $G_\bullet$. Thus  $c_{x,y}=\braces x y$ gives a family of functorial isomorphisms
\begin{equation*}
  c_{x,y} \colon xy \lto yx
\end{equation*}
for each pair $(x,y)$ of objects in $\Gamma$, so that $c\colon m\circ \tau\Rightarrow m\colon \Gamma\times \Gamma\to \Gamma$, where $\tau$ is the interchange functor. The properties above can be derived from functoriality and Mac~Lane's hexagon diagrams. Conversely, if $\Gamma$ is a braided strict categorical group, setting $\braces x y = y^{-1}x^{-1}c_{x,y}\colon e\to y^{-1}x^{-1} y x$ defines a braiding on the corresponding crossed module.

The braiding is \strong{symmetric} if it has the property that $\braces y x={\braces x y}^{-1}$ for all $x,y\in \Pi$. A symmetric braiding is \strong{Picard} if in addition it satisfies the condition $\braces x x = e$. These conditions match the corresponding ones for the categorical group $\Gamma$.  As observed in \textcite[\S 1]{MR95m:18006}, for a Picard crossed module the bracket is a full lift of the commutator map.

The braided, symmetric, and Picard structures translate in the expected manner to the associated group-like stack $\stG\iso [\del\colon G\to \Pi]\sptilde$, and conversely, if $\stG$ is a braided (resp.\ symmetric, Picard) stack with presentation given by $[\del\colon G\to \Pi]$, then the latter acquires a braided (resp.\ symmetric, Picard) structure as above. Let us observe here that a braiding on $\stG=\sttors(G,\Pi)$ induces one on the crossed module by way of the butterfly representing the morphism $m\colon \stG\times \stG\to \stG$. In the Picard case this leads to two possible presentations: by a braided crossed module satisfying the Picard condition, or, according to \textcite{0259.14006}, by a length-one complex of abelian sheaves. We refer the reader to \cite[\S 7]{ButterfliesI} for full details on this correspondence. Here we limit ourselves to observe that if $(P,s)$ and $(Q,t)$ are two $(G,\Pi)$-torsors, the choice of two sections $u\in P$ and $v\in Q$ allows us to write the braiding morphism $c_{P,Q}\colon P\wedge^GQ\to Q\wedge^GP$ as $c_{P,Q}(u\wedge v) = (v\wedge u)\, \chi_{u,v}$, where $\chi\colon P\wedge^G Q\to G$, is a coordinate representation of $c_{P,Q}$ subject to certain equivariance conditions dictated by the requirements:
\begin{align*}
  c_{P,Q}(gu\wedge v) &= g\, c_{P,Q}(u\wedge v),\\
  c_{P,Q}(ug\wedge v) &= c_{P,Q}(u\wedge g\, v),\\
  c_{P,Q}(u\wedge vg) &= c_{P,Q}(u\wedge v)\,g.\\
\end{align*}
\begin{lemma}
  \label{lem:1}
  The coordinate $\chi\colon P\wedge^GQ\to G$ has the expression:
  \begin{equation}
    \label{eq:3}
    \chi_{u,v} = \braces{s(u)}{t(v)}^{-1}.
  \end{equation}
\end{lemma}
\begin{proof}[Proof (Sketch)]
  The relation~\eqref{eq:3} is immediate for $(P,s)=(G,x)$ and $(Q,t)=(G,y)$, where $x,y\in \Pi$, using Remark~\ref{rem:1}.  The general case follows by descent by exploiting the equivariance conditions for $\chi$ and for $\braces--$ recalled above.
\end{proof}

\subsection{Biextensions}
\label{sec:biextensions}

Biextensions were introduced in \textcite{MR0257089}, and later reexamined by \textcite{MR0354656-VII}. We refer to the latter and the text by \textcite{MR823233} for details on the standard concept of biextensions. We briefly recall the basic definitions, and later extend them to introduce crossed modules as coefficients.

Let $H$, $K$, $G$ be groups of $\T$, with $G$ assumed to be abelian. A \strong{biextension} of $H$, $K$ by $G$ is a $G_{H\times K}$-torsor $E$ on $H\times K$ equipped with two partial composition laws $\times_1$ (resp.\ $\times_2$) making $E$ a central extension of $H_K$ (resp. $K_H$) by $G_K$ (resp.\ by $G_H$). This means that $\times_1$ (resp.\ $\times_2$) is only defined on $E\times_KE$ (resp.\ $E\times_H E$), where $E\times_KE$ (resp.\ $E\times_H E$) means fiber product over $K$ (resp.\ $H$) with respect to the obvious projection. These composition laws are required to be compatible with one another.  Usually one also requires $\times_1$ and $\times_2$ to be commutative, which make sense whenever $H$ and $K$ are abelian. Our definition, in which $\times_1$ and $\times_2$ are not required to be commutative is then referred to as a \strong{weak biextension.} Even when $H$ and $K$ are commutative, this notion of biextension is weaker than that of loc.~cit.\ in which the partial multiplication laws are commutative.

Analogously to extensions, the composition laws can be represented by morphisms of $G$-torsors. Let $\B H$ and $\B K$ be the standard classifying simplicial objects of $\T$.\footnote{That is, $\B H=\N H[1]$, where $H[1]$ is the groupoid with a single object $e$ and $\Aut(e)=H$.} The bisimplicial object $\B H\times \B K$ has face maps $d^h$ and $d^v$ \parencite[see e.g.][]{MR1711612}. In particular, for $i=0,1,2$ we consider $d^h_i\colon H\times H\times K\to H\times K$ and $d^v_i\colon H\times K\times K\to H\times K$. Then $\times_1$ and $\times_2$ correspond to morphisms of $G$-torsors
\begin{subequations}
  \label{eq:4}
  \begin{align}
    \gamma^1\colon & {d^h_2}^*E \wedge^G {d^h_0}^*E \lto {d^h_1}^*E \\
    \intertext{and}
    \gamma^2\colon & {d^v_2}^*E \wedge^G {d^v_0}^*E \lto {d^v_1}^*E .
  \end{align}
\end{subequations}
The associativity and commutativity diagrams for $\gamma^1$ and $\gamma^2$ are the obvious ones. More interesting is the compatibility condition of the two structures, which can be expressed as follows.

Let $E_{h,k}$ be the fiber of $E$ over a generalized point $(h,k)$ of $H\times K$. Then the compatibility condition reads
\begin{equation}
  \label{eq:5}
  \vcenter{%
  \xymatrix@C-2pc{%
    & E_{h,k}E_{h',k}E_{h,k'}E_{h',k'} \ar[rr]
    \ar[dl]_(.6){\gamma^1_{h,h';k}\gamma^1_{h,h';k'}} & &
    E_{h,k}E_{h,k'}E_{h',k}E_{h',k'} \ar[dr]^(.6){\gamma^2_{h;k,k'}\gamma^2_{h';k,k'}} \\
    E_{hh',k} E_{hh',k'} \ar[drr]_{\gamma^2_{hh';k,k'}}
    & & & & E_{h,kk'}E_{h',kk'} \ar[dll]^{\gamma^1_{h,h';kk'}}\\
    & & E_{hh',kk'}
  }}
\end{equation}
where we have suppressed the torsor contraction symbol. The horizontal arrow is the canonical symmetry map swapping the factors in $\sttors(G)$, which exists whenever $G$ is abelian.  A more intrinsic way to express the same thing is to observe that the above compatibility condition amounts to the equality
\begin{equation}
  \label{eq:6}
  \gamma^2_{{d_1^h}^*E} \circ ({d^v_2}^*\gamma^1{d^v_0}^*\gamma^1)
  = (1\times c \times 1) \circ
  \gamma^1_{{d_1^v}^*E} \circ ({d^h_2}^*\gamma^2{d^h_0}^*\gamma^2)
\end{equation}
as morphisms from $(d_0^hd_2^v)^*E \wedge (d_2^hd_0^v)^*E \wedge (d_0^hd_0^v)^*E \wedge (d_1^hd_1^v)^*E$ to $(d_1^hd_1^v)^*E$ over $H\times H\times K\times K$. $c$ is the symmetry morphism which swaps the two terms in the middle.

With the straightforward notion of morphism, when $H$, $K$ are also abelian, biextensions of $\T$ form a Picard category $\grdbiext(H,K;G)$. The relative version of it, where we consider the various $\grdbiext(H_S,K_S;G_S)$ over a variable base $S$, is a Picard stack, denoted $\stbiext(H,K;G)$. As observed in 
\textcite{MR0354656-VII} and \textcite{MR823233}, $\stbiext(H,K;G)$ is biadditive in all three variables.

\subsection{Schreier-Grothendieck-Breen theory of extensions}
\label{sec:schr}

It is well known that an extension $E$ of $K$ with a possibly nonabelian kernel $G$ determines a morphism $\jmath\colon E\to \Aut(G)$ from the action of $E$ on $G$ by conjugation \parencite[see, e.g.][]{maclane:hom}. A refinement of this situation is when $G$ is part of a crossed module $\del\colon G\to \Pi$. Since part of the crossed module data is precisely a  homomorphism $\Pi\to \Aut(G)$, an extension of $K$ by that crossed module, or by $(G,\Pi)$ for short, is a commutative diagram of group objects of the form \parencite[\S 8]{MR92m:18019}
\begin{equation*}
  \xymatrix{%
    1 \ar[r] & G \ar[r]^\imath \ar[d]_\del & E \ar[r]^p \ar[dl]^\jmath &
    K \ar[r] & 1\\
    & \Pi
  }
\end{equation*}
where the row is exact and $\imath(g) e = e \imath (g^{\jmath(e)})$, for points $g$ of $G$ and $e$ of $E$.

An equivalent characterization is that $E$ is a $(G_K,\Pi_K)$-torsor over $K$, equipped with $(G,\Pi)$-torsor isomorphisms
\begin{equation*}
  \gamma_{k,k'} \colon E_k\wedge^{G}E_{k'} \lisoto E_{kk'}
\end{equation*}
satisfying the standard associativity condition \parencite{MR0354656-VII,MR92m:18019,ButterfliesII}, 
such that the equivariant structural morphism to $\Pi$ is a homomorphism.
By loc.~cit.\ and \parencite{ButterfliesII}, such an extension corresponds to a group-like stack morphism $F\colon K\to [G\to \Pi]\sptilde$, where $K$ is identified with the (discrete) group-like stack. In this sense the above diagram is a butterfly representing this morphism.

In all these characterizations the morphism $\jmath$ arises as the structural equivariant section of the $(G,\Pi)$-torsor, whereas $\imath$ is the identification $G\iso E_1$, where the $(G,\Pi)$-torsor $E_1$, isomorphic to the unit one, is the fiber over the unit section of $K$. This identification is explicitly given as $\imath(g) = g\, e_0=e_0\,g$, where $e_0$ is the central section of $E_1$ corresponding to the unit section of $G$. Since $\jmath$ is a homomorphism, we obtain the commutativity of the ``wing'' in the diagram above. The conjugation property satisfied by $\imath$ and $\jmath$ reflects the structure of $(G,\Pi)$-torsor of $E$. In particular, we have that $\imath(g)e$ (resp. $e\imath(g)$), product in $E$, agrees with the left (resp.\ right) $G_1$-action on $E$.

\section{Biadditive morphisms}
\label{sec:biadditive-morphisms}

If $H,K,G$ are groups of $\T$, a biadditive (or bimultiplicative) morphism is a map
\begin{math}
  f\colon H\times K\to G
\end{math}
which is a homomorphism in each variable. For abelian groups this is none other than a $\ZZ$-bilinear morphism. 

More generally, if $\cat{H}$, $\cat{K}$, and $\cat{G}$ are categories, which we assume to be group-like, we say that a bifunctor is \strong{biadditive} (or \strong{bimultiplicative}) if it is monoidal in both variables, in a compatible way.\footnote{A better choice would be to use ``bi-monoidal'' in place of biadditive. The latter is motivated by continuity with the naming convention in \cite{ButterfliesI}, where morphisms between gr-stacks are called additive—a convention derived from the particular but significant Picard case.}  For this, we must also assume $\cat{G}$ be endowed with a braiding $c$.  More precisely, we have:
\begin{definition}
  \label{def:1}
  A bifunctor $F\colon \cat{H}\times \cat{K}\to \cat{G}$ is biadditive if:
  \begin{enumerate}
  \item \label{item:1} it has the structure of additive functor with respect to each variable, namely there exist functorial  (iso)morphisms
    \begin{equation*}
      \lambda^1_{h,h';k} \colon F(h,k)F(h',k) \lto F(hh',k)
      \quad \text{and} \quad
      \lambda^2_{h;k,k'} \colon F(h,k)F(h,k') \lto F(h,kk')
    \end{equation*}
    satisfying the standard associativity conditions and compatibility with the braiding of $\cat{G}$; 
  \item \label{item:2} for all objects $h,h'$ of $\cat{H}$ and $k,k'$ of $\cat{K}$ we have a functorial commutative diagram
    \begin{equation}
      \label{eq:7}
      \vcenter{%
        \xymatrix{%
          \bigl(F(h,k)F(h',k)\bigr)\,
          \bigl(F(h,k')F(h',k')\bigr) \ar[rr]^{\Hat c}
          \ar[d]_{\lambda^1_{h,h';k}\lambda^1_{h,h';k'}} & &
          \bigl(F(h,k)F(h,k')\bigr)\,\bigl(F(h',k)F(h',k')\bigr)
          \ar[d]^{\lambda^2_{h;k,k'}\lambda^2_{h';k,k'}} \\
          F(hh',k) F(hh',k') \ar[r]_{\lambda^2_{hh';k,k'}}
           & F(hh',kk') & F(h,kk')F(h',kk') \ar[l]^{\lambda^1_{h,h';kk'}}
        }}
    \end{equation}
    where the upper horizontal arrow is the ``commuto-associativity'' morphism
    \begin{equation*}
      (x\,y)\,(z\,w) \lto (x\,z)\,(y\,w)
    \end{equation*}
    arising from the braiding of $\cat{G}$ (an explicit definition can be found in \cite{MR2369166});
  \item \label{item:3} the two morphisms $F(I_{\cat{H}}, I_{\cat{K}})\to I_{\cat{G}}$ that can be deduced from the first condition coincide.
  \end{enumerate}
\end{definition}
The second condition ensures that the two possible ways to compute $F(hh',kk')$ agree.  The obvious similarity with Diagram~\eqref{eq:5} will be exploited below.

The definition can be extended to the case where all three categories are braided. In this case, we add to the biadditivity the condition that $F$ be \emph{braided} in each variable, namely that the following diagrams commute:
\begin{equation*}
  \vcenter{%
    \xymatrix{%
      F(h,k)F(h',k) \ar[r]^(.6){\lambda^1} \ar[d]_{c_{F,F}}
      & F(hh',k) \ar[d]^{F(c,k)} \\
      F(h',k)F(h,k) \ar[r]_(.6){\lambda^1} & F(h'h,k)
    }}
  \qquad\text{and}\qquad
  \vcenter{%
    \xymatrix{%
      F(h,k)F(h,k') \ar[r]^(.6){\lambda^2} \ar[d]_{c_{F,F}}
      & F(h,kk') \ar[d]^{F(h,c)} \\
      F(h,k')F(h,k) \ar[r]_(.6){\lambda^2} & F(h,k'k)
    }}
\end{equation*}
A natural transformation $\phi\colon F\Rightarrow F'$ between two biadditive morphisms is a natural transformation of bifunctors such that in each variable the conditions to be a natural transformation of additive functors are satisfied. We denote by $\grdbiadd(\cat{H},\cat{K}; \cat{G})$ the resulting groupoid of biadditive morphisms.

It is clear that the same definitions can be stated in the case of group-like stacks as opposed to categories. Similarly to the pointwise case, we obtain a groupoid $\grdbiadd(\stH,\stK;\stG)$ of biadditive morphisms. For the latter, by way of a standard process of restricting to variable base $S$, we obtain a stack $\stbiadd(\stH,\stK;\stG)$ arising from the various fiber categories $S\to \grdbiadd(\stH\rvert_S,\stK\rvert_S;\stG\rvert_S)$.

\section{Biextensions by braided crossed modules}
\label{sec:biext-braid}

We extend the notion of biextension by adapting it to include braided crossed modules as coefficients.

Let $\mathsf{G}=(\del\colon G\to \Pi, \braces \cdot \cdot )$ be a braided crossed module of $\T$, and let $H$, $K$ be groups of $\T$.
\begin{definition}
  \label{def:2}
  A biextension of $H$, $K$ by $(G,\Pi)$ is a $(G,\Pi)_{H\times K}$-torsor $E$ over $H\times K$ equipped with two partial composition laws $\times_1$ (resp.\ $\times_2$) making $E$ into an extension of $H_K$ (resp.\ $K_H$) by $(G,\Pi)_K$ (resp.\ by $(G,\Pi)_H$)—cf.\ sect.~\ref{sec:schr}. We require the composition laws to be compatible in the manner described by a diagram formally identical to~\eqref{eq:5} in sect.~\ref{sec:biextensions}.
\end{definition}

To be precise, in the last condition of the above definition we must specify that the horizontal arrow in diagram~\eqref{eq:5}, adapted to the present case,
\begin{equation}
  \label{eq:8}
  \Hat c \colon
  (E_{h,k} \wedge^G E_{h',k}) \wedge^G (E_{h,k'} \wedge^G E_{h',k'})
  \lto
  (E_{h,k} \wedge^G E_{h,k'}) \wedge^G (E_{h',k} \wedge^G E_{h',k'})
\end{equation}
is given by the braiding of the stack $\stG=[\del\colon G\to \Pi]\sptilde$ in a manner analogous to the horizontal arrow of diagram~\eqref{eq:7} in Definition~\ref{def:1}. A global version of the compatibility condition is given by equation~\eqref{eq:6}.

The two partial composition laws give rise to morphisms of $(G,\Pi)$-torsors
\begin{align*}
  \gamma^1_{h,h';k} &\colon E_{h,k}\wedge^G E_{h',k} \lisoto E_{hh',k} \\
  \intertext{and}
  \gamma^2_{h;k,k'} &\colon E_{h,k}\wedge^G E_{h,k'} \lisoto E_{h,kk'}
\end{align*}
each satisfying the obvious associativity condition, relative to the relevant variables, which can be obtained, mutatis mutandis, from \parencite[refs.][]{MR0354656-VII,MR823233}.  Again, a global version is given by~\eqref{eq:4}.

The morphism~\eqref{eq:8} can be explicitly computed through the braiding of $\mathsf{G}$ as follows.
\begin{proposition}
  \label{prop:1}
  Let $u,u',v,v'$ be points of $E_{h,k}, E_{h',k}, E_{h,k'}$, and $E_{h',k'}$, respectively. Then we have
  \begin{equation}
    \label{eq:9}
    (u\times_1 u')\times_2 (v\times_1 v') =
    (u\times_2 v)\times_1 (u'\times_2 v') \,
    \braces{f(u')}{f(v)}^{-f(v')},
  \end{equation}
  where $f$ is the equivariant section into the trivial $\Pi$-torsor.
\end{proposition}
\begin{proof}
  Use Lemma~\ref{lem:1} to express the morphism $E_{h',k}\wedge^G E_{h,k'}\to E_{h,k}\wedge^G E_{h',k}$. The right action by $f(v')$ arises from the relation between the left and right $G$-torsor structures of $E$.
\end{proof}
\begin{example}
  \label{ex:1}
  A standard biextension $E$ of $H,K$ by $G$ (all groups assumed commutative) is a biextension in the sense specified above with coefficients in the complex $G[1]: [G \to 0]$.  The interchange law, when expressed in terms of points, takes the standard form
  \begin{equation*}
    (u\times_1 u')\times_2 (v\times_1 v') =
    (u\times_2 v)\times_1 (u'\times_2 v').
  \end{equation*}
  Dropping the commutativity of the partial multiplication laws, $E$ becomes a biextension of the sort mentioned in the ``variants'' in \cite{MR0354656-VII}. For such a biextension each extension determined by the partial multiplication laws is a central extension.
\end{example}
\begin{remark}
  \label{rem:2}
  We do not require the partial multiplication laws to be commutative, even if the groups $H$ and $K$ are. If we do require the multiplication laws to be commutative, then $G$ must be abelian and the composite of $\del\colon G\to \Pi$ with $\Pi\to \Aut(G)$ arising from the crossed module structure, factors through zero. This follows from the fact that the strict commutativity forces the braiding to be identically equal to the identity morphisms, and hence, by Lemma~\ref{lem:1}, we have $\chi=1$. By evaluating on torsors of the form $(G,x)$ we obtain $\braces \cdot \cdot = 1$ identically on $\Pi\times \Pi$. Then from the properties~\eqref{eq:2} it follows at once that $G$ and $\Pi$ must be abelian. As a consequence, the action of $\del g$, for any $g\in G$, is trivial, therefore the biextension reduces to one of the type mentioned in Example~\ref{ex:1} above. See however below for a more appropriate notion of commutativity.
\end{remark}
\begin{definition}
  A morphism of biextensions of $H,K$ by $(G,\Pi)$ is a morphism of the underlying $(G,\Pi)$-torsors compatible with the partial composition laws. 
\end{definition}
The compatibility with the composition laws means that there are commutative diagrams of $(G,\Pi)$-torsors
\begin{equation*}
  \vcenter{%
    \xymatrix{%
      E_{h,k}\wedge^G E_{h',k} \ar[d]_{\phi\wedge\phi} \ar[r]^{\gamma_1}
      &  E_{hh',k} \ar[d]^\phi \\
      E'_{h,k}\wedge^G E'_{h',k} \ar[r]_{\gamma'_1} & E'_{hh',k} 
    }}
  \qquad
  \text{and}
  \qquad
  \vcenter{%
    \xymatrix{%
      E_{h,k}\wedge^G E_{h,k'} \ar[d]_{\phi\wedge\phi} \ar[r]^{\gamma_2}
      & E_{h,kk'} \ar[d]^\phi \\
      E'_{h,k}\wedge^G E'_{h,k'} \ar[r]_{\gamma'_2} & E'_{h,kk'}.
    }}
\end{equation*}
The \strong{trivial biextension} is $I=(G_{H\times K},e_\Pi)$. More generally, a biextension will be considered trivial if it is isomorphic, in the sense specified above, to the trivial one. Such an isomorphism amounts to the existence of a ``central'' section of the underlying $(G,\Pi)$-torsor, namely a global $s \colon H\times K\to E$ commuting with the left and right actions of $G$ on $E$, obtained as the image by $\phi$ of the unit section $e_G$ of $I$. Then the two above diagrams show that $s$ is a homomorphism in each variable.

With the above notion of morphism and trivial object, biextensions of $H,K$ by $\mathsf{G}$ form a \emph{pointed} groupoid $\grdbiext(H,K;\mathsf{G})$. A standard argument based on descent shows that over a variable base $S$, the collection of all $\grdbiext(H_S,K_S;\mathsf{G}_S)$ forms a stack, equally pointed, denoted $\stbiext(H,K;\mathsf{G})$.  The question of whether either of $\grdbiext(H,K; \mathsf{G})$ and $\stbiext(H,K; \mathsf{G})$ possess a monoidal structure is more delicate, and it will require additional commutativity properties from $\mathsf{G}$, as discussed below.

The groupoid $\grdbiext(H,K;\mathsf{G})$ is evidently additive with respect to the last variable, namely
\begin{equation*}
  \grdbiext(H,K;\mathsf{G}_1\times \mathsf{G}_2) \iso \grdbiext(H,K;\mathsf{G}_1)\times \grdbiext(H,K;\mathsf{G}_2),
\end{equation*}
and the same holds for $\stbiext(H,K;\mathsf{G})$.  Additivity on the other two variables relies on the commutativity of the partial multiplications for the same reason as \cite[\S 1.2]{MR823233}, hence it will not hold in our case.

Biextensions by $(G,\Pi)$ have certain structural properties analogous to those of extensions reviewed in sect.~\ref{sec:schr}, more precisely, we have:
\begin{lemma}
  \label{lem:2}
  If $E\to H\times K$ is a biextension of $H$, $K$ by $(G,\Pi)$, then $E$ is equipped with a map $\jmath\colon E\to \Pi$ which is a homomorphism for each partial composition law. For all points $e\in E_{h,1}$, or $e\in E_{1,k}$, and $g\in G$ we have the relation
  \begin{equation}
    \label{eq:10}
    \imath (g) e = e \imath (g^{\jmath (e)}).
  \end{equation}
  and the diagram
  \begin{equation}
    \label{eq:11}
    \begin{gathered}
      \xymatrix{%
        G \ar[d]_\del \ar[r]^\imath & E \ar[r]^(.4)p \ar[dl]^\jmath & H\times K \\
        \Pi
      }
    \end{gathered}
  \end{equation}
  commutes, where $\imath$ is the identification $(G,e_\Pi)\isoto E_{1,1}$ as $(G,\Pi)$-torsors.
\end{lemma}
The identification of the fiber $E_{1,1}$ above the unit $(e_H,e_K)$ of $H\times K$ with the unit torsor $I=(G,e_\Pi)$ amounts to the definition of exactness of the row in~\eqref{eq:11}.
\begin{proof}
The arguments are virtually the same as those of \cite{MR0354656-VII}, except for the part concerning the morphism $\jmath$. Following loc.\ cit., we can use the facts about extensions recalled in sect.~\ref{sec:schr} by working over the terminal object $K$ (resp. $H$) when we view $E$ as an extension of $H_K$ (resp.\ $K_H$) by $(G,\Pi)_H$ (resp.\ $(G,\Pi)_K$). 

With this in mind, let $f=(\jmath,p)\colon E\to {\Pi}_{H\times K}$, with $\jmath\colon E\to \Pi$, be the equivariant morphism which is part of the $(G,\Pi)_{H\times K}$-torsor structure. As an extension of $H_K$ by $(G,\Pi)_K$, $E$ has a structure of $(G,\Pi)_K\times_K H_K$-torsor over $H_K$, hence $f \colon E\to \Pi_K\times_K H_K\iso \Pi\times H\times K$ is written as $f=(\jmath_K,p)$, where $\jmath_K\colon E\to \Pi_K$. We have $\jmath_K=(\jmath,p_K)$, where $p_K$ (resp. $p_H$) is the composite of $p$ with the projection to $K$. A similar picture holds by exchanging $H$ with $K$. Since $\jmath_K$ and $\jmath_H$ are homomorphisms for the corresponding product structure, $\jmath$ is a homomorphism for both.

Let $e_{H/K}=(e_H,\id_K)$ be the unit section of $H_K$ and $e_{K/H}=(\id_H,e_K)$ that of $K_H$. Since the fibers $E_{1,k} = e_{H/K}^*E$ and $E_{h,1}=e_{K/H}^*E$ are respectively identified with $G_K$ and $G_H$ as $(G,\Pi)$-torsors, we have the unit section $e_{E/K}$ (resp.\ $e_{E/H}$) in $E_{1,k}$ (resp. $E_{h,1}$). Now let $E_{1,1}=(e_H,e_K)^*E$ be the fiber over the unit section $(e_H,e_K)$ of $H\times K$. This is the common pullback of both $E_{1,k}$ and $E_{h,1}$. The same argument as in \cite[\S 2.2]{MR0354656-VII} shows that the restriction $e_{E/K}$ and $e_{E/H}$  to $(e_H,e_K)$ agree, so let $e_E\in E_{1,1}$ be their common value.  From \cite[\S 2.1]{MR0354656-VII} $e_E$ is central, since both $e_{E/K}$ and $e_{E/H}$ are, and we define $\imath$ by $\imath (g) = g \, e_E = e_E\, g$. Note, $\imath$ is the common restriction of $\imath_K$ and $\imath_H$ respectively defined by $e_{E/K}$ and $e_{E/H}$.

Analogously to the case of extensions, the commutativity of the ``wing'' in~\eqref{eq:11} follows from the equivariance of $f$, and hence of $\jmath$, having observed that $\jmath(e_E)$ must be equal to the unit of $\Pi$.  Finally, the relation~\eqref{eq:10} follows from the definition of the bitorsor structure underlying that of a $(G,\Pi)$-bitorsor (cf.\ the end of sect.~\ref{sec:cat-groups-cross}).
\end{proof}

\begin{remark}
  In the converse direction to Lemma~\ref{lem:2} a biextension of $H,K$ by $(G,\Pi)$ can be recovered from the datum $(E,\times_1,\times_2,\imath,\jmath)$, where $p\colon E\to H\times K$ is an object of $\T/_{H\times K}$ equipped with two partial composition laws satisfying the interchange law, and $\imath$ and $\jmath$ satisfy the conditions of the lemma.
\end{remark}

If $E'$ is a second biextension, with maps $\imath'\colon G\to E$ and $\jmath'\colon E\to \Pi$ satisfying the conditions of Lemma~\ref{lem:2}, a biextension morphism $\phi\colon E\to E'$ makes the following diagram
\begin{equation*}
  \xymatrix{%
    G \ar@/^2pc/[rr]^{\imath'} \ar[d]_\del \ar[r]^\imath & E \ar[r]^(.4)\phi \ar[dl]^\jmath
    & E' \ar@/^1pc/[dll]^{\jmath'}\\
    \Pi
  }
\end{equation*}
commutative.  

For a biextension which is globally of the form $(G_{H\times K},x)$, where $x\colon H\times K\to \Pi$ characterizes the equivariant structural morphism, we obtain a decomposition in terms of a pair of nonabelian cocycles, which we briefly describe.

The partial multiplication laws are described by
\begin{subequations}
  \label{eq:12}
  \begin{align}
    \label{eq:13}
    (h,k,a)(h',k,a') &= (hh',k,g_1(h,h';k)a^{x(h',k)}a'), \\
    \label{eq:14}
    (h,k,b)(h,k',b') &= (h,kk',g_2(h;k,k')b^{x(h,k')}b').
  \end{align}
\end{subequations}
The actions stem from the change from right to left action,
\begin{equation*}
  (h,k,a)(h',k,a') = ((h,k,e_G) a)(h',k,a') = (h,k,e_G) (a(h',k,a'))
  = (h,k,e_G) (h',k,a^{x(h',k)}a'),
\end{equation*}
and similarly for the other one. From the associativity property and the fact that the equivariant section $G_{H\times K}\to \Pi$ must be a homomorphism for both laws, we find that $(g_1,g_2,x)$ must satisfy a pair of nonabelian cocycle conditions:
\begin{subequations}
  \label{eq:15}
  \begin{align}
    \label{eq:16}
    g_1(hh',h'';k)g_1(h,h';k)^{x(h'',k)} &=
    g(h,h'h'';k)g_1(h',h'';k) \\
    \label{eq:17}
    x(h,k)x(h',k) &= x(hh',k)\,\del g_1(h,h';k) \\
    \intertext{and}
    \label{eq:18}
    g_2(h;kk',k'')g_2(h;k,k')^{x(h,k'')} &=
    g_2(h;k,k'k'')g_2(h;k',k'') \\
    \label{eq:19}
    x(h,k)x(h,k') &= x(h,kk')\,\del g_2(h;k,k')  
  \end{align}
\end{subequations}
The cocycles $(g_1,x)$ and $(g_2,x)$ are not independent: from the compatibility between the partial multiplication laws, using Proposition~\ref{prop:1}, we find:
\begin{multline}
  \label{eq:20}
  g_2(hh';k,k')g_1(h,h';k)^{x(hh',k')}g_1(h,h';k') = \\
  g_1(h,h';kk')g_2(h,k,k')^{x(h',kk')}g_2(h';k,k')
  \braces{x(h',k)}{x(h,k')}^{-x(h',k')}.
\end{multline}
These cocycles are a coboundary if there exists $u\colon H\times K\to G$ such that:
\begin{subequations}
  \label{eq:21}
  \begin{align}
    u(hh',k)\, g_1(h,h';k) &= u(h,k)u(h',k)\\
    u(h,kk')\, g_2(h;k,k') &= u(h,k)u(h,k')\\
    x(h,k) &= \del\, u(h,k).
  \end{align}
\end{subequations}

\section{Biadditive morphisms and biextensions}
\label{sec:biadditive-biext}

Let $E$ be a biextension of $H,K$ by $(G,\Pi)$. Then $E$ defines a biadditive morphism
\begin{equation*}
  F_E \colon H\times K\lto \stG,
\end{equation*}
in the sense of Definition~\ref{def:1}, where $\stG$ is the group-like associated stack, by assigning to a point $(h,k)\in H\times K$ the $(G,\Pi)$-torsor $E_{h,k}$.  Indeed, it is easily seen that the isomorphisms $\gamma^1$ and $\gamma^2$ plus the compatibility of the two composition laws expressed by~\eqref{eq:5} satisfy the required conditions. It is also easily verified that a morphism of biextensions $\phi\colon E\to E'$ induces a natural transformation (denoted with the same symbol)
\begin{equation*}
  \phi\colon F_E\Rightarrow F_{E'} \colon H\times K\lto \stG.
\end{equation*}
This establishes a functor
\begin{equation*}
  u\colon \grdbiext(H,K;\mathsf{G}) \lto \grdbiadd(H,K;\stG),
\end{equation*}
and in fact one between stacks
\begin{equation*}
  u\colon \stbiext(H,K;\mathsf{G}) \lto \stbiadd(H,K;\stG),
\end{equation*}
These functors are evidently fully faithful. In the converse direction, we have:
\begin{proposition}
  \label{prop:2}
  Let $F\colon H\times K\to \stG$ be a biadditive morphism, and let $\stG$ have a presentation by way of the braided crossed module $\mathsf{G}=(G,\Pi,\del,\braces\cdot\cdot)$. Then the pullback of the sequence~\eqref{eq:1} along $F$ determines a biextension $E=E_F$ of $H,K$ by $\mathsf{G}$.
\end{proposition}
\begin{proof}
  Let $E = (H\times K)\times_\stG \Pi$ be the pullback. Set $p\colon E\to H\times K$ (resp.\ $\jmath\colon E\to \Pi$) equal to the first (resp.\ second) projection. We claim that $E$ is a biextension of $H,K$ by $\mathsf{G}$.

  The sequence~\eqref{eq:1} can be seen as the universal extension by $(G,\Pi)$.  In particular, $\Pi$ can be identified with the universal $G,\Pi)$-torsor (the identity $\id_\Pi$ is tautologically the structural equivariant section). This makes apparent that the pullback to $H\times K$ is a $(G,\Pi)$-torsor with structural map $\jmath$.  As for the partial multiplication laws, observe that, with the same notation as sect.~\ref{sec:biextensions}, the biadditivity of $F$ amounts to a pair of natural transformations
  \begin{equation*}
    (F\circ d^h_2)\,(F\circ d^h_0) \Longrightarrow (F\circ d^h_1) \colon
    H\times H\times K \lto \stG, \qquad
    (F\circ d^h_2)\,(F\circ d^h_0) \Longrightarrow (F\circ d^h_1) \colon
    H\times K\times K \lto \stG,
  \end{equation*}
  and the equality of two transformations (see eqns.~\eqref{eq:6} and~\eqref{eq:7})
  \begin{equation*}
    (F d_0^hd_2^v) \, (Fd_2^hd_0^v) \, (Fd_0^hd_0^v) \, (Fd_1^hd_1^v)
    \Longrightarrow (Fd_1^hd_1^v) \colon
    H\times H\times K\times K \lto \stG.
  \end{equation*}
  This translates into the required properties for $E$ once we observe that the universality of $\Pi$ (as a $(G,\Pi)$-torsor over $\stG$) implies that for any pair $f,g\colon S\to \stG$ we have an isomorphism $f^*\Pi\wedge^G g^*\Pi\iso (fg)^*\Pi$ of $(G,\Pi)$-torsors. Last, the morphism $\imath$ is the image, under a further pullback along the unit section $(e_H,e_K)\in H\times K$, of $\del\colon G\to \Pi$. It is clear $\imath$ and $\jmath$ have the properties stated in Lemma~\ref{lem:2}.

  In more down-to-earth terms, everything can be explicitly checked by writing out explicit expressions for the points of the pullback. Thus, a point $e\in E$ is a triple $e=((h,k),a,x)$, where $(h,k)\in H\times K$, $x\in \Pi$, and $a\colon F(h,k)\to (G,x)$.  The $G$-action becomes evident by looking at two points $e$, $e'$ in the same fiber. We must have a commutative diagram of $(G,\Pi)$-torsors
  \begin{equation*}
    \xymatrix{%
      F(h,k) \ar[r]^{a'} \ar[dr]^a & (G,x') \ar[d]^g \\ & (G,x)
    }
  \end{equation*}
  where the vertical arrow, as a map of $(G,\Pi)$-torsors whose underlying $G$-torsor is trivial, is completely determined by a point $g\in G$ (the image of the unit section $e_G$) such that
  \begin{equation*}
    x' = x \,\del g.
  \end{equation*}
  The situation for the left action is completely analogous, with $x'=\del g\, x$.
  
  When $(h,k)=(e_H,e_K)$, the point $e_E=((e_H,e_K),l, e_\Pi)$, where $l$ is the unique morphism coming from the third condition in Definition~\ref{def:1}, is the central unit section of $E$. This gives the morphism $\imath\colon G\to E$.  The partial multiplications are obtained as follows. Let $e=((h,k),a,x)$ and $e=((h',k),a',x')$. Then $ee' = ((hh',k),a'', xx')$, where the morphism $a''$ is obtained from the diagram
  \begin{equation*}
    \xymatrix{%
      F(h,k)F(h',k) \ar[d]_{\lambda^1_{hh',k}}\ar[r]^{aa'} & (G,xx') \\
      F(hh',k) \ar[ur]_{a''}
    }
  \end{equation*}
  The second multiplication is defined in the same way. The interchange law immediately follows from the above and~\eqref{eq:7}.
\end{proof}
Since it is clear that if $\phi\colon F\Rightarrow F'$ is a natural transformations of biadditive morphisms by pullback we obtain a corresponding isomorphim $E_F\to E_{F'}$ of biextensions, the construction of Proposition~\ref{prop:2} provides functors
\begin{equation*}
  v\colon \grdbiadd(H,K;\stG) \lto \grdbiext(H,K;\mathsf{G})
\end{equation*}
and
\begin{equation*}
  v\colon \stbiadd(H,K;\stG) \lto \stbiext(H,K;\mathsf{G}).
\end{equation*}
\begin{proposition}
  \label{prop:3}
  The functors $u$ and $v$ are quasi-inverses.
\end{proposition}
\begin{proof}
  We need to show that the functor $u$ is essentially surjective.
  
  For each $F$ biadditive, we construct a morphism $F\to F'=F_E$, where $E=(H\times K)\times_\stG\Pi$. The biadditive morphism determined by $E$ assigns to each $(h,k)\in H\times K$ the $(G,\Pi)$-torsor $E_{h,k}$ consisting of pairs $(f,x)$ such that
  \begin{equation*}
    f\colon F(h,k) \lisoto (G,x).
  \end{equation*}
  Having observed this, the proof proceeds in a manner identical to that of Proposition 4.4.2 in~\cite{ButterfliesI} (in particular, cf. 4.4.2.1 and 4.4.2.2).
\end{proof}
To summarize the previous discussion, $\grdbiadd(H,K;\stG)$ and $\grdbiext(H,K;\mathsf{G})$ are equivalent, \emph{pointed} groupoids. The distinguished point is the trivial biextension, corresponding to the trivial biadditive morphism sending $(h,k)\in H\times K$ to the trivial $(G,\Pi)$-torsor. The same holds for $\stbiadd(H,K;\stG)$ and $\stbiext(H,K;\mathsf{G})$.

\section{Symmetry properties of biextensions}
\label{sec:properties}

It is well known (and see the discussion at the end of \cite{ButterfliesII}) that the groupoid of extensions of $K$ by $(G,\Pi)$ admits a monoidal structure if and only if the crossed module $\del\colon G\to \Pi$ (or equivalently the associated stack $\stG$) is braided. This monoidal structure is itself equipped with a commutativity isomorphism if and only if the braiding carried by $\stG$ is symmetric.

For biextensions the situation is more rigid. This is already apparent from the fact that in the very definition of a biextension a braided structure is required in order to be able to formulate the compatibility condition between the two multiplication laws.

If $E$ and $F$ are biextensions, we can use the braiding of $\stG$ to construct the partial multiplication laws in the expected manner:
\begin{equation}
  \label{eq:76}
  \xymatrix@1@C+1.5pc{%
    (E_{h,k}\wedge^G F_{h,k}) \wedge^G (E_{h',k}\wedge^G F_{h',k})
    \ar[r]^{\Hat c} &
    (E_{h,k}\wedge^G E_{h',k}) \wedge^G (F_{h,k}\wedge^G F_{h',k})
    \ar[r]^(.6){\gamma^1\wedge\mu^1}&
    E_{hh',k} \wedge F_{hh',k} \,,
  }
\end{equation}
and similarly for the second partial multiplication.
\begin{proposition}
  \label{prop:4}
  For any pair of biextensions $E$ and $F$ of $(H,K)$ by $(G,\Pi)$, the partial multiplication laws~\eqref{eq:76} comprise a biextension structure on the contracted  product (as $(G,\Pi)$-torsors) $E\wedge^GF$ if the braiding on the crossed module $\mathsf{G}=(G,\Pi,\del,\braces--)$ (or equivalently $\stG$) is symmetric.
\end{proposition}
\begin{proof}
  With reference to \eqref{eq:76}, let us simply denote the first multiplication morphism so obtained by $\gamma^1\mu^1$, and let $\gamma^2\mu^2$ denote the second multiplication structure obtained in exactly the same way.

  For the compatibility condition between the two multiplications we must write diagram~\eqref{eq:5} for $\gamma^1\mu^1$ and $\gamma^2\mu^2$. To simplify notations and readability, let us suppress the symbol $\wedge^G$ and the associativity morphisms. Set:\ %
  $a=E_{h,k}$, $x=F_{h,k}$,\ %
  $y=E_{h',k}$, $z=F_{h',k}$,\ %
  $u=E_{h,k'}$, $v=F_{h,k'}$,\ %
  $w=E_{h',k'}$, $b=F_{h',k'}$.
  Also, let us use parentheses to denote the various codomains of $\gamma$, $\mu$, namely $(ay)=E_{hh',k}$, $(xz)=F_{hh',k}$, etc. Consider the following diagram:
  
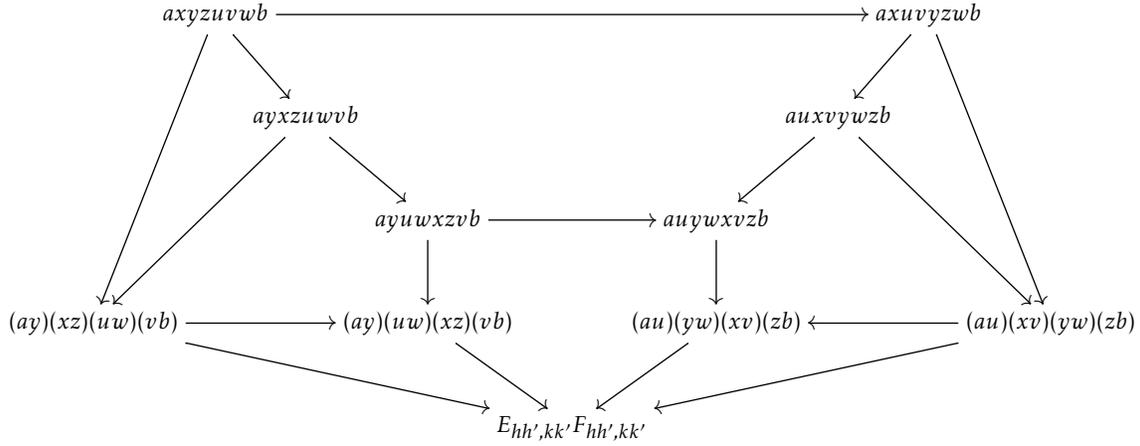
\begin{small}
  \begin{equation*}
    \xymatrix@C-3pc{%
      & axyzuvwb
      \ar[rrrrrr] \ar[dr] \ar[dddl] &&&&&&
      axuvyzwb \ar[dl] \ar[dddr]\\
      & &  ayxzuwvb \ar[dr] \ar[ddll] &&&&
      auxvywzb \ar[dl] \ar[ddrr]\\
      & && ayuwxzvb \ar[rr] \ar[d] &&
      auywxvzb \ar[d] \\
      (ay)(xz)(uw)(vb)
      \ar[rrr] \ar[drrrr] & &&
      (ay)(uw)(xz)(vb)
      \ar[dr] &&
      (au)(yw)(xv)(zb)
      \ar[dl] && &
      (au)(xv)(yw)(zb)
      \ar[lll] \ar[dllll]\\
      & && & E_{hh',kk'}F_{hh',kk'}
    }
  \end{equation*}
\end{small}
The outer pentagon corresponds to the sought-after compatibility for the partial multiplications of $E\wedge^GF$, whereas the little inner one is obtained by juxtaposing the compatibility conditions for $E$ and $F$. Thus the inner pentagon commutes. The triangles commute by definition, since they just hold the definitions of $\gamma^1\mu^1$ and $\gamma^2\mu^2$. The quadrangles are commutative by inspection. This leaves the big hexagon in the middle. The position of the external variables is unchanged throughout, hence the analysis of the diagram reduces to the one in Lemma~\ref{lem:11}. This proves the proposition.
\end{proof}

Once the coefficient crossed module carries a symmetric braiding, we are provided with a product of biextensions. An easy argument, based on the definition of the partial multiplication laws in the proof of Proposition~\ref{prop:4} and Lemma~\ref{lem:10}, or alternatively \cite[\S 8.2]{ButterfliesII}, shows that the braiding provides a morphism
\begin{equation*}
  c\colon E\wedge^GF \lto F\wedge^GE
\end{equation*}
of biextensions, which is evidently symmetric. We summarize the previous discussion as
\begin{proposition}
  For any $H,K$ and $\mathsf{G}=(G,\Pi,\del,\braces--)$, 
  $\grdbiext(H,K;\mathsf{G})$ (resp.\ $\stbiext(H,K;\mathsf{G})$) is a braided symmetric group-like groupoid (resp.\ stack) if and only if \ $\mathsf{G}$ carries a symmetric braiding.\qed
\end{proposition}
By equivalence, the same statement holds for $\grdbiadd(H,K;\stG)$ and $\stbiadd(H,K;\stG)$, respectively; the resulting symmetric  monoidal structure is given by the pointwise product of functors (with respect to the monoidal structure of $\stG$).

Let $\stG$ be symmetric. If $F_1,F_2\colon H\times K\to \stG$ correspond, via Proposition~\ref{prop:3}, to the biextensions $E_1$ and $E_2$, then it is not difficult to see that $E_1\wedge^G E_2$ is the biextension corresponding to $F_1\,F_2 \colon H\times K\to \stG$, where $F_1\,F_2$ is computed pointwise, i.e.\ for all $h\in H$ and $k\in K$ we set $(F_1\,F_2)(h,k)\coloneqq F_1(h,k)\,F_2(h,k)$, where the juxtaposition on the right indicates the monoidal structure in $\stG$.  Thus, Proposition~\ref{prop:3} is upgraded to an equivalence of symmetric group-like groupoids or stacks. More precisely, we get:
\begin{proposition}
  \label{prop:5}
  The assignment defined by $u$ in Proposition~\ref{prop:3} gives an equivalence
  \begin{equation*}
    u\colon \grdbiext(H,K;\mathsf{G}) \lisoto \grdbiadd(H,K;\stG)
  \end{equation*}
  (resp.\ 
  \begin{equation*}
    u\colon \stbiext(H,K;\mathsf{G}) \lisoto \stbiadd(H,K;\stG))
  \end{equation*}
  of symmetric group-like categories (resp.\ stacks).\qed
\end{proposition}

\section{\Bfls}
\label{sec:butterflies}

Let $H_\bullet: H_1\overset{\del}{\to} H_0$ and $K_\bullet: K_1\overset{\del}{\to}K_0$ be a pair of crossed modules, and let 
$\del\colon G_1\to G_0$ be a crossed module equipped with a braiding $\braces--$. We denote by $\stH$, $\stK$, and $\stG$ the corresponding associated stacks.

\begin{definition}
  \label{def:3}
  A \strong{\bfl} from $H_\bullet \times K_\bullet$ to $G_\bullet$ is a biextension $E\in \grdbiext(H_0,K_0;G_\bullet)$ equipped with trivializations of its pullbacks $(\del,\id)^*E$ and $(\id,\del)^*E$, i.e.\ maps $s_1\colon H_1\times K_0\to E$ and $s_2\colon H_0\times K_1\to E$ subject to the following conditions:
  \begin{enumerate}
  \item\label{item:4} (Restriction) $s_1$ and $s_2$ agree on $H_1\times K_1$: $(\id,\del)^*s_1 = (\del,\id)^*s_2$.
  \item\label{item:5} (Compatibility) For all $(h,z)\in H_1\times K_0$ and $(y,k)\in H_0\times K_1$, and $e\in E_{y,z}$, $s_1$ and $s_2$ satisfy:
    \begin{equation}
      \label{eq:22}
      \begin{aligned}
        s_1(h,z) \times_1 e &= e\times_1 s_1(h^y,z)\,,\\
        s_2(y,k) \times_2 e &= e\times_2 s_2(y,k^z)\,.
      \end{aligned}
    \end{equation}
  \end{enumerate}
  A morphism $\phi\colon (E,s_1,s_2)\to (E',s'_1,s'_2)$ is a morphism of the underlying biextensions preserving the trivializations.
  \Bfls form a pointed groupoid, denoted $\grdbfl(H_\bullet,K_\bullet; G_\bullet)$. The distinguished object is the trivial biextension.
\end{definition}
\begin{remark}
  \label{rem:17}
  As noted, a trivialization of a biextension is given by a (central) global section. In the previous definition the pullbacks are $(\del,\id)^*E = (H_1\times K_0)\times_{H_0\times K_0} E$ and $(\id,\del)^*E = (H_0\times K_1)\times_{H_0\times K_0} E$, so, for example, a global section $\hat s_1\colon H_1\times K_0\to (\del,\id)^*E$ is equivalently given by a global map $s_1\colon H_1\times K_0\to E$. Same for $s_2$.
\end{remark}
If we identify the sections corresponding to the trivializations with $s_1\colon H_1\times K_0\to E$ and $s_2\colon H_0\times K_1\to E$, condition~\ref{item:4} in the Definition implies that
\begin{equation*}
    s_1(h,\del k) = s_2(\del h, k).
  \end{equation*}
Let $s\colon H_1\times K_1\to E$ be the resulting morphism. The next lemmas characterize these objects in terms of explicit diagrams (see~\eqref{eq:23} and~\eqref{eq:27} below). Their shape justifies the name.
\begin{lemma}
  \label{lem:4}
  The trivializations $s_1$  and $s_2$, as maps $s_1\colon H_1\times K_0\to E$ and $s_2\colon H_0\times K_1\to E$, render the following diagrams commutative
  \begin{equation}
    \label{eq:23}
        \vcenter{%
      \xymatrix@C-1pc@!C{%
        H_1\times K_0 \ar[dd]_{(\del,\id)} \ar[dr]^{s_1} & &
        G_1 \ar[dd]^\del \ar[dl]_\imath \\
        & E \ar[dl]^p \ar[dr]_\jmath \\
        H_0\times K_0 & & G_0 
      }} \qquad\qquad
        \vcenter{%
      \xymatrix@C-1pc@!C{%
        H_0\times K_1 \ar[dd]_{(\id,\del)} \ar[dr]^{s_2} & &
        G_1 \ar[dd]^\del \ar[dl]_\imath \\
        & E \ar[dl]^p \ar[dr]_\jmath \\
        H_0\times K_0 & & G_0 
      }}
  \end{equation}
  with $\jmath\circ s_1$ and $\jmath\circ s_2$ equal to the trivial map (identically equal to the unit $e_0$ of $G_0$). In addition, $s_1$ and $s_2$ have the following properties:\footnote{In some of the following formulas we explicitly write the the product symbols $\times_1$ and $\times_2$ to avoid ambiguities.}
  \begin{enumerate}
  \item\label{item:6} (Multiplicative) $s_1$ and $s_2$ are multiplicative in each variable, namely
    \begin{equation}
      \label{eq:24}
      \begin{aligned}
        s_1(h,z)\times_1 s_1(h',z) &= s_1(hh',z),  & s_1(h,z)\times_2 s_1(h,z') &= s_1(h,zz'),\\
        s_2(y,k)\times_1 s_2(y',k) &= s_2(yy',k),  & s_2(y,k)\times_2 s_2(y,k') &= s_2(y,kk');
      \end{aligned}
    \end{equation}
  \item\label{item:7} (Central) For all $g\in G_1$, and whenever $\del h$ (resp.\ $\del k$) is equal to the unit section of $H_0$ (resp.\ $K_0$), we have:
    \begin{equation}
      \label{eq:25}
      \begin{aligned}
        s_1(h,z)\imath(g) & = \imath(g)s_1(h,z),\\ s_2(y,k)\imath(g) & = \imath(g)s_2(y,k).
      \end{aligned}
    \end{equation}
  \end{enumerate}
\end{lemma}

\begin{proof}
  The relations~\eqref{eq:24}, as well as the fact that the compositions  $\jmath\circ s_1$ and $\jmath\circ s_2$ must be equal to the unit element, follow from the condition that the pullbacks $(\del,\id)^*E$ and $(\id,\del)^*E$ split as biextensions. Indeed, as it was previously described, the trivialization of, say, $(\del,\id)^*E$ amounts to an isomorphism of biextensions $G_{H_1\times K_0} \iso (\del,\id)^*E$. The image of the unit central section of $G_{H_1\times K_0}$ defines the trivializing one for $(\del,\id)^*E$, and, hence, the map $s_1$. In practice, the isomorphism is written $(h,z,g)\mapsto s_1(h,z)g$, and the central condition on the trivializing section means that
  \begin{equation*}
    s_1(h,z)\, g = g\, s_1(h,z)\,,
  \end{equation*}
  where the juxtaposition indicates the $G_1$-action. Of course, similar considerations hold for $s_2$, as well.

  Then $\jmath\circ s_1$ and $\jmath\circ s_2$ are trivial because, under the isomorphisms determined by $s_1$ and $s_2$, they must indeed be equal to the equivariant section for the trivial bitorsor which simply sends the unit $e_1\in G_1$ to $e_0\in G_1$.

  The relations~\eqref{eq:24} follow from the fact that the isomorphisms determined by $s_1$ and $s_2$ are compatible with the partial multiplication laws. For example:
  \begin{equation*}
    (s_1(h,z)\, g)\times_1 (s_1(h',z)\, g')
    = s_1(h,z)\times_1 (g\,s_1(h',z)\, g')
    = s_1(h,z)\times_1 s_1(h',z)\: gg'\,,
  \end{equation*}
  on the other hand in the trivial bitorsor $(h,z,g)\times_1(h',z,g') = (hh',z,gg')$ is mapped to $s_1(hh',z)\, gg'$, which proves the first of~\eqref{eq:24}. The others are similar.

  The relations~\eqref{eq:25} follow from the centrality and~\eqref{eq:10} of Lemma~\ref{lem:2}.
\end{proof}
\begin{remark}
  \label{rem:3}
  An alternative way to characterize a \bfl is to say that it consists of a biextension $E$ of $(H_0,K_0)$ by $G_\bullet$ such that $s_1$ and $s_2$ provide it with the structure of \bfls in the ordinary sense of \cite{ButterfliesI} from $(H_\bullet)_{K_0}$ (resp.\ $(K_\bullet)_{H_0}$) to $G_\bullet$ in a compatible way.
\end{remark}
It is clear that properties similar to those pertaining to $s_1$ and $s_2$ stated in Lemma~\ref{lem:4} hold for their common restriction $s\colon H_1\times K_1\to E$. In particular, the following diagram
\begin{equation}
  \label{eq:27}
  \vcenter{%
    \xymatrix@C-1pc@!C{%
      H_1\times K_1 \ar[dd]_{(\del,\del)} \ar[dr]^s & &
      G_1 \ar[dd]^\del \ar[dl]_\imath \\
      & E \ar[dl]^p \ar[dr]_\jmath \\
      H_0\times K_0 & & G_0 
    }}
\end{equation}
is commutative, with $\jmath\circ s$ equal to the trivial map (identically equal to the unit $e_0$ of $G_0$). In addition, $s$ is multiplicative in each variable:
\begin{equation}
  \label{eq:28}
  \begin{aligned}
    s(h,k)\times_1 s(h',k) &= s(hh',k)\\
    s(h,k)\times_2 s(h,k') &= s(h,kk').
  \end{aligned}
\end{equation}
Relations~\eqref{eq:25} (collapsed into one) and~\eqref{eq:22} also hold. As an easy consequence we have the following
\begin{lemma}
  \label{lem:5}
  Let $(E,s_1,s_2)\in \grdbfl(H_\bullet,K_\bullet; G_\bullet)$.
  The  pullback $(\del,\del)^*E$ is isomorphic (via $s$) to the trivial biextension in $\grdbiext(H_1,K_1;G_\bullet)$.\qed
\end{lemma}
\begin{remark}
  The correspondence $(E,s_1,s_2)\mapsto (E,s)$ determines a morphism
  \begin{equation*}
    \grdbfl (H_\bullet,K_\bullet; G_\bullet) \lto
    \mathrm{H}\Ker \bigl(
    \xymatrix@1@C+0.1pc{%
      \grdbiext (H_0,K_0;G_\bullet) \ar[r]^{(\del,\del)^*} &
      \grdbiext (H_1,K_1;G_\bullet)
    }\bigr)
  \end{equation*}
  where $\mathrm{H}\Ker$ denotes the homotopy kernel (recall both groupoids are pointed) which is not an equivalence, in general. Requesting that the biextension $E$ become trivial when pulled back to $H_1\times K_1$ is a weaker condition for it does not provide for the two other pullbacks to be trivializable.
\end{remark}
\begin{remark}
  A \bfl has a description in terms of cocycles if the underlying $(G_1,G_0)$-torsor $E$ is globally trivial as a right $G_1$-torsor. Therefore, as seen at the end of sect.~\ref{sec:biext-braid}, if $E$ has the form ($H_0\times K_0\times G_1,x)$, with $x\colon H_0\times K_0\to G_0$, then a cocyclic description~\eqref{eq:12}, \eqref{eq:15}, \eqref{eq:20} (with $h,h'\in H_0$ and $k,k'\in K_0$) is available. A \bfl will described by a cocycle consisting of that set \emph{plus} additional relations for the trivializations of the two pullbacks to $H_1\times K_0$ and $H_0\times K_1$. (The required triviality conditions are the same as those for a nonabelian 1-cocycles of the sort that appears in the theory of extensions.)

  Since $E = H_0\times K_0\times G_1$, there exist $u_1\colon H_1\times K_0\to G_1$ and $u_2\colon H_0\times K_1\to G_1$, such that $s_1$ and $s_2$ are written as $s_1(h,z) = (\del h, z, u_1(h,z)^{-1})$ and $s_2(y,k) = (y, \del k, u_2(y,k)^{-1})$. The conditions in Lemma~\ref{lem:4} then express the triviality of the cocycle pair $(g_1,g_2)$ when pulled back to $H_1\times K_0$ and $H_0\times K_1$, namely we get:
  \begin{subequations}
    \label{eq:29}
    \begin{align}
      u_1(hh',z)\,g_1(\del h, \del h'; z) &= u_1(h,z)\,u_1(h',z) \\
      u_1(h,zz')\,g_2(\del h;z, z') &= u_1(h,z)\,u_1(h,z') \\
      x (\del h,z) &= \del \, u_1(h,z) \\
      \intertext{and}
      u_2(yy',k)\,g_1(y, y'; \del k) &= u_2(y,k)\,u_2(y',k) \\
      u_2(y,kk')\,g_2(y; \del k, \del k') &= u_2(y,k)\,u_2(y,k') \\
      x (y,\del k) &= \del \, u_2(h,z)
    \end{align}
\end{subequations}
  for all pairs $(h,z),(h',z)\in H_1\times K_0$ and $(y,k),(y,k')\in H_0\times K_1$.  Moreover, since the two trivializations must agree when pulled back to $H_1\times K_1$, it follows that $u_1(h,\del k) = u_2(\del h,k)$. Denoting this common restriction by  $u\colon H_1\times K_1\to G_1$, the two previous sets coalesce into
  \begin{subequations}
    \label{eq:30}
  \begin{align}
    \label{eq:31}
    u(hh',k)\, g_1(\del h,\del h';\del k) &= u(h,k)u(h',k)\\
    u(h,kk')\, g_2(\del h;\del k,\del k') &= u(h,k)u(h,k')\\
    x(\del h,\del k) &= \del\, u(h,k).
  \end{align}
\end{subequations}
for all $h,h'\in H_1$ and $k,k'\in K_1$.
\end{remark}

\section{Biadditive morphisms and \bfls}
\label{sec:biadd-morph-butt}

We now consider biadditive morphisms
\begin{equation*}
  F\colon \stH\times \stK\lto \stG
\end{equation*}
with $\stG$ braided (cf.\ sect.~\ref{sec:biadditive-morphisms}).  As before, we assume we have presentations by crossed modules $\stH\iso [H_1\to H_0]\sptilde$, $\stK\iso [K_1\to K_0]\sptilde$, and $\stG\iso [G_1\to G_0]\sptilde$, the latter equipped with a braiding structure $\braces--$. Our purpose is to prove the following
\begin{theorem}
  \label{thm:1}
  There is an equivalence of (pointed) groupoids
  \begin{equation*}
    u\colon \grdbfl(H_\bullet,K_\bullet;G_\bullet) \lisoto
    \grdbiadd(\stH,\stK;\stG).
  \end{equation*}
\end{theorem}
In essence, the theorem states that any biadditive functor $F\colon \stH\times \stK\to \stG$ can be represented by a \bfl involving the presentations. This equivalence is compatible with restriction and base-change so that, via the usual mechanism of considering the above equivalence relative to a variable object $S$ of $\T$, we obtain a corresponding equivalence
\begin{equation*}
  u\colon \stbfl(H_\bullet,K_\bullet;G_\bullet) \lisoto
  \stbiadd(\stH,\stK;\stG).
\end{equation*}
Following~\cite{ButterfliesI} and sect.~\ref{sec:biadditive-biext}, we exhibit a pair of quasi-inverse functors. We will essentially confine ourselves to just exhibit the relevant definitions, for the methods are quite similar to those in \cite[\S 4]{ButterfliesI}, with the exception of biadditivity. The latter is discussed explicitly, where the need arises.

Let $\bar F$ be the composition of the projection $\pi\colon H_0\times K_0\to \stH\times \stK$ with $F$. It is evidently biadditive. Consider the pullback $E=(H_0\times K_0)\times_{\bar F,\stG,\pi}G_0$. By Proposition~\ref{prop:2} it is a biextension.
\begin{lemma}
  \label{lem:6}
  $E$ defined above satisfies the conditions in the statement of Lemma~\ref{lem:4}.
\end{lemma}
\begin{proof}
  We must verify the pullbacks of $E$ to $H_1\times K_0$ and $H_0\times K_1$ are trivializable.  This is a direct consequence of the fact that $\xymatrix@1@C-1pc{H_1\ar[r]& H_0\ar[r] &  \stH}$ and $\xymatrix@1@C-1pc{K_1\ar[r]&  K_0\ar[r] &  \stK}$ are (homotopically) exact.  Furthermore, the decomposition $(\del,\del) = (\id,\del)\circ(\del,\id) = (\del,\id)\circ (\id,\del)$ and condition~\ref{item:3} of definition~\ref{def:1} ensure that the restriction condition in definition~\ref{def:3} is satisfied.

  In explicit terms, a point of the pullback $(\del,\id)^*E$ is given by a triple $((h,z),f,x)$ where $(h,z)\in H_1\times K_0$, $x\in G_0$ and $f$ is an isomorphism
  \begin{equation*}
    \xymatrix@1{F(\pi(\del h),\pi(z)) \ar[r]^(0.6)f& (G_1,x).}
  \end{equation*}
  We have $\pi(\del h)=(H_1,\del h)$, so there must be an isomorphism
  \begin{math}
    I_\stH=(H_1,e_{h_0}) \lto (H_1,\del h).
  \end{math}
  Thus, there is a chain of isomorphisms
  \begin{equation*}
    \xymatrix{%
      F(I_\stH,\pi (z)) \ar[d] \ar[r] & F(\pi(\del h),\pi(z))
      \ar[r]^(0.6)f & (G_1,x) \\
      I_\stG
    }
  \end{equation*}
  so that we must have $x=\del g$, $g\in G_1$. This provides an explicit trivializing section $s_1$ for $(\del,\id)^*E$.  Similarly for the other one, $s_2$, and their common restriction $s$.  

  A computation based on the biadditivity of $F$ and the same technique at the end of the proof of~\ref{prop:2} shows that both $s_1$ and $s_2$ are multiplicative, i.e. each satisfies condition~\ref{item:6} of lemma~\ref{lem:4}, with respect to both variables. Condition~\ref{item:5} also follows from a direct calculation, as in \cite[\S 4.3.6]{ButterfliesI}.
\end{proof}
\begin{proposition}
  \label{prop:6}
  There exists a functor $v\colon \grdbiadd(\stH,\stK;\stG)\to \grdbfl(H_\bullet,K_\bullet; G_\bullet)$ defined by assigning to an object $F$ the \bfl whose underlying bitorsor is $E=(H_0\times K_0)\times_{\bar F,\stG,\pi}G_0$.
\end{proposition}
\begin{proof}
  The pullback biextension construction is functorial (cf.\ the end of section~\ref{sec:biadditive-biext}). In particular, if $\phi\colon F\To F'$ is a morphism of biadditive functors, the resulting morphism of biextensions is compatible with the trivializations. This observation proves Proposition~\ref{prop:6}.
\end{proof}

To get a biadditive functor from a \bfl in functorial way, we proceed as in ref.\ \cite[\S\S 4.3.2–4.3.4]{ButterfliesI}. Specifically, an object of $\stH\times \stK$ will be represented by a pair of torsors, as $((Y,y),(Z,z))$ where $(Y,y)$ is an $(H_1,H_0)$-torsor and $(Z,z)$ an $(K_1,K_0)$-torsor. Let $(E,s)$ be an object of $\grdbfl(H_\bullet,K_\bullet;G_\bullet)$. We define a $F_{E,s_1,s_2}\colon \stH\times \stK\to \stG$ by assigning to the above pair the $(G_1,G_0)$-torsor $(X,x)$ where:
\begin{equation*}
  X\coloneqq \Hom_{H_1,K_1} (Y,Z;E)_{(y,z)}, \qquad
  \begin{aligned}[t]
    x\colon & X \lto G_0 \\
    & e \longmapsto \jmath\circ e.
  \end{aligned}
\end{equation*}
$X$ consists of separately $H_1$ or $K_1$-equivariant local lifts of $(y,z)\colon Y\times Z\to H_0\times K_0$ to $E$.  By this we mean that the lift of $(y\,\del h,z)$ is related to that of $(y,z)$ by $e(y\,\del h,z) = e (y,z)\, s_1(h,z)$. Similarly, $e(y,z\,\del k) = e (y,z)\, s_2(y,k)$. The notion is consistent thanks to Conditions~\ref{item:4} and the multiplicativity of $s_1$ and $s_2$. This construction is obviously functorial with respect to each of its arguments. It gives the defining part of the
\begin{proposition}
  \label{prop:7}
  There exists a functor 
  \begin{equation*}
    u\colon \grdbfl (H_\bullet,K_\bullet;G_\bullet) \lto
    \grdbiadd (\stH,\stK;\stG),
  \end{equation*}
  whose value at $(E,s_1,s_2)$ is the bifunctor $F_{E,s_1,s_2}$ defined above.
\end{proposition}
\begin{proof}
  We must verify the biadditivity property, namely that $F(Y_1,Z)\wedge^{G_1}F(Y_2,Z) \lisoto F(Y_1\wedge^{H_1}Y_2,Z)$, which follows from the diagram:
  \begin{equation*}
    \xymatrix{%
      & & E \times_{K_0} E \ar[r]^{\gamma^1} \ar[d]^{(p,p)} & E \ar[d]^p \\
      Y_1\times Y_2 \times Z \ar[r]^{\Delta_Z} & Y_1 \times Z \times Y_2 \times Z \ar[r] \ar[ur]^{(e_1,e_2)}
      & (H_0\times K_0)\times_{K_0} (H_0\times K_0) \ar[r] & H_0\times K_0
    }
  \end{equation*}
  It is easy (and left to the reader) to check that the diagram is invariant under replacing $(e_1g,e_2)$ by $(e_1,ge_2)$ and $(y_1h,y_2)$ by $(y_1,hy_2)$. (The first arrow to the left is the diagonal of $Z$ followed by the swap of the two inner factors.)
\end{proof}
\begin{remark}
  \label{rem:4}
  A coordinate version of the construction of the functor $u$ above is as follows.  According to the beginning of section~\ref{sec:biadditive-biext}, if $E$ is the underlying biextension of an object in $\grdbfl(H_\bullet,K_\bullet;G_\bullet)$, we obtain a biadditive morphism $H_0\times K_0\to \stG$ by sending the pair $(y,z)$ to the $(G_1,G_0)$-torsor $E_{y,z}$. Since $E$ is part of a \bfl, this construction is compatible with morphisms in $\stH$ and $\stK$ (in fact, in the prestacks defined by the presentations) because, if say $y' = y\, \del h$, we have
  \begin{equation*}
    E_{y',z} \longleftarrow E_{y,z}\wedge^{G_1} E_{\del h,z} \overset{\sim}{\longleftarrow} E_{y,z},
  \end{equation*}
  since from the definition the existence of $s_1$ implies $E_{\del h,z}$ is a trivial $(G_1,G_0)$-torsor. (Similarly for $E_{y,\del k}$.)
  We have a similar calculation whenever $z' = z\,\del k$, and the properties of the \bfl (plus the compatibility of $\times_1$ and $\times_2$) ensure we obtain a unique morphism
  \begin{equation*}
    E_{y,z} \lto E_{y',z'}\,,
  \end{equation*}
  which allows us to define $u(E)$ on more general objects by descent. The connection with the global version above is of course that $\Hom_{H_1,K_1} (Y,Z;E)_{(y,z)}$ reduces to $E_{y,z}$ when $(Y,y)=(H_1,y)$ and $(Z,z)=(K_1,z)$.
\end{remark}
\begin{proof}[Proof of Theorem~\ref{thm:1}]
  We prove that $u$ and $v$ are quasi-inverses. To this end, recall that for a $G_1$-torsor $P$ we have the isomorphism
  \begin{equation}
    \label{eq:32}
    \Hom_{G_1}(G_1,P) \isoto P,\qquad m \mapsto m(e_{G_1}),
  \end{equation}
\cite[Proposition III 1.2.7]{MR49:8992}. This extends to $(G_1,G_0)$-torsors by assigning to $m\colon G_1\to P$ the element $s(m(e_{G_1}))\in G_0$ (see \cite[n.\ 4.4.2]{ButterfliesI}). We apply this observation to both $v(u(E))$ and $u(v(F))$, where $(E,s_1,s_2)$ is a \bfl with underlying biextension $E$ and $F\colon \stH\times \stK\to \stG$ is biadditive. In the first case, a point of the biextension $v(u(E))$ is given by a tuple $((y,z),f,x)$ with $(y,z)\in H_0\times K_0$ and $x\in G_0$, such that
\begin{equation*}
  f\colon E_{y,z} \lisoto (G_1,x).
\end{equation*}
Considering the fiber over $(y,z)$ we have the following chain of morphisms of $(G_1,G_0)$-torsors:
\begin{equation}
  \label{eq:33}
  v(u(E))_{y,z} \lisoto \Hom_{G_1}(G_1,E_{y,z}) \lisoto E_{y,z},
\end{equation}
where the first arrow sends $(f,x)$ to $f^{-1}$, and the projection $(f,x)\to x$, namely the equivariant section, to $\jmath\circ f^{-1}$. Therefore we have obtained an isomorphism of biextensions $v(u(E))\isoto E$, by virtue of the result quoted at the beginning. This is (tautologically) an isomorphism of \bfls. For this, consider the composite isomorphism $I_\stG=(G_1,e)\isoto E_{\del h,z}\isoto (G_1,x)$ (resp.\ $I_\stG=(G_1,e)\isoto E_{y,\del k}\isoto (G_1,x)$) which provides the trivialization of the pullback $v(u(E))$ to $H_1\times K_0$ (resp.\ $H_0\times K_1$), as in the proof of Lemma~\ref{lem:6}. The resulting identifications of the pair $(f,x)$ with an element $g\in G_1$ is clearly compatible via the chain~\eqref{eq:33}, with the trivialization $(G_1,e)\isoto E_{y,z}$.

In the second case, $v(F) = (H_0\times K_0)\times_\stG G_0$ and $u(v(F))$ is the biadditive morphism that assigns to $(y,z)\in H_0\times K_0$ the $(G_1,G_0)$-torsor
\begin{equation*}
  v(F)_{y,z} = \lbrace (f,x) \rvert F(\pi(y),\pi(z))\overset{f}{\lto} (G_1,x) \rbrace\lisoto F(\pi(y),\pi(z)),
\end{equation*}
where the isomorphism on the right is by way of~\eqref{eq:32} and loc.\ cit.  The pullback $v(F)$ has the required trivializations by Lemma~\ref{lem:6}, which are evidently compatible with $F(\pi(\del h),\pi (z))\isoto (G_1,e)$ and $F(\pi(y),\pi(\del k))\isoto (G_1,e)$, showing that the above isomorphism holds for general objects and is functorial, proving there is an isomorphism $u(v(F))\iso F$.
\end{proof}

\begin{theorem}[Theorem~\ref{thm:1}, symmetric version]
  \label{thm:2}
  Let the monoidal structure of $\stG$ be symmetric. Then the equivalence in Theorem~\ref{thm:1} extends to one of group-like groupoids or stacks.
\end{theorem}
\begin{proof}[Proof (Sketch)]
  We need only check the trivializations. Suppose $F$, $F'$ are two biadditive morphisms and  $E$, $E'$ are the corresponding \bfls. Let us also use the letters $E,E'$ to denote the underlying torsors, and let $s_1,s_2$ and $s'_1,s'_2$ be the trivialization morphisms of $E$ and $E'$, respectively. Let us denote by $F\wedge F'$ the biadditive morphism $\stH\times \stK \to \stG$, $(y,z) \mapsto F(y,z)F'(y,z)$. (The unnamed juxtaposition denotes the product structure in $\stG$.) Consider the bitorsor $E\wedge^G E'$, equipped with the biextension structure of Proposition~\ref{prop:4}.

  We claim that the pairs $(s_1,s'_1)$ and $(s_2,s'_2)$ give the two trivialization morphisms turning the biextension $E\wedge^G E'$ into a \bfl.

  Indeed, the restriction condition~\ref{item:4} in Definition~\ref{def:3} is immediate. The compatibility condition~\ref{item:5} can be verified by a simple computation, using the definition of the partial multiplications from the proof of Proposition~\ref{prop:4} and Lemma~\ref{lem:1}.
\end{proof}

\section{Commutative structures}
\label{sec:comm-struct}

With the equivalence between $u\colon \grdbfl(H_\bullet,K_\bullet; G_\bullet)\to\grdbiadd(\stH,\stK;\stG)$ at our disposal, we can discuss commutativity conditions for biextensions (cf.\ Remark~\ref{rem:2} above).
\begin{definition}
  \label{def:4}
  Let $\stH$, $\stK$, and $\stG$ all be equipped with a braiding structure. Then we say that $(E,s_1,s_2)\in \grdbfl(H_\bullet,K_\bullet;G_\bullet)$ (or simply $E$ by abuse of language) is \strong{braided} if the biadditive morphism $u(E)$ is braided in the sense of sect.~\ref{sec:biadditive-morphisms}.
\end{definition}
In order to turn the definition into actual diagrams, we use the explicit variance of the biextension $E$, Remark~\ref{rem:4}. Here the point of Remark~\ref{rem:1} becomes relevant. Whereas the morphism $u(E)\colon \stH\times \stK\to \stG$ is covariant in each variable, the biextension $E$ itself is a \emph{$(G_0,G_1)$-torsor} over $H_0\times K_0$, and $\sttors(G_1,G_0)$ is anti-equivalent to $\stG$.

Using the trivialization $s_1\colon H_1\times K_0\to E$ and $s_2\colon H_0\times K_1\to E$, and the braidings $\braces--_H$ and $\braces--_K$ for the presentations, we have morphisms $\eta^1$ and $\eta^2$:
\begin{align*}
  \eta^1_{y,y';z} \colon E_{y'y,z} & \lto E_{yy',z} &
  \eta^2_{y;z,z'} \colon E_{y,z'z} & \lto E_{y,zz'} \\
  e & \longmapsto e\times_1 s_1(\braces{y'}{y}_H,z)  &
  e' & \longmapsto e'\times_2 s_2(y,\braces{z'}{z}_K).
\end{align*}
Thus, for $y,y'\in H_0$ and $z,z'\in K_0$ the following diagrams must commute:
\begin{equation}
  \label{eq:34}
  \vcenter{%
    \xymatrix{%
      E_{y,z}\wedge^{G_1} E_{y',z} \ar[r]^(0.6){\gamma^1_{y,y';z}} \ar[d]_{c_{y,y';z}}
      & E_{yy',z} \ar@{<-}[d]^{\eta^1_{y,y';z}} \\
      E_{y',z}\wedge^{G_1} E_{y,z} \ar[r]_(0.6){\gamma^1_{y',y;z}}
      & E_{y'y,z} 
    }}\qquad\text{and} \qquad
  \vcenter{%
    \xymatrix{%
      E_{y,z}\wedge^{G_1} E_{y,z'} \ar[r]^(0.6){\gamma^2_{y;z,z'}} \ar[d]_{c_{y;z,z'}}
      & E_{y,zz'} \ar@{<-}[d]^{\eta^2_{y;z,z'}} \\
      E_{y,z'}\wedge^{G_1} E_{y,z} \ar[r]_(0.6){\gamma^2_{y;z',z}}
      & E_{y,z'z} 
    }}
\end{equation}
In both diagrams the vertical arrow to the left comes from the braiding in $\stG$. The two vertical arrows on the right express the functoriality with respect to the braiding structures of $H_\bullet$ and $K_\bullet$.

Alternatively, the directions in the right vertical arrows of both diagrams in~\eqref{eq:34} can be restored if we interpret them as morphism in $\sttors(G_1,G_0)$ arising from the morphisms
\begin{equation*}
  (H_0,yy') \lto (H_0,y'y) \qquad \text{and} \qquad (K_0,zz') \lto (K_0,z'z)
\end{equation*}
in $\sttors(H_1,H_0)$ and $\sttors(K_1,K_0)$, respectively.  This requires writing $\eta^1$ and $\eta^2$ (going in the opposite direction) in terms of
\begin{equation*}
  s_1(\braces{y'}{y}^{-1}_H,z)  \qquad \text{and} \qquad
  s_2(y,\braces{z'}{z}^{-1}_K).
\end{equation*}
Using Lemma~\ref{lem:1} to express the braiding morphisms $c_{y,y';z}$ and $c_{y;z,z'}$ we arrive at the expressions, valid for $e\in E_{y,z}$, $e'\in E_{y',z}$:
\begin{subequations}
  \label{eq:35}
  \begin{align}
    e'\times_1 e &= \bigl( e\times_1 e'\times_1 s_1(\braces{y'}{y}^{-1}_H,z)\bigr)\,\braces{\jmath(e)}{\jmath(e')}_G; \\
    \intertext{and for $e\in E_{y,z}$, $e'\in E_{y,z'}$:}
    e'\times_2 e &= \bigl( e\times_2 e'\times_2 s_2(y,\braces{z'}{z}^{-1}_K)\bigr)\,\braces{\jmath(e)}{\jmath(e')}_G.
  \end{align}
\end{subequations}
It follows that the (ordinary) butterfly corresponding to each variable must be braided in the sense of \cite[\S 7.4.1]{ButterfliesI}. 

Theorems~\ref{thm:1} and~\ref{thm:2} specialize to this situation. We will use  superscript $(\ )^b$ to denote the groupoids (or stacks) of braided biadditive morphisms and \bfls.
\begin{theorem}[Theorems~\ref{thm:1}, \ref{thm:2} fully commutative case]
  \label{thm:3}
  Let $\stK$, $\stH$, and $\stG$ be all braided and have presentations by braided crossed modules. The equivalence $u$ of Theorem~\ref{thm:1} restricts to an equivalence
  \begin{equation*}
    u\colon \grdbfl^b(H_\bullet,K_\bullet;G_\bullet) \lisoto
    \grdbiadd^b(\stH,\stK;\stG)
  \end{equation*}
  of groupoids. Furthermore, if $\stG$ is braided symmetric, $u$ (from Theorem~\ref{thm:2}) becomes an equivalence of group-like groupoids. Similar statements hold for $\stbfl^b(-,-;-)$ and $\stbiadd^b(-,-;-)$.
\qed
\end{theorem}
As an application of these additional assumptions, we can discuss whether $\grdbfl (H_\bullet,K_\bullet; G_\bullet)$, $\grdbiadd(\stH,\stK;\stG)$, etc.\ are biadditive in their variables. This means asking whether, for example, we have an equivalence:
\begin{equation*}
  \grdbiadd^b(\stH_1\times \stH_2,\stK;\stG) \iso
  \grdbiadd^b(\stH_1,\stK;\stG) \times
  \grdbiadd^b(\stH_2,\stK;\stG)\,,
\end{equation*}
and a similar one relative to the variable $\stK$, as well as similar statements for $\grdbfl(\stH,\stK;\stG)$, $\stbfl (H_\bullet,K_\bullet; G_\bullet)$, and $\stbiadd(\stH,\stK;\stG)$.

We have observed that regardless the commutativity assumptions (but keeping $\stG$ braided symmetric) they always are biadditive in the third variable. Biadditivity in the first and second variables only holds under the additional commutativity hypotheses. More precisely, we have
\begin{proposition}
  \label{prop:8}
  The groupoid $\grdbiadd^b (-,-;-)$ is biadditive in all variables for symmetric group-like stacks. (The same conclusion holds for $\grdbfl^b(-,-;-)$, $\stbfl^b(-,-;-)$ and $\stbiadd^b(-,-;-)$.)
\end{proposition}
\begin{proof}
  We can use the same argument as \cite[\S 1.2]{MR823233}. Specifically, consider the three (additive) functors $d_i\colon \stH\times \stH\to \stH$ ($d_1$ is additive since $\stH$ is braided).  Let $F\colon \stH\times \stK\to \stG$ be a biadditive morphism, and let $E=v(F)$ be the corresponding \bfl. Then
  \begin{equation*}
    \lambda^1\colon d_2^*F d_0^*F \lto d_1^*F
  \end{equation*}
  is a morphism in $\grdbiadd^b (\stH\times \stH,\stK;\stG)$. As in loc.\ cit.\ this follows from the compatibility between $\lambda^1$ and $\lambda^2$ and the fact that $\lambda^1$ is an additive morphism thanks to the (symmetric) braiding. In view of the equivalence $\grdbiadd^b(-,-;-)\iso \grdbfl^b(-,-;-)$, the above morphism corresponds to a morphism of biextensions
  \begin{equation*}
    \gamma^1\colon d_2^*E \wedge^{G_1} d_0^*E \lto d_1^*E
  \end{equation*}
  which is in fact a morphism of \bfls in $\grdbfl^b(H_\bullet\times H_\bullet,K_\bullet; G_\bullet)$, after one checks the trivializations. From Proposition~\ref{prop:4} we see that the braiding must be symmetric.
  Now, for additive morphisms $R,S\colon \stH'\to \stH$, we obtain a morphism
  \begin{equation*}
    (R\times I)^*F\, (S\times I)^*F \lto ((RS)\times I)^*F
  \end{equation*}
  in $\grdbiadd^b(\stH',\stK;\stG)$, showing biadditivity in the first variable. The second variable is treated analogously.
\end{proof}
\begin{remark}
  \label{rem:5}
  In the proof, if $E_R$ and $E_S$ are the butterflies corresponding to $R,S \colon \stH'\to \stH$, we obtain the morphism of \bfls in $\grdbfl (H'_\bullet,K_\bullet;G_\bullet)$ is
  \begin{equation*}
    \bigl( (E_R\times I)\times^{H_1}_{H_0} E \bigr) \wedge^{G_1} \bigl( (E_S\times I)\times^{H_1}_{H_0} E \bigr) \lto
    \bigl( (E_R\wedge^{H_1}E_S)\times I\bigr)\times^{H_1}_{H_0} E
  \end{equation*}
  expressing the biadditivity in the first variable at the level of biextensions. Here $I$ is a shorthand for the diagram corresponding to the identity morphism. (This kind of compositions is systematically studied in the next Part~\ref{part:part2}.)
\end{remark}

\part{}\label{part:part2}
\section{Multiextensions and compositions}
\label{sec:auxiliary}
The generalizations of the previous notions of biextension and biadditive morphism to the case of $n$-variables is straightforward \cites(see)()[\S 2.10.2]{MR0354656-VII}[\S 7]{MR1702420}. Let $(G,\Pi,\del,\braces--)$ be a braided crossed module.  A \strong{multiextension,} or \strong{$n$-extension,} of $(H_1,\dots H_n)$ by $(G,\Pi)$ is a $(G,\Pi)_{H_1\times \dots \times H_n}$-torsor $E$ over $H_1\times \dots \times H_n$ equipped with $n$ partial multiplication laws $\times_1,\dots,\times_n$, plus a compatibility relation of the type \eqref{eq:5} for each pair $(\times_i,\times_j)$.

Lemma~\ref{lem:2} remains valid, as well as the analogs of Propositions~\ref{prop:2} and~\ref{prop:3}.  In fact, the notion of \bfl can be extended to this case. Let us consider $\stH_1,\dots,\stH_n$ and $\stG$, the latter equipped with a braiding as usual. There is an evident notion of $n$-additive functor $F\colon \stH_1\times\dots\times \stH_n\to \stG$, which can be defined by an appropriate generalization of Definition~\ref{def:1}.  If we suppose that each $\stH_i$ has a presentation $\stH_i\iso [H_{i,1}\to H_{i,0}]\sptilde$, we can define a \bfl (or, more precisely, an $n$-\bfl) from $H_{1,\bullet}\times \dots \times H_{n,\bullet}$ to $G_\bullet$ as an $n$-extension $E$ of $(H_{1,0}\times \dots \times H_{n,0})$ by $G_\bullet$ equipped with $n$ trivializations $s_i\colon H_{1,0}\times \dots \times H_{i,1}\times \dots \times H_{n,0}\to E$, each satisfying the conditions of Definition~\ref{def:3}, with the obvious modifications. Theorem~\ref{thm:1} and its symmetric variant~\ref{thm:2} generalize to this case.  Let us state this independently for future reference. Denoting by $\grdnbfl$ the groupoid of $n$-\bfls, we have
\begin{theorem}
  \label{thm:4}
  There is an equivalence of (pointed) groupoids:
  \begin{equation*}
    u\colon \grdnbfl(H_{1,\bullet},\dots, H_{n,\bullet};G_\bullet) \lisoto
    \grdnadd(\stH_1,\dots,\stH_n;\stG),
  \end{equation*}
  sending a multi-extension $E$, to the $n$-additive functor $u(E)$ that to the object $(y_1,\dots,y_n)$ of $\stH_1\times\dots \times\stH_n$, where each $y_i$ is in the essential image of $H_{i,0}$, assigns the $(G,\Pi)$-torsor $E_{y_1,\dots,y_n}$. It is an equivalence of group-like groupoids whenever $\stG$ is symmetric.  \qed
\end{theorem}
Once again, the axioms of Definition~\ref{def:3} ensure this is compatible with the descent, ensuring $u(E)$ is a well defined morphism. Just like for biextensions, the equivalence is compatible with localization, with the stacks $\stnbfl$ and $\stnadd$ in place of the global groupoids.

Multi-additive functors can be composed in the following way.  Let $\stG$, $\stH_1,\dots,\stH_n$, $\stK_{i,1},\dots,\stK_{i,m_i}$, $i=1,\dots,n$, be group-like stacks, with $\stG$ and $\stH_1,\dots,\stH_n$ braided. Let $F\in \grdnadd(\stH_1,\dots,\stH_n;\stG)$, and for $i=1,\dots,n$ let $G_i\in \grdnadd(\stK_{i,1},\dots,\stK_{i,m_i};\stH_i)$. 
Then if $x_{1,1},\dots,x_{n,m_n}$, collectively denoted $x_1,\dots,x_m$, are objects of $\stK_{1,1},\dots,\stK_{n,m_n}$, define $F(G_1,\dots,G_n)$ as usual by
\begin{equation}
  \label{eq:36}
  F(G_1,\dots,G_n)(x_1\dots,x_m) \coloneqq F(G_1(x_1,\dotsc),\dots,G_n(\dotsc,x_m)).
\end{equation}
\begin{proposition}
  \label{prop:9}
  The composition defined in~\eqref{eq:36} assigns to the tuple $(F,G_1,\dots,G_n)$ a well defined object $F(G_1,\dots,G_n)$ of $\grdnadd(\stK_{1,1},\dots,\stK_{n,m_n};\stG)$.  This composition is associative. 

  An identical statement holds with $\grdnadd$ replaced by $\stnadd$.
\end{proposition}
\begin{proof}
  The only thing to check is that the composition $F(G_1,\dots,G_n)$ satisfy the conditions~\eqref{eq:7} for each pair $(i,j)$ of indices within the list $\lbrace 1,\dots, m_1+m_2+\dots +m_n\rbrace$, where $m_k$ is the arity of $G_k$.  There are two cases depending on whether the variables corresponding to the pair $(i,j)$ belongs to the same ``slot,'' say relative to $G_k$, or when $i$ and $j$ fall into two different slots, relative to $G_k$ and $G_l$, with $k\neq l$.  We now indicate the main points of the verification, leaving the easy task of writing the complete diagrams to the reader. 

  In the former case, the mechanics of the verification are completely captured by considering the values $n=1$, $m_1=2$. First we write the pentagonal diagram corresponding to~\eqref{eq:7} for $F(G)$. In it, we use the functoriality of $F$ to reduce the arrow
  \begin{equation*}
    \xymatrix@1{
      (F(G(x,y))\, F(G(x',y)))\; (F(G(x,y'))\, F(G(x',y'))) \ar[r] &
      (F(G(x,y))\, F(G(x,y')))\; (F(G(x',y))\, F(G(x',y'))) }
  \end{equation*}
  to the arrow
  \begin{equation*}
    \xymatrix@1{
      F((G(x,y)\, G(x',y))\, (G(x,y')\, G(x',y'))) \ar[r] &
      F((G(x,y)\, G(x,y'))\, (G(x',y)\, G(x',y'))) },
  \end{equation*}
  and then use the fact that $G$ itself satisfies~\eqref{eq:7}.  For the latter case, the general situation is captured by considering $n=2$, $m_1=m_2=1$, so we need to write the diagram~\eqref{eq:7} for $F(G_1,G_2)$. For this, the interchange law for $F$ gives us a unique morphism from
  \begin{equation*}
      (F(G_1(x),G_2(y))\, F(G_1(x'),G_2(y)))\;
      (F(G_1(x),G_2(y'))\, F(G_1(x'),G_2(y'))) 
    \end{equation*}
    to
    \begin{equation*}
      F(G_1(x)\,G_1(x'),G_2(y)\,G_2(y')).
    \end{equation*}
    Now the functoriality of $F$ gives the morphism
  \begin{equation*}
    F(G_1(x)\,G_1(x'),G_2(y)\,G_2(y')) \lto F(G_1(xx'),G_2(yy'))
  \end{equation*}
  by way of the square
  \begin{equation*}
    \xymatrix{%
      F(G_1(x)\,G_1(x'),G_2(y)\,G_2(y')) \ar[r] \ar[d]&
      F(G_1(x)\,G_1(x'),G_2(yy')) \ar[d] \\
      F(G_1(xx'),G_2(y)\,G_2(y')) \ar[r] &
      F(G_1(xx'),G_2(yy'))
    }
  \end{equation*}
  The statement about associativity is immediate.
\end{proof}

\section{Compositions of $n$-\bfls}
\label{sec:comp-n-bfls}

From Theorem~\ref{thm:4}, the multilinear functors can be expressed in terms of multi-extensions, and from Proposition~\ref{prop:9} multilinear functors can be composed.  Hence we expect an analogous composition exists for multi-extensions. In this section we provide a construction of this composition, generalizing the composition for single \bfls found in \cite[\S 5.1]{ButterfliesI}.

For $i=1,\dots,n$, and integers $j_i=1,\dots, m_i$, let $G_\bullet$, $H_{1,\bullet},\dots H_{n,\bullet}$ and $K_{i,1,\bullet},\dots,K_{i,m_i,\bullet}$ crossed modules, of which $G_\bullet$, $H_{1,\bullet},\dots H_{n,\bullet}$ are assumed to be braided.  Let $E$ be a braided $n$-\bfl from $H_{1,\bullet}\times\dots\times H_{n,\bullet}$ to $G_\bullet$, and for $i=1,\dots,n$ let $F_i$ be an $m_i$-\bfl from $K_{i,1,\bullet}\times\dots\times K_{i,m_i,\bullet}$ to $H_{i,\bullet}$.
In the following we let 
\begin{equation*}
  P = 
  (F_1\times\dots\times F_n) \varprod_{H_{1,0}\times\dots \times H_{n,0}} E,
\end{equation*}
be the $(H_{1,1}\times\dots\times H_{n,1})$-equivariant $(G_1,G_0)$-torsor \emph{over} $F_1\times\dots \times F_n$.
\begin{theorem}
  \label{thm:5}
  The  $(G_1,G_0)$-torsor $P$ has a well defined quotient by $H_{1,1}\times\dots\times H_{n,1}$, denoted
  \begin{equation*}
    (F_1\times\dots\times F_n) \varprod^{H_{1,1}\times\dots\times H_{n,1}}_{H_{1,0}\times\dots
      \times H_{n,0}} E,
  \end{equation*}
  which carries a structure of $m_1+\dots+m_n$-extension (=\bfl) from $K_{1,1,\bullet}\times \dots \times K_{n,m_n,\bullet}$ to $G_\bullet$.
\end{theorem}
\begin{definition}
  \label{def:5}
  The composition $E(F_1,\dots,F_n)$ of $E$ with $F_1,\dots,F_n$ is
  \begin{equation*}
    E(F_1,\dots,F_n) \coloneqq
    (F_1\times\dots\times F_n) \varprod^{H_{1,1}\times\dots\times H_{n,1}}_{H_{1,0}\times\dots \times H_{n,0}} E.
  \end{equation*}
\end{definition}
We call this composition the \emph{juxtaposition product,} since it entails placing the \bfls wing-by-wing next to one another.
\begin{proof}[\textbf{Proof of Theorem~\ref{thm:5}–Construction of the juxtaposition product}]
  We will use element notation throughout. The procedure consists of several steps. The first is to form the fiber product exactly as in loc.\ cit.
\begin{equation*}
  P = 
  (F_1\times\dots\times F_n) \varprod_{H_{1,0}\times\dots \times H_{n,0}} E,
\end{equation*}
which as noted is an $(H_{1,1}\times\dots\times H_{n,1})$-equivariant $(G_1,G_0)$-torsor \emph{over} $F_1\times\dots \times F_n$, and to mod out the left and right actions of $H_{1,1}\times\dots\times H_{n,1}$. Let us use the notation
\begin{equation*}
  (y_1,\dots,y_n) \coloneqq
  (\jmath_1(v_1),\dots \jmath_n(v_n)) = p(u)
\end{equation*}
for the tuple in $H_{1,0}\times\dots \times H_{n,0}$,  where $\jmath_i\colon F_i\to H_{i,0}$ for each $(H_{i,1},H_{i,0})$-torsor $F_i$, $i=1,\dots,n$.

Consider first the right action of elements of the form $(1,\dots,h_i,\dots, 1)$, where $h_i\in H_{i,1}$. The action on $P$ is given by sending a point $(v_1,\dots,v_n,u)$ to
\begin{equation*}
  (v_1,\dots,v_i\,h_i, \dots v_n, u\times_i s_i(y_1,\dots ,h_i,\dots, y_n)),
\end{equation*}
For the left action, we have
\begin{equation*}
  (v_1,\dots,h_i\,v_i, \dots v_n, s_i(y_1,\dots ,h_i,\dots, y_n) \times_i u) = 
  (v_1,\dots,v_i\,h_i^{y_i}, \dots v_n, u\times_i s_i(y_1,\dots ,h_i^{y_i},\dots, y_n)),
\end{equation*}
which follows from Definition~\ref{def:3}, adapted to the multi-variable case.  These are actions thanks to the multiplicativity properties of $s_j$ established in Lemma~\ref{lem:4}.
These actions are compatible for $i\neq j$.  Acting with $h_i$ and then with $h_j$, assuming for example that $i < j$, we have
\begin{multline*}
  \bigl(u \times_i s_i(y_1,\dots,h_i,\dots,y_n)\bigr)\times_j
  s_j(y_1,\dots,y_i\del h_i,\dots, h_j,\dots, y_n) \\
  = \bigl(u \times_i s_i(y_1,\dots,h_i,\dots,y_n)\bigr)\times_j
  \bigl(s_j(y_1,\dots,y_i\dots, h_j,\dots, y_n) \times_i
  s_j(y_1,\dots,\del h_i,\dots, h_j,\dots, y_n)\bigr),
\end{multline*}
where we used the multiplicativity property of $s_j$. 
Now, using the interchange property for $\times_i$ and $\times_j$, and the compatibility between the trivializations $s_i$ and $s_j$, the right hand side becomes
\begin{equation*}
  \bigl(u \times_j s_j(y_1,\dots,h_j,\dots,y_n)\bigr)\times_i
  \bigl(s_i(y_1,\dots,h_i,\dots, y_n) \times_j
  s_i(y_1,\dots, h_i,\dots, \del h_j,\dots, y_n)\bigr),
\end{equation*}
which coincides with the action of $h_j$ first, followed by that of $h_i$.  This ensures that following formula for the right action of a generic point $(h_1,\dots,h_n)$ is well defined:
\begin{equation}
  \label{eq:37}
  \bigl( v_1\,h_1,\dots v_n\,h_n,
  \bigl( \dotsm \bigl( u \times_1 s_1(h_1,y_2,\dots,y_n) \bigr) \times_2
  s_2(y_1\del h_1,h_2,\dots,y_n) \dotsm \bigr) \times_n
  s_n(y_1\del h_1,\dots,y_{n-1}\del h_{n-1}, h_n)\bigr).
\end{equation}
The reader will have no difficulty in writing the corresponding formulas for the left action. Since each $F_i$ is an $(H_{i,1},H_{i,0})$-torsor, the action is free.

Note also that the $G_1$-actions on $P$ (both left and right) are compatible with the action of $H_{1,1}\times \dots \times H_{n,1}$ thanks to the relations ~\eqref{eq:25} in Lemma~\ref{lem:4} (property \ref{item:7}).  In particular, the $G_1$-actions happen by way of those on the last element of the tuple; denoting the class of a tuple by brackets we have:
\begin{equation*}
  [v_1,\dots,v_n,u]\,g = [v_1,\dots,v_n,u\,g], \qquad
  g\, [v_1,\dots,v_n,u] = [v_1,\dots,v_n,g\, u]
  = [v_1,\dots,v_n,u\,g^{\jmath{u}}].
\end{equation*}
Finally, the equivariant section from the quotient of $P$ to $G_0$ is defined to be
\begin{equation*}
  (v_1,\dots, v_n,u) \longmapsto \jmath (u).
\end{equation*}

Next, we define $m_1+\dots +m_n$ partial product structures $\times_{i,j}$ for $i=1,\dots,n$ and $j=1,\dots,m_i$ as follows.  Using our established index convention, let us use the notation
\begin{equation*}
  (z_{1,1},\dots, z_{i,j}, \dots, z_{n,m_n}) \in
  K_{1,1,0}\times \dots \times K_{i,j,0} \times\dots \times K_{n,m_n,0}
\end{equation*}
for a point of the base of $P/(H_{1,1}\times\dots\times H_{n,1})$.  To begin with, consider the special case of two points $(v_1,\dots,v_j,\dots,v_n,u)$ and $(v_1,\dots, v'_j,\dots, v_n,u')$ of $P$ such that
\begin{equation*}
  p_i(v_i) = (z_{i,1},\dots, z_{i,j}, \dots, z_{i,m_i}), \quad
  p_i(v'_i) = (z_{i,1},\dots, z'_{i,j}, \dots, z_{i,m_i}),
\end{equation*}
in $K_{i,j,0}$, for fixed $i$ and $j\in \lbrace 1,\dots,m_i\rbrace$, everything else being equal.  Define:
\begin{equation}
  \label{eq:38}
  (v_1,\dots,v_j,\dots,v_n,u) \times_{i,j} (v_1,\dots,v'_j,\dots,v_n,u) \coloneqq
  (v_1,\dots,v_j\times_{i,j} v'_j,\dots,v_n,u\times_i u').
\end{equation}
The symbol $\times_{i,j}$ on the right hand side of~\eqref{eq:38} denotes the $j$\textsuperscript{th} product structure (within $j=1,\dots,m_i$) of $F_i$, and the resulting point of $P$ projects onto the point $(z_{1,1},\dots, z_{i,j}z'_{i,j}, \dots, z_{n,m_n})$.

In general, let us consider points $e=[v_1,\dots,v_n,u]$ and $e'=[v'_1,\dots,v'_n,u']$ of $P/(H_{1,1}\times\dots\times H_{n,1})$ above $(z_{1,1},\dots, z_{i,j}, \dots, z_{n,m_n})$ and $(z_{1,1},\dots, z'_{i,j}, \dots, z_{n,m_n})$, respectively.  (As before, only the $z_{i,j}$ and $z'_{i,j}$ coordinates are different.) For $j=1,\dots,n$, $j\neq i$, there exist unique $h_j\in H_{j,1}$ such that $v'_j = v_j\, h_j$.  As a result,
\begin{equation*}
  [v'_1,\dots,v'_n,u'] = [v_1,\dots,v'_i,\dots,v_n,u''],
\end{equation*}
where $u''$ is related to $u'$ by an application of~\eqref{eq:37}.  Then we define $e \times_{i,j} e'$ as the class
\begin{equation*}
  [v_1,\dots,v_j\times_{i,j} v'_j,\dots,v_n,u\times_i u''],
\end{equation*}
which is computed using~\eqref{eq:38} above. We must show that this is independent of the various choices involved through the use of~\eqref{eq:37}.  The computation is quite elaborate, but otherwise not illuminating nor eventful, therefore we omit it. We also omit the easy verification that each of these partial multiplication laws is associative.

To claim that we have constructed a genuine multi-extension, we must prove the partial multiplication laws just defined obey pairwise compatibility (interchange) laws. This we prove explicitly. Like in the proof of Proposition~\ref{prop:9}, there are two distinct cases to address, depending on whether the two partial product laws have the same first index. The easiest is when they do not, so we treat it first. Thus, let $i\neq k\in \lbrace 1,\dots n\rbrace$. For brevity let $Q = P/(H_{1,1}\times\dots\times H_{n,1})$ and consider:
\begin{gather*}
  e_{i,j} = [v_1,\dots,v_i,\dots,v_k,\dots,v_n,u_i] \in
  Q_{z_{1,1},\dots,z_{i,j},\dots,z_{k,l},\dots,z_{n,m_n}} \\
  e'_{i,j} = [v_1,\dots,v'_i,\dots,v_k,\dots,v_n,u'_i] \in
  Q_{z_{1,1},\dots,z'_{i,j},\dots,z_{k,l},\dots,z_{n,m_n}} \\
  e_{k,l} = [v_1,\dots,v_i,\dots,v'_k,\dots,v_n,u_k] \in
  Q_{z_{1,1},\dots,z_{i,j},\dots,z'_{k,l},\dots,z_{n,m_n}} \\
  e'_{k,l} = [v_1,\dots,v'_i,\dots,v'_k,\dots,v_n,u'_k] \in
  Q_{z_{1,1},\dots,z'_{i,j},\dots,z'_{k,l},\dots,z_{n,m_n}}.
\end{gather*}
We first compute the products with $\times_{i,j}$:
\begin{equation*}
  e_{i,j} \times_{i,j} e'_{i,j} =
  [v_1,\dots,v_i\times_{i,j} v'_i,\dots,v_k,\dots,v_n,u_i \times_i u'_i]
  \qquad
  e_{k,l} \times_{i,j} e'_{k,l} =
  [v_1,\dots,v_i\times_{i,j} v'_i,\dots,v'_k,\dots,v_n,u_k \times_i u'_k].
\end{equation*}
Then
\begin{align*}
  (e_{i,j} \times_{i,j} e'_{i,j}) \times_{k,l}
  (e_{k,l} \times_{i,j} e'_{k,l}) &=
  [v_1,\dots,v_i\times_{i,j} v'_i,\dots,v_k \times_{k,l}v'_k,\dots,v_n,
  (u_i \times_i u'_i)\times_{k}(u_k \times_i u'_k)] \\
  &= [v_1,\dots,v_i\times_{i,j} v'_i,\dots,v_k \times_{k,l}v'_k,\dots,v_n,
  (u_i \times_k u_k)\times_{i}(u'_i \times_k u'_k) \,
  \braces{\jmath(u'_i)}{\jmath(u_k)}^{-\jmath(u'_k)} ] \\
  &= (e_{i,j} \times_{k,l} e_{k,l}) \times_{i,j}
  (e'_{i,j} \times_{k,l} e'_{k,l})\, \braces{\jmath(u'_i)}{\jmath(u_k)}^{-\jmath(u'_k)},
\end{align*}
where in the second line we have used the interchange law for the $i$\textsuperscript{th} and $k$\textsuperscript{th} product laws in $E$, Proposition~\ref{prop:1}.

The other case is the one with the same first index, say $\times_{i,j}$ and $\times_{i,l}$, for $j\neq l$. In this situation we consider instead:
\begin{align*}
  e_{i,j} &= [v_1,\dots,v_i,\dots,v_n,u_i] \in
  Q_{z_{1,1},\dots,z_{i,j},\dots,z_{i,l},\dots,z_{n,m_n}} \\
  e'_{i,j} &= [v_1,\dots,v'_i,\dots,v_n,u'_i] \in
  Q_{z_{1,1},\dots,z'_{i,j},\dots,z_{i,l},\dots,z_{n,m_n}} \\
  e_{i,l} &= [v_1,\dots,w_i,\dots,v_n,t_i] \in
  Q_{z_{1,1},\dots,z_{i,j},\dots,z'_{i,l},\dots,z_{n,m_n}} \\
  e'_{i,l} &= [v_1,\dots,w'_i,\dots,v_n,t'_i] \in
  Q_{z_{1,1},\dots,z'_{i,j},\dots,z'_{i,l},\dots,z_{n,m_n}}.
\end{align*}
Proceeding in an analogous way we first get:
\begin{equation*}
  e_{i,j} \times_{i,j} e'_{i,j} =
  [v_1,\dots,v_i\times_{i,j} v'_i,\dots,v_k,\dots,v_n,u_i \times_i u'_i],
  \qquad
  e_{k,l} \times_{i,j} e'_{k,l} =
  [v_1,\dots,w_i\times_{i,j} w'_i,\dots,v_n,t_k \times_i t'_k].
\end{equation*}
And then:
\begin{align*}
  (e_{i,j} \times_{i,j} e'_{i,j}) \times_{i,l}
  (e_{i,l} \times_{i,j} e'_{i,l}) &=
  [v_1,\dots,(v_i\times_{i,j} v'_i) \times_{i,l} (w_i\times_{i,j}w'_i),\dots,v_n,
  (u_i \times_i u'_i)\times_{i}(t_i \times_i t'_i)] \\
  &= [v_1,\dots, (v_i\times_{i,j} v'_i) \times_{i,l} (w_i\times_{i,j}w'_i) \,
  \braces{\jmath_i(w_i)}{\jmath_i(v'_i)}_{H_i}^{-\jmath_i(w'_i)},\dots,v_n, \\
  &  (u_i \times_i u'_i)\times_{i}(t_i \times_i t'_i)\times_i
  s_i(y_1,\dots,\braces{\jmath_i(w_i)}{\jmath_i(v'_i)}_{H_i}^{-\jmath_i(w'_i)},\dots,y_n) ] \\
  &= [v_1,\dots,(v_i\times_{i,l} w_i) \times_{i,j} (v'_i\times_{i,l}w'_i),\dots,v_n,
  (u_i \times_i t_i)\times_{i}(u'_i \times_i t'_i) \,
  \braces{\jmath(u'_i)}{\jmath(t_i)}_G^{-\jmath(t'_i)} ],
\end{align*}
where in the intermediate equality we used the the action by an element of $H_{i,1}$, whereas in the last one the interchange law (within the multi-extension $F_i$) and~\eqref{eq:35} were used. But the last line equals
\begin{equation*}
  (e_{i,j} \times_{i,l} e_{i,l}) \times_{i,j}
  (e'_{i,j} \times_{i,l} e'_{i,l}) \, \braces{\jmath(u'_i)}{\jmath(t_i)}_G^{-\jmath(t'_i)}, 
\end{equation*}
as wanted.

The final step is to show that $Q$, in addition to being a multi-extension of $K_{1,1,0}\times \dots \times K_{n,m_n,0}$ by $G_\bullet$, is in fact a \bfl, that is, it carries a trivialization for each of its pullbacks to $K_{1,1,0}\times \dots \times K_{i,j,1}\times \dots \times K_{n,m_n,0}$.  Such trivializations can be defined as follows.

For each $i=1,\dots,n$ and for each $j=1,\dots,m_i$, let $s_{i,j}$ be the trivialization  $s_{i,j}\colon K_{i,1,0}\times\dots \times K_{i,j,1}\times\dots \times K_{i,m_i,0}\to F_i$. Observe that the fiber of $P$ above $(v_1,\dots, s_{i,j}(z_{i,1},\dots,k_{i,j},\dots,z_{i,m_i}),\dots, v_n) \in F_1\times\dots \times F_i\times \dots \times F_n$ (where $v_k$ is a point of $F_k$, for $k\neq i$) can be identified with $G_1$.  This is because that point maps via $(\jmath_1,\dots,\jmath_n)$ to $(y_1,\dots, 1,\dots, y_n)$, and the fiber of $E$ above it is identified with $G_1$. (In turn, this follows from the fact that biextensions in particular, and multi-extensions in general, are canonically trivialized over identity sections, see sect.~\ref{sec:biext-braid}.) It follows that the fiber of $Q$ over $(z_{1,1},\dots,\del k_{i,j},\dots, z_{n,m_n})$ is similarly identified with $G_1$.  Thus, we obtain a trivializing section $\hat s_{i,j}$ of the pullback of $Q$ to $K_{1,1,0}\times \dots \times K_{i,j,1}\times \dots \times K_{n,m_n,0}$ by sending the point $(z_{1,1},\dots, k_{i,j},\dots, z_{n,m_n})$ to $[v_1,\dots, s_{i,j}(z_{i,1},\dots,k_{i,j},\dots,z_{i,m_i}),\dots, v_n,e_i]$, where $e_i$ is the unit (central) section of the restriction of $E$ to $H_{1,0}\times \dots \{1\}\times \dots\times H_{n,0}$. Observe this does not depend on the choice of $v_j\in F_j$, for $j\neq i$, since $F_j/H_j$ is just $(z_{j,1},\dots, z_{j,m_j})$.  Note that $\jmath\colon Q\to G_0$ composed with any of $\hat s_{i,j}$ is identically equal to $1$, since so is the result of applying $\jmath_E\colon E\to G_0$ to $e_i$. The conditions (Compatibility and Restriction) of Definition~\ref{def:3} are easy to verify and so this task is left to the reader.

This ends the proof.
\end{proof}
\begin{proposition}
  \label{prop:10}
  The composition in Definition~\ref{def:5} gives rise to a well-defined composition functor
  \begin{equation*}
    \grdnbfl(K_{1,1,\bullet},\dots;H_{1,\bullet}) \times\dots \times
    \grdnbfl(\dots, K_{n,m_n,\bullet};H_{n,\bullet}) \times
    \grdnbfl(H_{1,\bullet},\dots,H_{n,\bullet};G_\bullet) \lto
    \grdnbfl(K_{1,1,\bullet},\dots,K_{n,m_n,\bullet};G_\bullet),
  \end{equation*}
  where we are using the same convention for the indices as above and in sect.~\ref{sec:formalism}.  Same with $\stnbfl$ in place of $\grdnbfl$.
\end{proposition}
\begin{proof}
  We need only check that if $f\colon E\to E'$ and $g_i\colon F_i\to F'_i$ are \bfl isomorphisms, then we get a well defined isomorphism of $k=m_1+\dots+m_n$-\bfls
  \begin{equation*}
    f(g_1,\dots,g_n) \colon E(F_1,\dots,F_n)\lto E'(F'_1,\dots,F'_n).
  \end{equation*}
  This is easy to check from the construction of the juxtaposition product provided above.
\end{proof}
Associativity only holds up to isomorphism. This is due to the same phenomenon, ultimately a manifestation of the same lack of strict associativity for fiber products, observed in the unary case \cite[see][]{ButterfliesI,Noohi:weakmaps}. We omit the proof.
\begin{proposition}
  \label{prop:11}
  The composition in Definition~\ref{def:5} is associative up to coherent isomorphism. \qed
\end{proposition}

\section{The bi-multicategories of group-like stacks and crossed modules}
\label{sec:bi-mult-braid}

By using Proposition~\ref{prop:9} we can rely on the framework recalled in the technical section~\ref{sec:formalism} to promote the 2-category of braided group-like stacks to a 2-multicategory. In particular, using Definition~\ref{def:10}, we define the categories $\lbrace\stH,\stG\rbrace$, where to each $n$ we assign
\begin{equation*}
  \grdnadd (\stH^n;\stG) = \grdnadd
  (\underbrace{\stH,\dots,\stH}_{n\mathrm{-times}};\stG).
\end{equation*}
We ambiguously use the same symbol to denote the stack obtained using $\stnadd$ in place of $\grdnadd$, as which version will be in use will be clear from the context.  From Definition \ref{def:14} we immediately obtain:
\begin{proposition}
  \label{prop:12}
  For group-like stacks $\stK$, $\stH$, and $\stG$, we have the assembly map
  \begin{equation}
    \label{eq:39}
    \medcirc\colon \braces{\stH}{\stG}\wr \braces{\stK}{\stH} \lto
    \braces{\stK}{\stG}
  \end{equation}
  which to the object $(F;G_1,\dots,G_n)$ assigns the composition $F(G_1,\dots,G_n)$.

  In particular, $M\stG\coloneqq \braces{\stG}{\stG}$ is a monoid for this composition, whenever $\stG$ is braided, with identity object given by $(\id;\ )$ (the empty string in the second slot). \qed
\end{proposition}
\begin{remark}
  \label{rem:6}
  As alluded in sect.~\ref{sec:formalism}, the formalism includes the symmetry structure, which makes it slightly more general than needed in the sequel. All objects of interest shall have augmentations factoring through the discrete subcategory $\NN$ of $\SS$, which amounts to ignoring the permutation structure and hence symmetry conditions.
\end{remark}

From Propositions~\ref{prop:10} and~\ref{prop:11} and the formalism of sect.~\ref{sec:formalism} the bicategory of braided crossed modules of $\T$ equipped with \bfls as morphisms \cite[see][Theorem 5.3.6]{ButterfliesI} is promoted to a genuine bi-multicategory.
In particular, we obtain objects $\lbrace H_\bullet,G_\bullet\rbrace$ and an analog of Proposition~\ref{prop:12} with assembly map 
\begin{equation}
  \label{eq:40}
  \medcirc \colon \lbrace H_\bullet,G_\bullet\rbrace \wr \lbrace
  K_\bullet,H_\bullet\rbrace \lto \lbrace K_\bullet,G_\bullet\rbrace
\end{equation}
for braided crossed modules $K_\bullet$, $H_\bullet$, and $G_\bullet$.

The equivalence between the 2-category $\grst$ of (braided) group-like stacks and the bicategory $\xmod$ of (braided) crossed modules~\cite[cf.][Theorem 5.3.6]{ButterfliesI} lifts to the present case. Let $\M\grst$ (resp.\ $\M\xmod$) the bi-multicategory with morphisms $\grdHom(\stH_1,\dots,\stH_n;\stG)$ (resp.\
$\grdnbfl(K_{1,1,\bullet},\dots,K_{n,m_n,\bullet};G_\bullet)$).
\begin{theorem}
  \label{thm:6}
  There is an equivalence of bi-multicategories
  \begin{equation*}
    \M a\colon 
    \M\xmod \lisoto \M\grst
  \end{equation*}
  induced by the associated stack functor.
\end{theorem}
\begin{proof}
  Recall that a functor $F\colon M\bC\to \M\bD$ between bi-multicategories is an equivalence if for each tuple $(x;y_1,\dots,y_n)$ of objects of $\M\bC$ the functor $F_{(x;y_1,\dots,y_n)}\colon \grdHom_\bC (y_1,\dots,y_n;x) \to \grdHom_\bD(F(y_1), \dots, F(y_n); F(x))$ is an equivalence of categories, and the underlying ordinary homomorphism $[F]_1\colon \bC\to \bD$ is essentially surjective.\footnote{$[F]_1$ is the homomorphism we obtain by restricting $F$ to arrows of arity equal to one.}

  The associate stack functor at the level of the underlying bicategories $a\colon \xmod\to \grst$ is an equivalence, hence, in particular, essentially surjective. The equivalence is proved in \cite[Theorem 5.3.6]{ButterfliesI} without the braiding hypothesis, however Theorem 4.3.1 in loc.\ cit. applies to the braided case as well by \S 7, ibid., and presentations by braided crossed modules are easily obtained by using Lemma~\ref{lem:1}.  Thus essential surjectivity follows.

  The equivalence at the level of the multi-arrow categories directly follows from Theorem~\ref{thm:4}.  Finally we check that the functor $u$ of Theorem~\ref{thm:4} preserves the (horizontal) compositions up to coherent isomorphism. Like the single variable case, this follows from the fiber product construction of the quasi-inverse to $u$.
\end{proof}
\begin{remark}
  \label{rem:7}
  The situation becomes decidedly more pleasant if we confine ourselves to symmetric objects. We can consider a symmetric variant of the above based on the bicategories $\sxmod$ and $\sgrst$. The equivalence in Theorem~\ref{thm:6} restricts to between these new entities:
  \begin{equation*}
    \M a\colon \M\sxmod \lisoto \M\sgrst,
  \end{equation*}
  which, by Theorem~\ref{thm:4}, both ought to be regarded as multicategories enriched over symmetric group-like groupoids. Furthermore, using the version of $\braces{-}{-}$  based on $\stnadd$, $\sgrst$ becomes closed. Indeed one has the analog of \cite[Theorem 2]{MR0340371}, in that
  \begin{equation*}
   \stHom_\bC (\stK\wr \stH,\stG) \iso \stHom_\bC(\stK,\braces{\stH}{\stG}),
  \end{equation*}
  where $\bC=\sgrst$ (or even $\bC=\sgrst/\SS$ using extraordinary structure briefly discussed in sect.~\ref{sec:formalism}—see \cite{MR0340371} for all the details). A proof can be obtained along the same lines, proceeding from Theorem~\ref{thm:4} and sect.~\ref{sec:formalism}. We will not pursue this further.
\end{remark}

\part{}\label{part:part3}
\section{Bimonoidal structures and weakly categorical rings }
\label{sec:bimon-struct}

For a braided stack $\stR$, a biadditive bifunctor $m\colon \stR\times \stR\to \stR$ gives a binary law which is automatically distributive with respect to the  monoidal structure of $\stR$ by virtue of \eqref{eq:7} in Definition~\ref{def:1}. If $m$ is itself part of a monoidal structure, that is, there exists $\mu\colon m(m,\id)\Rightarrow m(\id,m)$ satisfying the standard pentagon identity, then $\stR$ is said to be bimonoidal. In fact it satisfies the axioms of a categorical ring in the usual sense \cite{MR2369166}, except possibly the requirement for the underlying group-like structure of $\stR$ to be symmetric. (One axiom was apparently missing in \cite{MR2369166}, and the gap was filled in \cite{MR3085798}, thus completing the classification.)

From sections~\ref{sec:auxiliary} and~\ref{sec:formalism}, the binary operation $m$ and $\id$ are objects of $\M\stR=\lbrace \stR,\stR\rbrace$, and $\mu\colon m(m,\id) \To m(\id,m)$ is a morphism of $\M\stR$, with the pentagon identity expressing the equality between two morphisms thereof. Rather than using the classical, ``biased,'' definition of the monoidal  structure, it is more convenient to exploit the multi-categorical structure comprising all the multilinear functors. Now, biased and unbiased definitions yield equivalent structures \cite{MR2094071}, hence we can simply define it in terms of the multicategory $\M\grst$.  
\begin{definition}[See Definition~\ref{def:11}]
  \label{def:6}
  Let $\stR$ be a braided stack. A \emph{(weak) categorical ring structure} on $\stR$ is a lax monoidal functor $(t,\theta) \colon \NN\to \M\stR$. 
  A \emph{(weak) ring-like stack} is a braided stack $\stR$ equipped with such a structure, namely a pseudo-monoid in the 2-multicategory $\M\grst$.\footnote{In the sequel we shall use ``weak'', when referring to bimonoidal structures, in this sense; it will not refer to the laxness of the monoidal structures.}
\end{definition}
In plainer language, $\NN$ represents all the $n$-ary operations in a ``disembodied'' form. Thus, an object $n$ of $\NN$ corresponds to an operation with $n$ abstract variables. The functor $t$ assigns to $n$ an actual $t(n)\colon \stR\times \dots \times \stR\to \stR$ in $\grdnbfl(\stR^n;\stR)$. Given $n$ operations of various arities, they can be substituted into one of arity $n$, as hinted above. This set of operations is represented by the object $(n;m_1,\dots,m_n)$ of $\NN\wr\NN$ (see Definitions \ref{def:14} and~\ref{def:11}) and the substitution, resulting in an operation of arity $m_1+\dots +m_n$, gives a morphism $\NN\wr\NN\to \NN$ which has the properties of a monoidal structure. This is mapped by $t$ into the composition of actual multi-additive multi-functors of $\stR$, satisfying the same combinatorial structure. In summary, the functor $(t,\theta)$ gives a collection of multi-additive $t(n)$ subject to isomorphisms of functors of the form $\theta_{n;m_1,\dots,m_n}\colon t(m_1+\dots +m_n) \to t(n)(t(m_1),\dots, t(m_n))$, controlling the compositions.
\begin{remark}
  \label{rem:8}
  The 2-multicategory $\M\grst$ has of course more structure than the abstract situation discussed in sections~\ref{sec:formalism} and~\ref{sec:formalism-monoid}. In particular, the objects of $\grst$, being themselves categories, have an internal structure, which we can use to define the wreath product $\NN \wr \stG$ for any $\stG$ following sect.~\ref{sec:formalism}, or, literally, the arguments of~\cite[\S 2.1]{MR0340371}. Briefly, an object of $\NN\wr\stG$ consists of a string $(n;x_1,\dots, x_n)$ (where the label $n$ is in effect redundant) with $x_i$ an object of $\stG$, $i=1,\dots,n$. There are no morphisms between objects with different integer labels, and a morphism $(n;x_1,\dots, x_n)\to (n;x'_1,\dots,x'_n)$ is simply an $n$-tuple of morphisms in $\stG$.
  
  It follows that by a special case of the adjunction discussed in~\cite{MR0340371}, which in the present case can be verified by hand, the lax monoidal functor providing the exterior monoidal structure on $\stR$ corresponds to
  \begin{equation*}
    (\tilde t,\tilde\theta)\colon \NN\wr \stR \lto \stR,
  \end{equation*}
  which is required to satisfy the diagrams displayed in Remark~\ref{rem:19}, which define the structure of a \emph{pseudo algebra} over $\NN$.
\end{remark}

We drop the ``weak'' attribute is the underlying braiding is symmetric. In such a case, we will see the fibers of $\stR$ are categorical rings in the usual sense, so that $\stR$ becomes the stack analog of the categorical rings described in e.g. \cite{MR2369166}. 
\begin{remark}
  \label{rem:9}
  For a symmetric $\stR$, we can define the ring-like structure in the same way, but working in $\sgrst$. In such case the pseudo-algebra structure of $\stR$ can be verified directly from the adjunction mentioned in Remark~\ref{rem:7} above.
\end{remark}
A pseudo monoid $\stR$ in $\M\grst$ ought to correspond to one in $\M\xmod$.  More precisely, any pseudo monoid $\stR$ ought to be equivalent to one whose underlying stack is the one associated to a presentation by a braided crossed module, i.e.\ an object of $\xmod$. This will be made precise below. First, we can write the analog of Definition~\ref{def:6}, namely:
\begin{definition}
  \label{def:7}
  A \emph{(weakly) ring-like crossed module} is a pseudo monoid in the bi-multicategory $\M\xmod$, that is, a  braided crossed module $R_\bullet : R_1 \overset{\del}{\to} R_0$ equipped with a pseudo monoidal functor
  \begin{equation*}
    (t,\theta) \colon \NN \lto \M R_\bullet.
  \end{equation*}
  We drop the ``weak'' attribute if the underlying braiding is symmetric.
\end{definition}
An operation of arity $n$ is realized in this case by an $n$-\bfl $E_n=t(n)$, an object of $\grdnadd (R_\bullet^n;R_\bullet)$, composed according to~\eqref{eq:70}, realized by the juxtaposition product
\begin{equation}
  \label{eq:41}
  \theta_{n;m_1,\dots,m_n}\colon E_{m_1+\dots +m_k}\lto E_n( E_{m_1},\dots, E_{m_n}),
\end{equation}
plus the coherence condition expressed by the full diagram~\eqref{eq:72}, which reads
\begin{equation}
  \label{eq:42}
  \vcenter{\xymatrix{%
      E_k (E_{l_1},\dots,E_{l_k}) \ar[d] & E_h \ar[r] \ar[l] &
      E_n (E_{j_1},\dots,E_{j_n}) \ar[d] \\
      E_n (E_{m_1},\dots,E_{m_n}) (E_{l_1},\dots,E_{l_k}) \ar[rr] &&
      E_n (E_{m_1}(E_{l_1},\dotsc),\dots,E_{m_n}(\dots,E_{l_k}))
    }}
\end{equation}
The bottom arrow is the association isomorphism in the \bfl juxtaposition.  For $n\leq 4$, with all the sequences emanating from $n=4$ as in sect.~\ref{sec:formalism-monoid}, we obtain the ``generalized pentagon'' diagram~\eqref{eq:74}.

Theorem~\ref{thm:6} and the fact that pseudo monoids are preserved by pseudo functors (Definition~\ref{def:13}) guarantee that a weakly ring-like $R_\bullet$ gives rise, through the associated stack construction, to a weakly ring-like stack. More interesting is the converse direction, namely:
\begin{proposition}
  \label{prop:13}
  If $(\stR,t,\theta)$ is a weakly categorical ring and $R_1\to R_0 \to \stR$ a presentation, then $R_\bullet : R_1\to R_0$ is a pseudo monoid in $\M\xmod$.
\end{proposition}
\begin{proof}
  From Theorem~\ref{thm:6} we get an equivalence $\M a_R \colon \lbrace R_\bullet,R_\bullet \rbrace \isoto \lbrace \stR,\stR\rbrace$.
\end{proof}
Here is an explicit construction of a quasi-inverse to $\M a_R$. By Theorem~\ref{thm:4}, or rather the generalization of the proof of Theorem~\ref{thm:1}, the functor $u$ has a quasi-inverse $v$ computed by the fiber product construction
which yields an $n$-\bfl $E_n = v(t(n))$; for a composition $t(n) ( t(m_1),\dots,t(m_n))$ we obtain a morphism
\begin{equation*}
  v( t(n) ( t(m_1),\dots,t(m_n)) ) \lisoto E_n(E_{m_1},\dots, E_{m_n}).
\end{equation*}
Thus, for any
\begin{equation*}
  \theta_{n;m_1,\dots,m_n}\colon t(m_1+\dots +m_n) \to t(n) ( t(m_1),\dots,t(m_n))
\end{equation*}
in $\M \stR$, we obtain
\begin{equation*}
  E_{m_1+\dots +m_n} \lisoto v( t(n) ( t(m_1),\dots,t(m_n)) ) \lisoto E_n(E_{m_1},\dots, E_{m_n}),
\end{equation*}
eventually arriving at~\eqref{eq:42}.

The notion of \strong{morphism} for both weakly categorical rings and crossed modules is straightforward. In each instance, the notion is a special case of that of morphism of pseudo-monoid examined in sect.~\ref{sec:formalism-monoid}.
\begin{notation*}
  From now on it will be convenient to refer to the monoidal structure of $\stR$ as a group-like object as the ``intrinsic'' or ``internal'' one, and denote it by a plus. Correspondingly, the relative unit object will be denoted by $0_\stR$, or simply $0$, if no confusion is bound to arise. (This retroactively justifies the choice of the attribute ``biadditive'' for the functor $m$.) The second monoidal structure $(m,\mu)$ will be referred to as the ``extrinsic'' or ``external'' one. The result of the application of $m$ will be denoted by a juxtaposition. The corresponding unit object, if it exists, will be denoted by $I_\stR$ or simply $I$.
\end{notation*}

\section{The presentation of a categorical ring}
\label{sec:pres-categ-ring}

In this section we show that the presentation of a (weak) categorical ring or ring-like stack is a crossed module with certain additional properties. In particular, its zeroth and first homotopy sheaves are, respectively, a ring and a bimodule over that ring, as in the standard case.

Let $\stR$ be as above.  Let $A = \pi_0(\stR)$ and $M=\pi_1(\stR)$ be the sheaves of connected components and of automorphisms of the (additive) identity object.  Since the definition of $\pi_i$ applies to the underlying (group-like) stack, by the usual arguments $A$ and $M$ are also zeroth and first homotopy sheaves of any crossed module $R_\bullet$ used to present $\stR$. In fact the homotopy kernel of the projection $\varpi \colon \stR \to A$ is identified with the Picard stack $\stM = \tors (M)$  of $M$-torsors. One has the familiar diagram
\begin{equation*}
  \xymatrix{%
    M \ar[d] \ar@[blue][r]^\imath & R_1 \ar@[blue][d]^\del \ar[r]& 0 \ar[d]\\
    0 \ar[d]\ar[r] & R_0 \ar[d]^\pi \ar@[blue][r]^q & A \ar@{=}[d] \\
    \stM \ar[r] & \stR \ar[r]^\varpi & A
    \save "1,1"."2,3" *[F-:magenta]\frm{} \restore
  }
\end{equation*}
with homotopy-exact columns and where the top two levels of each column (the part within the box above) provide a presentation of the item immediately below. The colored sequence in it is the standard extension of $A$ by $M$ by way of the crossed module $R_\bullet$. Of course, the top two rows, considered as a sequence of complexes of length one, is exact only in the sense that the group-like categories or stacks they determine form a short exact sequence.

For objects $X_1,X_2,\dots, X_n$ of $\stR$ define the result of applying the $n$-ary operation $t(n)$ simply by $X_1X_2\cdots X_n$. Then, by defining $[X_1][X_2]\cdots [X_n] \coloneqq [X_1X_2\cdots X_n]$, we immediately get that $\varpi\colon \stR\to A$ is a morphism of pseudo monoids, so that $A$ is a ring—unital if $\stR$ possesses a unit object $I_\stR$ for the external monoidal structure. (It follows that $\stM$ behaves like a bilateral ideal in $\stR$, namely the monoidal structure on $\stR$, and more precisely its binary operation, restricts to a pair of actions:
\begin{equation*}
  \stR\times \stM \lto \stM \quad \text{and} \quad
  \stM\times \stR\lto \stM.\text{)}
\end{equation*}
While there is no direct map $R_0\times R_0\to R_0$, the exterior monoidal structure $m_2=t(2)\colon \stR\times \stR\to \stR$ is ``covered'' by the diagram
\begin{equation}
  \label{eq:43}
  \vcenter{%
  \begin{xy}
    (0,0)*+{R_0\times R_0}="r1";%
    (10,10)*+{E}="e";%
    (20,0)*+{R_0}="r2";%
    {\ar@{->}_p "e"; "r1"};%
    {\ar@{->}^\jmath "e"; "r2"}
  \end{xy}}
\end{equation}
which is part of the structure of the biextension.  Given a pair $(x,y)\in R_0\times R_0$, and the choice of a point in the fiber $e\in E_{x,y}$, we get a ``value'' $\jmath (e)\in R_0$.  The latter is of course only defined up to shifting $e$ by the action of $R_1$, namely $\jmath (e\, r) = \jmath(e)+\del r$, where $r\in R_1$.  The connected component $q(\jmath(e)))\in A$ is well-defined and independent of all choices. Moreover, by Theorem~\ref{thm:1}, or rather its proof, we have $E_{x,y} = m_2(\pi (x),\pi (y))$, and therefore $\varpi (E_{x,y}) = \varpi (\pi(x))\varpi (\pi(y))= q(x)q(y)$. We conclude that, with $e\in E_{x,y}$ as above, $q(\jmath(e))=q(x)q(y)$.  Put differently, the diagram displayed above, if interpreted as a ``virtual ring structure'' for $R_0$, covers the multiplication map of $A$.  This consideration extends to any number of variables.

Hence the following standard fact holds:
\begin{proposition}
  \label{prop:14}
  \label{lem:7}
  $M$ is an $A$-bimodule.
\end{proposition}
We recall the main idea of the proof, as it will be needed to translate the above fact in terms of the presentation of $\stR$ and the multi-extensions associated to its second monoidal structure.
\begin{proof}[Proof of~\ref{prop:14} (Sketch)]
  Recall that the standard way to relate the automorphisms of $0_\stR$ to those of any other object $X$ of $\stR$ is to use the diagram
  \begin{equation*}
    \xymatrix@C+1pc{%
      *!<-3pt,0pt>{X + 0_\stR} \ar[r]^{\id_X+\alpha} \ar[d]
      & *!<-3pt,0pt>{X + 0_\stR} \ar[d] \\
      X \ar[r]_{\tilde\alpha} & X
    }
  \end{equation*}
  Left and right multiplications by an object $Y$ (i.e.\ the additive functors $m_2(Y,-)$ and $m_2(-,Y)$) translate the above diagram into the corresponding ones for automorphisms $Y\alpha$ and $\alpha Y$. If $\beta\colon Y\to Z$, so that $[Y]=[Z]\in A$, then we arrive at the square
  \begin{equation*}
    \xymatrix@C+1pc{%
      XY + 0_\stR \ar[r]^{\id+\alpha Y} \ar[d]_{\id_X\beta + \id_0}
      & XY + 0_\stR \ar[d]^{\id_X\beta + \id_0} \\
      XZ + 0_\stR \ar[r]_{\id + \alpha Z}
      & XZ + 0_\stR 
    }
  \end{equation*}
  from which we conclude that $\alpha Y=\alpha Z$.  The situations for $Y\alpha$, $Z\alpha$ is analogous.
\end{proof}
Combining this with Proposition~\ref{prop:13} we have:
\begin{corollary}
  \label{cor:1}
  Every categorical ring $\stR$ has a presentation
  \begin{equation*}
    \xymatrix@1{0 \ar[r]& M\ar[r]& R_1\ar[r]^\del& R_0\ar[r]& A\ar[r]& 0,}
  \end{equation*}
  where $R_\bullet:R_1\to R_0$ is a braided crossed module, $A=\pi_0(\stR)$ a ring and $M=\pi_1(\stR)$ an $A$-bimodule.\qed
\end{corollary}
In fact, something more can be said about the presentation, despite $R_0$ having the weaker structure given by~\eqref{eq:43}.  In addition, we will need to describe the $A$-bimodule structure of $M$ in terms of~\eqref{eq:43} and the biextension $E$.

Consider the morphism examined in Remark~\ref{rem:4}:  for an arrow $x \to x' = x + \del r$ of $[R_1\to R_0]\sptilde$, with $x\in R_0$ and $r\in R_1$, we have
the morphism of $(R_1,R_0)$-torsors $\rho_{y}\colon E_{x,y}\to E_{x',y}$, $e\mapsto e\times_1 s_1(r,y)$. If $\del r=0$, so $r\in \pi_1(\stR) = M$, then $\rho_{y}$ is an automorphism of $E_{x,y}$, hence $\rho\in M$ as well (depending on $x,y$), since in general $\Aut_{\sttors(R_1,R_0)}(E)\iso M_{R_0\times R_0}$.  Similarly, we construct an automorphism $\lambda_x\colon E_{x,y}\to E_{x,y}$ corresponding to the automorphism of $y\in R_0$  given by $r\in R_1$.
\begin{lemma-def}
  \label{lem:8}
  The automorphism $\rho_y$ (resp.\ $\lambda_x$) is identified with a section $-ry$ (resp.\ $-xr$) of $M$ such that if $y,y'\in R_0$ (resp.\ $x,x'\in R_0)$ and $q(y)=q(y')\in A$ (resp. $q(x)=q(x')\in A$), then $ry=ry'$ (resp.\ $xr=x'r$).

  With the above notations, define $ar\coloneqq xr$ and $rb\coloneqq ry$, where $a=q(x),b=q(y)\in A$.
\end{lemma-def}
\begin{remark}
  \label{rem:10}
  The apparently bizarre choice of inserting an inverse (denoted additively) is due to the point raised in Remark~\ref{rem:1}.
\end{remark}
\begin{proof}[Proof of the Lemma.]
  The identification between the automorphisms of $E$ and $M$ is standard: for any section $e\in E_{x,y}$ we have $\rho_y(e) = e\, r_{\rho_y}(e)$, with $r_{\rho_y}(e)\in M$.  On the other hand, this must be equal to $e\times_1 s_1(r,y)$, the latter coming from the composition
  \begin{equation*}
    E_{x,y} \lto E_{x,y} \wedge^{R_1} E_{0,y} \lto E_{x,y}.
  \end{equation*}
  Recall from sect.~\ref{sec:biext-braid} that, analogously to the standard case of abelian biextensions \cite{MR0354656-VII}, for $x,y\in R_0$, the fibers $E_{x,0}$ and $E_{0,y}$ are canonically identified with the unit $(R_1,R_0)$-torsor: the identification
  \begin{equation*}
    R_1 \lisoto E_{0,y},\quad 0_{R_1} \longmapsto e_y,
  \end{equation*}
  sends the unit section of $R_1$ (which we write as $0$ according to the current additive convention) to the central section $e_y\in E_{0,y}$. Note that as a result $\jmath (e_y) = 0_{R_0}$. By the above, we have
  \begin{equation}
    \label{eq:44}
    s_1(r,y) = e_y\, (-ry),
  \end{equation}
  where $ry \coloneqq -r_{\rho_y}(e_y)\in R_1$, and the juxtaposition stands for the right $R_1$-action. In fact, we have $ry\in M$, since $\jmath\circ s_1$ is trivial (cf.\ sect.~\ref{sec:butterflies}).

  Now, assume we have $y\to y' = y + \del u$, $u\in R_1$. We have
  \begin{equation*}
    \vcenter{\xymatrix@C+1pc@R-1pc{%
      & E_{0,y} \ar[dd]^u \\
      R_1 \ar[ur]^{e_y} \ar[dr]_{e_{y'}} \\
      & E_{0,y'} 
    }}\qquad e \longmapsto e\times_2 s_2(0,u)
  \end{equation*}
  Using the identities~\eqref{eq:24} and the restriction condition~\ref{item:4} in Definition~\ref{def:3} we calculate:
  \begin{align*}
    s_1(r,y)\times_2 s_2(0,u) & = s_1(r,y)\times_2 s_2(\del r,u) \\
                              & = s_1(r,y)\times_2 s_1(r, \del u) \\
                              & = s_1(r,y + \del u) \\
                              & = s_2(r, y').
  \end{align*}
  Using~\eqref{eq:44} the above calculation gives
  \begin{align*}
    e_y\, (ry) \times_2s_2(0,u) & = e_y\times_2 (ry)\,s_2(0,u) \\
                                & = e_y\times_2 s_2(0,u)\,(ry) \\
                                & = e_{y'} \, (ry')\,.
  \end{align*}
  The fact that the identification is canonical gives $e_y\times_2 s_2(0,u) = e_{y'}$ and hence $ry = ry'$.

  The situation with $e_x\in E_{x,0}$, $xr$, and $x'r$ is entirely analogous.
\end{proof}
\begin{remark}
  \label{rem:11}
In the extension of $A$ by $M$ above, $R_1$ and $R_0$ are only groups, in general, even though the crossed module $R_\bullet$ they form comes equipped with a braiding. In particular, unless certain strong triviality conditions on the biextension $E_2=t(2)\in \grdbfl(R_\bullet,R_\bullet; R_\bullet)$ hold, $R_0$ is \emph{not} a ring, in general. This remark applies to stock categorical rings as well, which therefore admit a presentation of the above type with $R_\bullet$ a braided symmetric crossed module.  This is in sharp contrast with the so-called \emph{Ann-Categories} \autocites(see)(){MR2065328}{MR2458668}{MR3085798}, whose underlying categories are Picard groupoids. In a companion paper \cite{ann2015} we prove that Picard stacks with a ring-like structure of this type admit presentations given by crossed bimodules, namely crossed modules such as in the presentation above where $R_0$ is a ring and $R_1$ is an $R_0$-bimodule, with $\del$ a bimodule morphism satisfying an appropriate version of the Pfeiffer identity.  The general comparison between these structures is intricate and it will be addressed elsewhere.
\end{remark}

\section{Decomposition of categorical rings and the cohomology of rings}
\label{sec:decomp}

Let $\stR$, $A$, and $M$ be as above.  In general, a \strong{decomposition} \cite{MR95m:18006} of $\stR$ consists in the choice of local data (objects, morphisms) subordinate to a hypercovering of $A$ leading to the calculation of cohomological invariants of $\stR$.  We use the pseudo-monoidal structure carried by a presentation of $\stR$ for this purpose.

We let $L_i^\bullet(A)$, $i=2,3$ be the sharp truncations of the iterated bar complexes $B(A,i)$, $i=2,3$ of $A$ as an abelian group. By \cite[\S 11]{MR0098773} it carries a product structure, defined explicitly up to cells in $B(A,2)$.  For $i=2,3$ we let $H^3_i(A,M)$ denote the (hyper)cohomology groups computed using the multiplicative bar construction over $L^i_\bullet(A)$. From the previous reference we have the $H^3_3(A,M)\iso \HML^3(A,M)$, the third Mac~Lane cohomology of $A$ with values in $M$.  Furthermore, we let $\widetilde H^3_2(A,M)$ denote the group of twisted classes satisfying~\eqref{eq:51} below.
\begin{theorem}
  \label{thm:7}
  There is a bijective correspondence between equivalence classes of  weak ring-like stacks $\stR$, with $A=\pi_0(\stR)$ and $M=\pi_1(\stR)$, and twisted classes in $\widetilde H^3_2(A,M)$ defined below (cf.\ Definition~\ref{def:8}).  In the unital case, i.e.\ when the external monoidal structure of $\stR$ has a unit element, $A$ is likewise unital and the underlying braiding of $\stR$ is necessarily symmetric. Hence the weak ring-like structure of $\stR$ is fully ring-like, and $[\stR]\in \widetilde H^3_2(A,M)\iso \HML^3(A,M)$.
\end{theorem}
The statement follows from Proposition~\ref{prop:16}, Corollary~\ref{cor:2}, and Proposition~\ref{prop:17} below. The rest of this section is devoted to the details of the proof of the theorem.

\subsection{The bar complex}
\label{sec:complexes}

We establish the notation for the iterated bar complex of the abelian group $A$. We quote, with minor changes in the notation, from~\cite[\S 7]{MR3180827}.  The part of $B^\infty(A)$ of interest is the following complex, written homologically:
\begin{equation*}
  L^3_\bullet(A) :
  \xymatrix@1{%
    \ZZ [A] & \ZZ[A^2] \ar[l]_{\del_1} & \ZZ[A^3]\oplus \ZZ[A^2] \ar[l]_(0.6){\del_2} &
    \ZZ[A^4]\oplus \ZZ[A^3] \oplus \ZZ[A^3] \oplus \ZZ[A^2] \ar[l]_(0.6){\del_3} 
  }
\end{equation*}
We place the complex in homological degrees $[0,3]$, with generators:
\begin{center}
  \begin{tabular}[c]{ c | c  c  c  c | }
    \text{Degree} & \multicolumn{4}{|c|}{\text{Generators}} \\ \hline\hline
    0 & \vphantom{\rule{0pt}{3ex}} $[a]$
      & & & \\
    1 & $[a \binmid_1 b]$
      & & & \\ 
    2 & $[a \binmid_1 b \binmid_1 c]$ & $[a\binmid_2 b]$
      & &  \\ 
    3 & $[a \binmid_1 b \binmid_1 c \binmid_1 d]$
      & $[a \binmid_1 b \binmid_2 c]$
      & $[a \binmid_2 b \binmid_1 c]$ & $[a\binmid_3 b]$
  \end{tabular}
\end{center}
The differential is given by:
\begin{align*}
  \del_1 [a\binmid_1 b] & = [b] - [a+b] + [a] \\
  \del_2 [a\binmid_1 b\binmid_1 c] &= [b\binmid_1 c]
                                     -[a+b\binmid_1 c] + [a\binmid_1 b+c]
                                     -[a\binmid_1 b] \\
  \del_2 [a\binmid_2 b] & = [a\binmid_1 b] - [b\binmid_1 a] \\
  \del_3 [a\binmid_1 b\binmid_1 c\binmid_1 d] &= [b\binmid_1 c\binmid_1 d]
                                                - [a+b\binmid_1 c\binmid_1 d]
                                                + [a\binmid_1 b+c\binmid_1 d]
                                                - [a\binmid_1 b\binmid_1 c+d]
                                                + [a\binmid_1 b\binmid_1 c] \\
  \del_3 [a\binmid_1 b\binmid_2 c] &= [a\binmid_1 b\binmid_1 c]
                                     - [a\binmid_1 c\binmid_1 b]
                                     + [c\binmid_1 a\binmid_1 b]
                                     - [b\binmid_2 c] +[a+b\binmid_2 c] -[a\binmid_2 c] \\
  \del_3 [a\binmid_2 b\binmid_1 c] &= [a\binmid_1 b\binmid_1 c]
                                     - [b\binmid_1 a\binmid_1 c]
                                     + [b\binmid_1 c\binmid_1 a]
                                     + [a\binmid_2 c] -[a\binmid_2 b+c] +[a\binmid_2 b] \\
  \del_3 [a\binmid_3 b] &= [a\binmid_2 b] + [b\binmid_2 a].
\end{align*}
The subcomplex $L^2_\bullet(A)$ is the one obtained by dropping the component $\ZZ[A^2]$ of $L^3_3(A)$, namely the one generated by symbols $[a\binmid_3 b]$. The degrees are shifted in a way compatible with the reset needed to form the infinite bar construction, so in effect we have $L^i_\bullet(A) = B(A,i)[-i]$, for $i=2,3$, and therefore
\begin{equation*}
  H_2(L^2_\bullet(A)) \iso H_{4}(K(A,2)),\qquad
  H_2(L^3_\bullet(A)) \iso H_{5}(K(A,3)).
\end{equation*}
In the stable situation, this shift is compatible with the degrees in the $Q$ construction.

\subsection{Product structure}
\label{sec:product-l2a}

Let now $A$ be considered with its ring structure.  We describe the product structure on $L^2_\bullet(A)$ and $L^3_\bullet(A)$ \cite[see][\S 11]{MR0098773}. The non zero products among the generators, \emph{up to products taking values in degrees $\leq 2$,}\footnote{The remaining part of the product structure does not appear in \cite{MR0098773} and it is unknown, to the author, at least, whether it is explicitly available.} are the following:
\begin{alignat*}{2}
  [a][b] &= [ab] \\
  [a][b\binmid_1 c]&=[ab\binmid_1 ac]&\qquad [a\binmid_1 b][c]&=[ac\binmid_1 bc]\\
  [a][b\binmid_1 c\binmid_1 d]&=[ab\binmid_1 ac\binmid_1 ad] &\qquad
  [a\binmid_1 b\binmid_1 c][d]&=[ad\binmid_1 bd\binmid_1 cd] \\
  [a][b\binmid_2 c]&=[ab\binmid_2 ac]&\qquad [a\binmid_2 b][c]&=[ac\binmid_2 bc]
\end{alignat*}
and the most interesting one\footnote{There appears to be a discrepancy in the second slot of the second term of the right hand side, presumably a misprint in \cite{MR0098773}.}:
\begin{equation*}
  [a\binmid_1 b][c\binmid_1 d] =
  [ac\binmid_1 bc\binmid_1 ad+bd] -[ac\binmid_1 ad\binmid_1 bc+bd]
  +[ad\binmid_1 bc\binmid_1 bd] -[bc\binmid_1 ad\binmid_1 bd]
  -[bc\binmid_2 ad].
\end{equation*}
With this product $L^2_\bullet(A)$ and $L^3_\bullet(A)$ become DGAs.  Furthermore,
both are equipped with the augmentation $\eta\colon L^k_\bullet(A)\to A$ given by $\eta([a])=a$, for $a\in A$, and zero in all other degrees.  Thus, they become augmented DGAs to which we can apply the (reduced) bar construction $\Bar B_{i,\bullet}(A) \coloneqq \Bar B_N(L^i_\bullet(A),\eta)$ of loc.\ cit.\ (more details below).  We are interested in the resulting cohomology groups $H^3(\Hom(\Bar B_{i,\bullet}(A),M))$, where $M$ is an $A$-bimodule as above. As remarked, for $i=3$ we have $H^3(\Hom(\Bar B_{3,\bullet}(A),M)) \iso \HML^3(A,M)$.

\subsection{The multiplicative bar construction and third cohomology of rings}
\label{sec:bar-construction}

To avoid typographical clutter, the generators of $\Bar B_{i,\bullet}(A)$ are denoted $\llbracket u_1,\dots,u_n\rrbracket$, where the $u_k$s are homogeneous elements of $L^i_\bullet(A)$. The degree is
\begin{equation*}
  \deg \llbracket u_1,\dots,u_n\rrbracket =
  n + \abs{u_1} +\dots + \abs{u_n},
\end{equation*}
where we let $\abs{u_k}=\deg u_k$. We quote \cite[p.\ 323]{MR0098773} the expression for the differential in the bar complex; it is $\del_{\mathrm{tot}} = \del' + \del''$, where
\begin{gather*}
  \del' \llbracket u_1,\dots,u_n\rrbracket = - \sum_{i=1}^n (-1)^{\epsilon_{i-1}}
  \llbracket u_1,\dots,\del u_i,\dots, u_n\rrbracket \\
  \del'' \llbracket u_1,\dots,u_n\rrbracket =
  \eta(u_1) \llbracket u_2,\dots,u_n\rrbracket
  +\sum_{i=1}^{n-1} (-1)^{\epsilon_i}
  \llbracket u_1,\dots,u_iu_{i+1},\dots, u_n\rrbracket
  + (-1)^{\epsilon_n} \llbracket u_1,\dots, u_{n-1}\rrbracket\eta(u_n),
\end{gather*}
where we have set $\epsilon_i=\deg \llbracket u_1,\dots,u_i\rrbracket$.  On the right hand side of the expression for $\del'$ the inner $\del u_i$ indicates the differential in the complex $L^i_\bullet(A)$.  

The cells of total dimension (=degree) 3 are:
\begin{center}
  \begin{tabular}[c]{c | c c |}
    $n$ & \multicolumn{2}{c|}{Generators} \\ \hline \hline
    3 & \vphantom{\rule{1em}{3ex}}$\llbracket [a], [b], [c] \rrbracket$ & \\
    2 & $\llbracket [a\binmid_1 b] , [c] \rrbracket$ & $\llbracket [a], [b\binmid_1 c] \rrbracket$ \\
    1 & $\llbracket [a\binmid_1 b\binmid_1 c]\rrbracket$
      & $\llbracket [a\binmid_2 b]\rrbracket$
  \end{tabular}
\end{center}
Therefore a 3-dimensional cochain over $\Bar B_{i,\bullet}(A)$ with values in $M$ is a 5-tuple $\xi=(f,\alpha_1,\alpha_2,f_+,g_+)$, where
\begin{gather*}
  f \colon A^3 \lto M \\
  \alpha_1\colon A^2\times A \lto M,\quad
  \alpha_2\colon A\times A^2 \lto M \\
  f_+\colon A^3 \lto M,\quad g_+\colon A^2 \lto M.
\end{gather*}
To write the cocycle condition $\delta\xi=\xi\circ (\del'+\del'') = 0$ in explicit form we need the cells of dimension four:
\begin{center}
  \begin{tabular}[c]{c | c c c c c |}
    $n$ & \multicolumn{5}{c|}{Generators} \\ \hline\hline
    4 & \vphantom{\rule{1em}{3ex}}$\llbracket [a], [b], [c], [d] \rrbracket$
      & & & &\\
    3 & $\llbracket [a\binmid_1 b] , [c], [d] \rrbracket$
      & $\llbracket [a], [b\binmid_1 c], [d] \rrbracket$
      & $\llbracket [a], [b], [c\binmid_1 d] \rrbracket$ & & \\
    2 & $\llbracket [a\binmid_1 b], [c\binmid_1 d]\rrbracket$
      & $\llbracket [a\binmid_1 b\binmid_1 c], [d]\rrbracket$
      & $\llbracket [a\binmid_2 b], [c]\rrbracket$
      & $\llbracket [a], [b\binmid_1 c\binmid_1 d]\rrbracket$
      & $\llbracket [a], [b\binmid_2 c]\rrbracket$ \\
    1 & $\llbracket [a\binmid_1 b\binmid_1 c\binmid_1 d]\rrbracket$
      & $\llbracket [a\binmid_1 b\binmid_2 c]\rrbracket$
      & $\llbracket [a\binmid_2 b\binmid_1 c]\rrbracket$
      & $\llbracket [a\binmid_3 b]\rrbracket$\textdagger &
  \end{tabular}
\end{center}
The element marked with a (\textdagger) would only figure in the bar complex $\Bar B_{3,\bullet}(A)$. 
In explicit form, the cocycle condition on the five components of $\xi$ is quite complex, consisting of several equations. For ease of reading we arrange them in different groups we have (see below for specific comments):
\begin{subequations}
  \label{eq:45}
\begin{equation}
  \label{eq:46}
  a\,f(b,c,d) -f(ab,c,d) + f(a,bc,d) -f(a,b,cd) +f(a,b,c)\, d = 0
\end{equation}
\begin{equation}
  \label{eq:47}
  \begin{aligned}
    f(b,c,d) -f(a+b,c,d) +f(a,c,d) &=
    \alpha_1(ac,bc;d) -\alpha_1(a,b;cd) +\alpha_1(a,b;c)\, d \\
    -f(a,c,d) +f(a,b+c,d) -f(a,b,d) &= 
    a\, \alpha_1(b,c;d) -\alpha_1(ab,ac;d) -\alpha_2(a;bd,cd) +\alpha_2(a;b,c)\, d \\
    f(a,b,d) -f(a,b,c+d) +f(a,b,c) &=
    a\, \alpha_2(b;c,d) -\alpha_2(ab;c,d) +\alpha_2(a;bc,bd)
  \end{aligned}
\end{equation}
\begin{multline}
  \label{eq:48}
  f_+(ac,bc,ad+bd) -f_+(ac,ad,bc+bd) + f_+(ad,bc,bd) -f_+(bc,ad,bd) -g_+(bc,ad) = \\
  \alpha_1(a,b;d) -\alpha_1(a,b;c+d) +\alpha_1(a,b;c)
  +\alpha_2(b;c,d) -\alpha_2(a+b;c,d) +\alpha_2(a;c,d) \\
\end{multline}
\begin{equation}
  \label{eq:49}
  \begin{aligned}
    f_+(ad,bd,cd) -f_+(a,b,c)\,d &=
    -\alpha_1(b,c;d) +\alpha_1(a+b,c;d) -\alpha_1(a,b+c;d) +\alpha_1(a,b;d) \\
    f_+(ab,ac,ad) -a\, f_+(b,c,d) &=
    \alpha_2(a;c,d) -\alpha_2(a;b+c,d) +\alpha_2(a;b,c+d) -\alpha_1(a;b,c) \\
    g_+(ac,bc) -g_+(a,b)\, c &= -\alpha_1(a,b;c) +\alpha_1(b,a;c) \\
    g_+(ab,ac) -a\, g_+(b,c) &= \alpha_2(a;b,c) -\alpha_2(a;c,b)
  \end{aligned}
\end{equation}
\begin{equation}
  \label{eq:50}
  \begin{gathered}
    f_+(b,c,d) -f_+(a+b,c,d) +f_+(a,b+c,d) -f_+(a,b,c+d) +f_+(a,b,c) = 0\\[.5ex]
    \begin{aligned}
      f_+(a,b,c) -f_+(a,c,b) +f_+(c,a,b) &= g_+(b,c) -g_+(a+b,c) +g_+(a,c) \\
      f_+(a,b,c) -f_+(b,a,c) +f_+(b,c,a) &= -g_+(a,c) +g_+(a,b+c) -g_+(a,b)
    \end{aligned}\\[.5ex]
    g_+(a,b) + g_+(b,a) = 0 \,.
  \end{gathered}
\end{equation}
\end{subequations}
\begin{remarks}
  \label{rem:12}
  \hfill
  \begin{enumerate}
  \item The last group of equations~\eqref{eq:50} is closed and it is the condition for the pair $(f_+,g_+)$ to be an Eilenberg Mac~Lane cocycle. Thus $(f_+,g_+)$ represents a class in $H^4(K(A,2),M)$ or $H^5(K(A,3),M)$, depending on whether the last term is included. 
  \item The block~\eqref{eq:49} can be re-written as follows.  Let $\lambda_a$ and $\rho_a$ denote the left and right multiplications by $a\in A$. Also, write $a\,f_+$ and $f_+\,a$ for the action of $a\in A$ on the values of $f_+$.  Same for $g_+$. 
    Then we can rewrite the block as
    \begin{equation*}
      \begin{aligned}
        \rho_d^*f_+ -f_+ \,d&= -\alpha_1(-,-;d)\circ \del_3 \\
        \lambda_a^*f_+ -a\, f_+ &= \alpha_2(a;-,-)\circ \del_3 \\
        \rho_c^*g_+ -g_+\, c &= -\alpha_1(-,-;c)\circ \del_2 \\
        \lambda_a^*g_+ -a\, g_+ &= \alpha_2(a;-,-)\circ \del_2
      \end{aligned}
    \end{equation*}
    The block~\eqref{eq:49} expresses the invariance of the \emph{class} of $(f_+,g_+)$ under left and right multiplication in $A$.
 \item Equation~\eqref{eq:46} has the familiar form of a Hochschild cocycle. (As it might be expected, it arises from the associativity constraint on the exterior monoidal structure of $\stR$, as it will be shown below.) The behavior of $f$ with respect to the additivity, i.e. its failure to be multilinear, is expressed by the equation block~\eqref{eq:47}.
  \item As it will be explained below, the meaning of the somewhat more intricate relation~\eqref{eq:48}, is that it expresses the compatibility (interchange law) between the partial composition laws of the exterior monoidal structure of $\stR$. It relates to the last of the relations in the Mac~Lane's product structure of sect.~\ref{sec:product-l2a}.
  \end{enumerate}
\end{remarks}
It is well known that $H_5(K(A,3))\iso A/2A$ and that $H^5(K(A,3),M)\iso \Hom(A/2A,M)\iso \Hom(A,{}_2M)$~\cites[\S 23]{MR0065162}[see also][\S 7.3]{MR3180827}. The former isomorphism is given by $a\mapsto [a\binmid_2a]$, the latter sends the class of $[(f_+,g_+)]$ to the (linear) map $a\mapsto g_+(a,a)$. In fact, dropping the last of equations~\eqref{eq:50}, the same assignment gives a quadratic map from $A$ to $M$. (Recall that $q\colon A \to M$ is quadratic if $q(na) = n^2q(a)$ for all $n\in \ZZ$, and if the associated symmetric function $\Delta q(a,b) = q(a+b) -q(a) -q(b)$, $a,b\in A$, is bilinear.) Thus $H^5(K(A,3),M)\iso \Hom(\Gamma_2(A),M)$, where $\Gamma_2(A)$ is the Whitehead functor, i.e.\ the degree four component of the divided power algebra $\Gamma_\bullet(A)$ over $A$.

Recall that $m\in M$ is \emph{central} if $am=ma$ for all $a\in A$. Let $M^A$ be the $A$-module of central elements of $M$. Assume that $\xi$  represents a class in $H^3(\Hom(\Bar B_{2,\bullet}(A),M))$, that is, the pair $(f_+,g_+)$ satisfies, as part of the full set $\xi=(f,\alpha_1,\alpha_2,f_+,g_+)$, equations~\eqref{eq:49} and~\eqref{eq:50}, \emph{except the last one.} We have the following observation.
\begin{proposition}
  \label{prop:15}
  Let $A$ be unital. We have an isomorphism
  \begin{equation*}
    H^3(\Hom(\Bar B_{2,\bullet}(A),M))\iso H^3(\Hom(\Bar B_{3,\bullet}(A),M))\iso \HML^3(A,M).
  \end{equation*}
  Moreover, the map $\HML^3(A,M)\to \Hom(A,{}_2M)$ \cite[\S 11]{MR0098773} factors through ${}_2M^A\iso \Hom_A(A,{}_2M^A)$.
\end{proposition}
\begin{proof}
  Let $q$ be the map $a\mapsto g_+(a,a)$.
  The last two equations of~\eqref{eq:49} imply that $q$ is an $A$-bimodule homomorphism. As such, it is determined by a central element in $M$. This proves the first isomorphism. The second follows from~\cite{MR0098773}.
\end{proof}
\begin{remark}
  \label{rem:13}
  Because the complexes $\Bar B_{2,\bullet}(A)$ and $\Bar B_{3,\bullet}(A)$ are equal below degree $3$, their cohomologies coincide for $n\leq 2$. Combined with \cite[\S 6]{MR1702420}, we have that the statement of Proposition~\ref{prop:15} holds in general when $H^\bullet$ is interpreted as hypercohomology.
\end{remark}

The following variant of the previous constructions will be useful below. Given a 5-tuple $\xi$ as above, consider the map $\beta_\xi \colon B_{2,4}(A)_2\to M$ given by (the subscript $2$ refers to the elements in $B_{2,4}(A)$ with $n=2$):
\begin{equation*}
  \beta_\xi (\llbracket u_1, u_2\rrbracket) =
  \begin{cases}
    g_+(ab,ac) + g_+(ac,ab) & \llbracket u_1, u_2\rrbracket =
    \llbracket [a], [b\binmid_2 c]\rrbracket, \\
    g_+(ac,bc) + g_+(bc,ac) & \llbracket u_1, u_2\rrbracket =
    \llbracket [a\binmid_2 b],[c]\rrbracket, \\
    0 & \text{all other cases.}
  \end{cases}
\end{equation*}
Further, we define $\beta_\xi$ to be zero on the other components of $B_{2,4}(A)$. Consider the twisted cocycle equation condition:
\begin{equation}
  \label{eq:51}
  D\xi\coloneqq \delta\xi + \beta_\xi = 0.
\end{equation}
This replaces the last two equations in the block~\eqref{eq:49} with
\begin{equation}
  \label{eq:52}
  \begin{aligned}
    -g_+(bc,ac) -g_+(a,b)\, c &= -\alpha_1(a,b;c) +\alpha_1(b,a;c) \\
    -g_+(ac,ab) -a\, g_+(b,c) &= \alpha_2(a;b,c) -\alpha_2(a;c,b)
  \end{aligned}
\end{equation}
and it still drops the last one from~\eqref{eq:50}. Observe that for $i=3$ equation~\eqref{eq:51} is vacuously equal to the set~\eqref{eq:45}.
\begin{remark}
  \label{rem:14}
  Note that replacing $\xi$ with $\xi'=\xi+\delta \nu$ has the effect of adding to $g_+$ the alternation of a map $h_+\colon A\times A\to M$, so that $\beta_\xi$ does not depend on the particular choice of $g_+$ within its class modulo coboundaries. Hence we get a well defined class of solutions of equation~\eqref{eq:51} modulo adding coboundaries.
\end{remark}
\begin{definition}
  \label{def:8}
  Let us denote by $\widetilde H^3_2(A,M)$ the hypercohomology group of classes satisfying~\eqref{eq:51} in degree three, relative to $A$ (cf.\ Remark~\ref{rem:13}).
\end{definition}
There is an evident map $\HML^3(A,M)\to \widetilde H^3_2(A,M)$, which, at least in the cases we like to consider, is an isomorphism.  Indeed we have
\begin{lemma}
  \label{lem:9}
  Let $A$ be unital. Then~\eqref{eq:51} and~\eqref{eq:45} are equivalent.  Therefore $\widetilde H^3_2(A,M)\iso \HML^3(A,M)$.
\end{lemma}
\begin{proof}
  Choosing $a$ or $c=1$ in~\eqref{eq:52} shows that the last equation in in~\eqref{eq:50} holds. (Again, use Remark~\ref{rem:13} for the general situation.)
\end{proof}

\subsection{Decomposition of $\stR$}
\label{sec:decomposition-str}

For convenience of notation, let us write $[\del \colon R_1\to R_0] \equiv [\del \colon R\to \Lambda]$. Recall that since we denote the interior monoidal structure of $\stR$ additively, we do the same for $R$ and $\Lambda$ in the presentation.
\begin{remark}[\textbf{Warning}]
  In the following we use a set theoretic-type notation. However, in this sheaf theoretic context, the notation $a\in A$, or $a_1,\dots,a_n\in A$ means these are generalized points of $A$ of the form, say, $a\colon U\to A$, etc. where $U$ is an object of the topos. In fact, in order to properly carry out the hypercohomology calculations we will have to choose hypercovers of $A$ and of various simplicial objects associated to it, such as the various Eilenberg-Mac~Lane objects $K(A,2)$ and $K(A,3)$. We will proceed formally as in the pointwise case, and systematically appeal to the hypercohomology spectral sequence as in Remark~\ref{rem:13}. The actual hypercohomology arguments based are made precise in Appendix~\ref{sec:hyper}.
\end{remark}

For a point $a\in A$, we let $\Lambda_a=a^*\Lambda$ be the fiber. For any multi-extension $E_n$ in the weak ring-like structure of $R\to \Lambda$ denote by $E_{a_1,\dots, a_n}$ the pullback of $E$ to $\Lambda_{a_1}\times \dots\times \Lambda_{a_n}$, where $a_1,\dots,a_n\in A$.

There is an obvious morphism $\Lambda_{a_1}\times \dots \times \Lambda_{a_n}\to \Lambda_{a_1+\dots+a_n}$ covering the $n$-fold iteration of $+\colon A\times A\to A$.  If we assume a choice for a point $x_a\in \Lambda_a$ has been made for all points $a\in A$, i.e. we have a section $x$ of $q\colon \Lambda\to A$, then $x_a+x_b$ (the image of $(x_a,x_b)$ under $\Lambda_a\times \Lambda_b\to \Lambda_{a+b}$) and $x_{a+b}$ will in general be different. This gives rise to $\sigma\colon A\times A\to R$ by way of
\begin{equation*}
  x_{a+b} = x_a + x_b + \del \sigma_{a,b}.
\end{equation*}
In fact the objects $X_a\coloneqq \pi(x_a)$ of $\stR$ provide a decomposition of the sequence $\stM\to \stR\to A$, where $\stR$ is considered as a gerbe over $A$ \cite[\S 7]{MR95m:18006}. Note that the above corresponds to the morphism
\begin{equation*}
  \sigma_{a,b}\colon X_{a+b} \lto X_a + X_b
\end{equation*}
of $(R,\Lambda)$-torsors. Therefore we have the classical fact that the difference between the two possible comparisons between $X_a+X_b+X_c$ and $X_{a+b+c}$, and the application of equations~\eqref{eq:2} to $x_a$, $x_b$, and $x_c$ determines a pair $(f_+,g_+)$ satisfying the cocycle equations~\eqref{eq:50} above, hence a class in $H^4(K(A,2),M)$ or $H^5(K(A,3),M)$ \cite[see, e.g.\ ][]{MR1250465,MR1702420}. For future reference, the relevant relations are \cite[cf.\ ][]{MR95m:18006}:
\begin{equation}
  \label{eq:53}
  \begin{aligned}
    \sigma_{b,c} + \sigma_{a,b+c} - f_+(a,b,c)
    &= {\sigma_{a,b}}^{x_c} + \sigma_{a+b,c}, \\
    -g_+(a,b)+\sigma_{a,b} &= \braces {x_a} {x_b} + \sigma_{b,a}.
  \end{aligned}
\end{equation}
(There is an obvious generalization to $n$ variables, but the classical situation is sufficient to describe the decomposition with respect to the ``$+$'' operation of $\stR$.)

For the multiplication, i.e.\ the exterior monoidal structure, the reasoning at the end of sect.~\ref{sec:bimon-struct} implies we have instead
\begin{equation}
  \label{eq:54}
  \vcenter{%
  \begin{xy}
    (12,12)*+{E_{a_1,\dots,a_n}}="e";%
    (0,0)*+{\Lambda_{a_1}\times \dots \times \Lambda_{a_n}}="l1";%
    (24,0)*+{\quad\Lambda_{a_1\dots a_n}}="l2";%
    {\ar@{->}_{p} "e"; "l1"};%
    {\ar@{->}^{\jmath} "e"; "l2"};
  \end{xy}%
  }
\end{equation}
covering the multiplication in $A$. Using the choice of a section of the fibers $\Lambda_a$ as above, we have the isomorphism of $(R,\Lambda)$-torsors
\begin{equation}
  \label{eq:55}
  e_{a_1,\dots,a_n}\colon E_{a_1,\dots,a_n}\lisoto X_{a_1\dotsm a_n} = (R,x_{a_1\dotsm a_n})
\end{equation}
which follows from~\eqref{eq:54} and again the end of sect.~\ref{sec:bimon-struct}. We can assume this isomorphism is realized by the choice of a point $e_{a_1,\dots,a_n}\in E_{a_1,\dots,a_n}$ such that $\jmath(e_{a_1,\dots,a_n})=x_{a_1\dotsm a_n}$.

We can write the morphism~\eqref{eq:41} with respect to this choice of local data.  According to the proof of Theorem~\ref{thm:5}, let $a_{1,1},\dots,a_{1,m_1},\dots,a_{n,1},\dots, a_{n,m_n}\in A$.  For $i=1,\dots, n$ define $b_i = a_{i,1}\dotsm a_{i,m_i}$. Consider the points $e_i=e_{a_{i,1},\dots,a_{i,m_i}}\in E_{a_{i,1},\dots,a_{i,m_i}}$ and $e_{b_1,\dots,b_n}\in E_{b_1,\dots,b_n}$. Then $[e_1,\dots,e_n, e_{b_1,\dots,b_n}]$ is a point of the composition $E_n(E_{m_1},\dots, E_{m_n})$ over $\Lambda_{a_{1,1}}\times\dots \times \Lambda_{a_{n,a_n}}$, and $\jmath([e_1,\dots,e_n,e_{b_1\dots, b_n}]) = \jmath(e_{b_1,\dots,b_n}) = b_1\dotsm b_n$. In other words, we have an isomorphism of $(R,\Lambda)$-torsors
\begin{equation*}
  E_n(E_{m_1},\dots, E_{m_n})_{a_{1,1},\dots,a_{n,m_n}}  \lisoto (R, x_{b_1\dotsm b_n}).
\end{equation*}
On the other hand, since $E_{m_1+\dots+m_n}$ covers the multiplication $(a_{1,1},\dots,a_{n,m_n})\to a_{1,1}\dotsm a_{n,m_n}=b_1\dotsm b_n$, we have another isomorphism
\begin{equation*}
  \bigl(E_{m_1+\dots+m_n}\bigr)_{a_{1,1},\dots,a_{n,m_n}} \lisoto (R,x_{b_1\dotsm b_n})
\end{equation*}
determined by a chosen section of $E_{m_1+\dots+m_n}$.  Thus, the morphism~\eqref{eq:41} amounts to an automorphism
\begin{equation}
  \label{eq:56}
  f_{n;m_1,\dots,m_n}(a_{1,1},\dots,a_{n,m_n}) \colon (R,x_{b_1\dotsm b_n}) \lisoto (R,x_{b_1\dotsm b_n}),
\end{equation}
which, using standard facts about $(R,\Lambda)$-torsors, we identify with an element of $M$.

\subsection{Cocycle computations}
\label{sec:cohom-class}

We compute the full class determined by $\stR$ from the multi-extension structure of the presentation $R\to \Lambda\to \stR$ and prove that it satisfies the full set of equations~\eqref{eq:45}.

\subsubsection{Setup}

First, consider the following (not necessarily commutative) diagrams of $(R,\Lambda)$-torsor morphisms
\begin{subequations}
  \label{eq:57}
  \begin{gather}
    \vcenter{%
      \xymatrix{%
        E_{a,c}\wedge E_{b,c} \ar[r] \ar[d] & E_{x_a+x_b,x_c}\ar[r] & E_{a+b,c} \ar[d]\\
        (R,x_{ac})\wedge (R,x_{bc}) \ar@{=}[r]
        & (R,x_{ac}+x_{bc}) \ar[r]^{\sigma_{ac,bc}^{-1}} & (R,x_{(a+b)c})
      }} \\
    \intertext{and}
    \vcenter{%
      \xymatrix{%
        E_{a,b}\wedge E_{a,c} \ar[r] \ar[d] & E_{x_a,x_b+x_c}\ar[r] & E_{a,b+c} \ar[d]\\
        (R,x_{ab})\wedge (R,x_{ac}) \ar@{=}[r]
        & (R,x_{ab}+x_{ac}) \ar[r]^{\sigma_{ab,ac}^{-1}} & (R,x_{a(b+c)})
      }}
  \end{gather}
\end{subequations}
arising from the two partial product laws of $E=E_2$. The vertical arrows are the identification~\eqref{eq:55}. Both diagrams in~\eqref{eq:57} can be construed as defining an automorphism of their respective lower right corners, which can be identified with an element of $M$: let them be $-\alpha_1(a,b;c)$ and $\alpha_2(a;b,c)$ respectively.

On the one hand the top rows can be written as
\begin{equation}
  \label{eq:58}
  e_{a,c}\wedge e_{b,c} \mapsto e_{a+b,c}\, g_1(a,b;c), \qquad
  e_{a,b}\wedge e_{a,c} \mapsto e_{a,b+c}\, g_2(a;b;c),
\end{equation}
by invoking~\eqref{eq:12}.  On the other hand, following the bottom part, we have
\begin{equation*}
  e_{a,c}\wedge e_{b,c} \mapsto \sigma_{ac,bc}^{-1}, \qquad
  e_{a,b}\wedge e_{a,c} \mapsto \sigma_{ab,ac}^{-1}.
\end{equation*}
Comparing the two we get the relations
\begin{equation}
  \label{eq:59}
  -\alpha_1(a,b;c) =g_1(a,b;c) + \sigma_{ac,bc},\qquad
  \alpha_2(a;b,c) =g_2(a;b,c) + \sigma_{ab,ac}.
\end{equation}

\subsubsection{The relations~\eqref{eq:49}}
Consider the diagrams:
\begin{subequations}
  \label{eq:60}
  \begin{gather}
    \label{eq:61}
    \vcenter{\xymatrix@C+1pc{%
        E_{x_{a+b}+x_c,x_d} \ar[r]^{\sigma_{a+b,c}} & E_{a+b+c,d} \ar[dd]\\
        E_{x_a+x_b+x_c,x_d} \ar[u]^{\sigma_{a,b}^{x_c}} \ar[d]_{\sigma_{b,c}} \\
        E_{x_a+x_{b+c},x_d} \ar[r]_{\sigma_{a,b+c}} & E_{a+b+c,d}}}
    \qquad
    \vcenter{\xymatrix@C+1pc{%
        E_{x_a,x_{b+c}+x_d} \ar[r]^{\sigma_{b+c,d}} & E_{a,b+c+d} \ar[dd]\\
        E_{x_a,x_b+x_c+x_d} \ar[u]^{\sigma_{b,c}^{x_d}} \ar[d]_{\sigma_{c,d}} \\
        E_{x_a,x_b+x_{c+d}} \ar[r]_{\sigma_{b,c+d}} & E_{a,b+c+d}}} \\
    \intertext{and}
    \label{eq:62}
    \vcenter{\xymatrix{E_{x_a+x_b,x_c} \ar[r]^{\sigma_{a,b}} &  E_{a+b,c} \\
        E_{x_b+x_a,x_c} \ar[u]^{\eta^1} \ar[r]_{\sigma_{b,a}}  &  E_{a+b,c} \ar[u]}}
    \qquad
    \vcenter{\xymatrix{E_{x_a,x_b+x_c} \ar[r]^{\sigma_{b,c}} &  E_{a,b+c} \\
        E_{x_a,x_c+x_b} \ar[u]^{\eta^2} \ar[r]_{\sigma_{c,b}}  &  E_{a,b+c} \ar[u]}}
  \end{gather}
\end{subequations}
Both diagrams~\eqref{eq:60} are parts of more extended ones, giving rise to relations linking $f_+$ and $g_+$ to the other quantities comprising a $5$-tuple satisfying relations~\eqref{eq:45} as follows.

First observe that by applying Lemma~\ref{lem:8}, the right vertical give the automorphisms corresponding to $f_+(a,b,c)\,d$ (resp.\ $a\,f_+(b,c,d)$) for~\eqref{eq:61}, and $g_+(a,b)\, c$ (resp.\ $a\,g_+(b,c)$) for~\eqref{eq:62}.

Then from~\eqref{eq:57} and~\eqref{eq:61} form the obvious associativity diagrams for the morphisms $E_{a,d}\wedge E_{b,d}\wedge E_{c,d}\to E_{a+b+c,d}$ and $E_{a,b}\wedge E_{a,c}\wedge E_{a,d}\to E_{a,b+c+d}$. Using the cocycle decomposition~\eqref{eq:15} (and Lemma~\ref{lem:8} for the right vertical arrows of~\eqref{eq:60}) we arrive at:
\begin{equation}
  \label{eq:63}
  \begin{aligned}
    g_1(a+b,c;d) + g_1(a,b;d)^{x_{cd}} & = g_1(a,b+c;d) + g_1(b,c;d) +f_+(a,b,c)\,d,\\
    g_2(a;b+c,d) + g_2(a;b,c)^{x_{ad}} & = g_2(a;b,c+d) + g_2(a;c,d) +a\,f_+(b,c,d),
  \end{aligned}
\end{equation}
(The difference with equations~\eqref{eq:16} and~\eqref{eq:18} arises because the top rows of~\eqref{eq:57}, contrary to the actual partial multiplication morphisms, lack associativity.) Using~\eqref{eq:59},~\eqref{eq:63}, and the first of~\eqref{eq:53}, we obtain the first two of the cocycle relations~\eqref{eq:49}.

The commutativity diagrams obtained from~\eqref{eq:62} and~\eqref{eq:57} can be analyzed in an analogous manner, utilizing the second of~\eqref{eq:53}. However, we do not directly obtain the other two equations in the block~\eqref{eq:49}; instead, we arrive at their ``flipped'' counterpart~\eqref{eq:52}.

\subsubsection{The interchange relation~\eqref{eq:48}}

We check the compatibility law~\eqref{eq:5} after pulling back to $\Lambda_a\times\Lambda_b \times \Lambda_c\times \Lambda_d$.  Use the diagram
\begin{equation}
  \label{eq:64}
  \vcenter{\xymatrix@C-3pc@R-1pc{%
    {(E_{a,c}\wedge E_{b,c})\wedge (E_{a,d}\wedge E_{b,d})} \ar[rrrr] \ar[dd] \ar[dr] & & & &
    {(E_{a,c}\wedge E_{a,d})\wedge (E_{b,c}\wedge E_{b,d})} \ar[dl] \ar[dd] \\
    & E_{x_a+x_b,x_c}\wedge E_{x_a+x_b,x_d} \ar[dl]\ar[dr] & *+[Fo:magenta]{\text{I}} &
    E_{x_a,x_c+x_d}\wedge E_{x_b,x_c+x_d} \ar[dl]\ar[dr]\\
    E_{a+b,c}\wedge E_{a+b,d} \ar[dr] \ar@<-1ex>@/_/[dddrr] & *+[Fo:magenta]{\text{II}} &
    E_{x_a+x_b,x_c+x_d} \ar[dl]\ar[dr] & *+[Fo:magenta]{\text{II}} &
    E_{a,c+d}\wedge E_{b,c+d} \ar[dl] \ar@<+1ex>@/^/[dddll]\\
    & E_{x_{a+b},x_c+x_d} \ar[ddr] & *+[Fo:magenta]{\text{III}} & E_{x_a+x_b,x_{c+d}} \ar[ddl]\\ \\
    & & E_{a+b,c+d}
  }}
\end{equation}
where the triangles commute by definition of the morphisms determined by the top rows of~\eqref{eq:57}; the pentagon (marked I) commutes by the compatibility law; the squares II and the square III are obviously commutative by functoriality (an explicit calculation uses Proposition~\ref{prop:1} and the equations~\eqref{eq:24}).  Thus the diagram commutes, and we can use the cocycle equation~\eqref{eq:20} directly, written additively, which gives:
\begin{equation*}
  g_2(a+b;c,d) + g_1(a,b;c)^{x_{(a+b)d}} + g_1(a,b;d) =
  g_1(a,b;c+d) + g_2(a;c,d)^{x_{b(c+d)}} + g_2(b;c,d) -\braces{x_{bc}}{x_{ad}}^{x_{bd}}.
\end{equation*}
Inserting equations~\eqref{eq:59} and using~\eqref{eq:53} finally gives the third block~\eqref{eq:48} of the cocycle relation~\eqref{eq:45}.

\subsubsection{The relations~\eqref{eq:47} and~\eqref{eq:46}}
\label{sec:step-four}

We specialize the expression~\eqref{eq:56} for the morphism~\eqref{eq:41} to the tuples $(2;2,1)$ and $(2;1,2)$.  The morphism $\mu$ given by~\eqref{eq:73} corresponds to the element
\begin{equation*}
  f(a,b,c) = -f_{(2;1,2)}(a,b,c) +f_{(2;2,1)}(a,b,c)\in M.
\end{equation*}
More precisely, using the isomorphism with $(R,x_{abc})$ as a reference trivialization, we can identify $f(a,b,c)$ with an automorphism of $E_3$ pulled back to $\Lambda_a\times \Lambda_b\times \Lambda_c$, hence with a section of $M$ over it. Explicitly, locally on $\Lambda_a\times \Lambda_b\times \Lambda_c$, we have the morphism of $(R,\Lambda)$-torsors
\begin{equation}
  \label{eq:65}
  \mu_{a,b,c}\colon E_2(E_2,I)_{a,b,c} \lto E_2(I,E_2)_{a,b,c},
\end{equation}
which, using the composition of multi-extensions given in section~\ref{sec:comp-n-bfls}, we can write as
\begin{equation}
  \label{eq:66}
  \mu ( [e_{a,b},0_c,e_{ab,c}] ) =  [0_a,e_{b,c},e_{a,bc}] -f(a,b,c),
\end{equation}
where $[e_{a,b},e_{ab,c}]\in E_2(E_2,I)_{a,b,c}$ and $[e_{b,c},e_{a,bc}]\in E_2(I,E_2)_{a,b,c}$.
(Recall the additivity in the notation; we extend it to the action of $R$ on torsors. The brackets denote the class under the action of $R$.  The identity map is represented by the trivial butterfly, and here $0_a$ represents the unit section of the underlying trivial torsor $\Lambda\times R$ at the point $x_a$, say.) We have $\jmath ([e_{a,b},e_{ab,c}]) =\jmath([e_{b,c},e_{a,bc}]) = x_{abc}$, so we can see directly that $f(a,b,c)\in M$.

As an isomorphism of tri-extensions, $\mu$ is a homomorphism for each of the three partial laws. Writing these conditions for~\eqref{eq:65}, we must compute the maps along the following diagram
\begin{equation*}
  \xymatrix{%
    E_2(E_2,I)_{a,c,d} \wedge^R E_2(E_2,I)_{b,c,d} \ar[r]^(.6){\times_1} \ar[d]_{\mu\wedge\mu}&
    *+!<-15pt,0pt>{E_2(E_2,I)_{x_a+x_b,x_c,x_d}} \ar[r]^>>>>{\sigma_{a,b}} \ar[d]^\mu &
    *+!<-5pt,0pt>{E_2(E_2,I)_{a+b,c,d}} \ar[d]^\mu\\
    E_2(I,E_2)_{a,c,d} \wedge^R E_2(E_2,I)_{b,c,d} \ar[r]_(.6){\times_1}&
    *+!<-15pt,0pt>{E_2(I,E_2)_{x_a+x_b,x_c,x_d}} \ar[r]_>>>>{\sigma_{a,b}} &
    *+!<-5pt,0pt>{E_2(I,E_2)_{a+b,c,d}}
  }
\end{equation*}
and the other two expressing the compatibility (or lack thereof) with the second and third partial laws.  Because there are some new elements compared to the calculations which have appeared thus far, we sketch some of the details.  In particular, to compute the two horizontal maps in the second square above we need the explicit form of the trivializations $s_1$ for both tri-extensions. Similarly for the other two diagrams. According to the last part of section~\ref{sec:comp-n-bfls}, the form of $s^{2,1}_1$ for $E_2(E_2,I)$ is:
\begin{equation*}
  s^{2,1}_1(r,x_c,x_d) = [s_1(r,x_c),0_{x_d},1_{x_d}],\quad r=\sigma_{a,b},
\end{equation*}
where $0_{x_d}$ denotes the unit section of $\Lambda\times R$ computed at $x_d\in\Lambda$, and $1_{x_d}$ denotes the unit section of $E_2\rvert_{\{0\}\times \Lambda}\iso \Lambda\times R$ similarly computed at $x_d$. Thus, the map denoted by $\sigma_{a,b}$ in the top row is given by
\begin{equation*}
  e\longmapsto e\times_1 s_1(\sigma_{a,b},x_c,x_d).
\end{equation*}
Similarly, for the analogous map in the bottom row we must use the expression
\begin{equation*}
  s^{1,2}_1(r,x_c,x_d) = [0_{\del r},e_{c,d},s_1(r,x_c)],\quad r=\sigma_{a,b},
\end{equation*}
with a similar interpretation of the notation.

Thus, the upper path to the lower right corner gives
\begin{align*}
  \mu\bigl(
  [e_{a,c},e_{ac,d}] \times_1 [e_{b,c},e_{bc,d}] \times_1
  [s_1(\sigma_{a,b},x_c),0_{d},1_{x_d}] \bigr)
  &= \mu [e_{a,c}\times_1 e_{b,c}\times_1 s_1(\sigma_{a,b},x_c),0_d,e_{ac,d}\times_1 e_{bc,d}] \\
  &= \mu [e_{a+b,c} + g_1(a,b;c), 0_{d},e_{ac,d}\times_1 e_{bc,d}] \\
  &= \mu [e_{a+b,c}, 0_{d},e_{ac,d}\times_1 e_{bc,d}\times_1s_1(-g_1(a,b;c),x_d)] \\
  &= \mu [e_{a+b,c}, 0_{d},e_{ac,d}\times_1 e_{bc,d}\times_1s_1(\sigma_{ac,bc}+\alpha_1(a,b;c),x_d)] \\
  &= \mu [e_{a+b,c}, 0_{d},e_{ac+bc,d}] + g_1(ac,bc,d) -\alpha_1(a,b;c)d \\
  &=  [0_{a+b},e_{c,d},e_{a+b,cd}] + g_1(ac,bc;d) -f(a+b,c,d) -\alpha_1(a,b;c)d
\end{align*}
where in the next to last we have used~\eqref{eq:24}, the first of~\eqref{eq:58}, and Lemma~\ref{lem:8}; to obtain the last we have used~\eqref{eq:66}.

On the other hand, the lower path to the lower right corner gives, with similar calculations

\begin{multline*}
  \mu[e_{a,c},e_{ac,d}] \times_1 \mu[e_{b,c},e_{bc,d}] \times_1
  [0_{\del \sigma_{a,b}},e_{c,d},s_1(\sigma_{a,b},x_{cd})]
   = [0_{a+b}, e_{c,d}, e_{a+b,cd}] + g_1(a,b;cd) -f(a,c,d) -f(b,c,d).
\end{multline*}
Comparing the two expressions and using~\eqref{eq:59} yield the first equation of block~\eqref{eq:47}. The others are obtained via identical means.

The last equation~\eqref{eq:46} becomes now the easiest to obtain, as a condition satisfied by the automorphism $\mu$ upon considering the five possible pullbacks to the quadri-extension $E_4$, as per the pentagon diagram~\eqref{eq:74}. We leave the details to the reader.

\subsection{The class of a ring-like stack}
\label{sec:class-ring-like}

Assembling the steps in sect.~\ref{sec:cohom-class}, we have the following
\begin{proposition}
  \label{prop:16}
  Let $\stR$ be a (weakly) ring-like stack with $\pi_0(\stR)=A$ and $\pi_1(\stR)=M$.  A decomposition of $\stR$ determines a twisted cocycle $\xi=\xi_\stR$ satisfying~\eqref{eq:51} with the same $A$-bimodule structure. An equivalence $\stR\to \stR'$ gives rise to two twisted cocycles $\xi_\stR$ and $\xi_{\stR'}$ differing by a coboundary, hence equivalence classes are in one-to-one correspondence with elements of $\widetilde H^3_2(A,M)$.  
\end{proposition}
\begin{proof}
  The first statement is a consequence of the preceding calculations. The statement about the equivalence follows from the definition of morphism of pseudo-monoid  and the fact that the structure is preserved across pseudo-functors between multi-bicategories (cf.~\ref{def:12} and~\ref{def:13}), in particular between~\ref{prop:13} 

  Since the only difference between all these complexes occurs in degrees 3 and 4, the statements carry over to the hypercohomology situation.
\end{proof}
Let us use the notation $[\stR]$ for the class determined by $\stR$. Thus of $[\stR]=[\xi_\stR]$.  As we have seen, a consequence of Proposition~\ref{prop:15} and Lemma~\ref{lem:9} is that in the unital case the classes in $\widetilde H^3_2(A,M)$ lift to $\HML^3(A,M)$. 
\begin{corollary}
  \label{cor:2}
  Let $\stR$ be as above, with in addition a unit object for the exterior monoidal structure. Then $A$ is unital and there exists an equivalence $\stR\isoto \stR'$ where the underlying categorical group structure of $\stR'$ is braided symmetric. Hence $[\stR] = [\stR']\in \HML^3(A,M)$.\qed
\end{corollary}
We briefly address the question of recovering $\stR$ (up to equivalence) from $[\stR]$. Consider a class $[\xi]\in \widetilde H^3_2(A,M)$.  A portion of $\xi$ will represent a class in $H^4(K(A,2),M)$, possibly lifting into the stable range.  Let $\xi_+$ denote this projection.  Standard techniques \cite[\S 7.6--7]{MR95m:18006} allow to reconstruct a braided (possibly symmetric) stack $\stR=\stR_\xi$ from $\xi_+$, equipped with $\varpi\colon \stR\to A$ fitting into the standard short exact sequence $\stM\to \stR\to A$, with $\stM=\sttors (M)$.  Briefly, $\xi_+$ determines a 2-gerbe over a simplicial model of $K(A,2)$ or $K(A,3)$, suitably re-scaled so that the relevant class appears in degree three.  $\stR$ is obtained by gluing trivial gerbes with band $M\to 0$ over $A$ along $\xi_+$. (We must supplement the cocycles in loc.\ cit.\ with those parts pertaining to the braiding structure.) The class of $\stR$ is that of $\xi_+$, and therefore it is equipped with a decomposition~\eqref{eq:53}.
\begin{proposition}
  \label{prop:17}
  Let $\stR$, $\xi$, and $\xi_+$ be as above. Let $R\to \Lambda\to \stR$ be a presentation by a braided crossed module.  Then $R\to \Lambda$ carries a bi-extension whose class is $\xi$.
\end{proposition}
\begin{proof}[Proof (Sketch)]
  The main point is to reverse the computation of the cohomology class, using the equations for the cocycle $\xi$ (either~\eqref{eq:45} or their twisted form~\eqref{eq:51}) to obtain a well-defined biextension $E\to \Lambda\times \Lambda$ providing $(R,\Lambda)$ with a pseudo-monoid structure.

  Starting with the implementation of~\eqref{eq:54}, for all $a,b\in A$ define $E_{a,b}\to \Lambda_a\times \Lambda_b$ as $E_{a,b}=\Lambda_a\times \Lambda_b\times R$ equipped with $\jmath\colon E_{a,b}\to \Lambda$ given by $0_R\mapsto x_{ab}$. From~\eqref{eq:59}, given $\xi$ and the decomposition~\eqref{eq:53}, we compute the nonabelian cocycles $g_1$ and $g_2$ we can use to define partial laws
  \begin{equation*}
    \times_1\colon E_{a,c}\wedge^R E_{b,c}\lto E_{a+b,c}\,,\qquad
    \times_2\colon E_{a,b}\wedge^R E_{a,c}\lto E_{a,b+c}\,,
  \end{equation*}
  which are well defined by equations~\eqref{eq:63} and~\eqref{eq:64}.

  Note that, unless $a=0$, the pullback of $\del\colon R\to \Lambda$ along $\Lambda_a\to \Lambda$ is trivial, whereas it is isomorphic to $\Im\del$ for $a=0$. This ensures (again from equation~\eqref{eq:54} and choosing a normalized section $x$) the triviality conditions required by Definition~\ref{def:3} to have a full-fledged \bfl are satisfied.

  Finally, we use~\eqref{eq:65} and~\eqref{eq:66} to define $\mu\colon E_2(E_2,I)\to E_2(I,E_2)$; the identity
  \begin{equation*}
    g_1(ac,bc;d) -f(a+b,c,d) -\alpha_1(a,b;c)d = g_1(a,b;cd) -f(a,c,d) -f(b,c,d),
  \end{equation*}
 and its companions found above, together with~\eqref{eq:46} and~\eqref{eq:47}, ensure $E$ satisfies the required pentagonal structure.
\end{proof}

\appendix

\part{}
\section{Multi-variable compositions}
\label{sec:formalism}

In this technical addendum we give a brief treatment of multivariable functor calculus in a bicategory. Our approach is descriptive and explicit, and it is aimed at a definition of pseudo-monoid suitable for the applications in the text to multi-additive functors and multi-linear \bfls (cf.\ sects.~\ref{sec:auxiliary} and following). 

We resort to multi categorical-based ideas, in fact we borrow some of Kelly's clubs formalism \parencite[see][]{MR0340371,MR0364393}, which is convenient in the present context.\footnote{See also \cite{MR1337494} for further applications to (symmetric) monoidal categories.} We include permutations, even though this is slightly more general than needed in the main part of the text. Permutations can be included at nearly no additional cost, covering the symmetric monoidal case, which is what the formalism was originally designed to do.  As a result, the formalism can still be used to symmetrize the (external) monoidal structures described in the text. Unlike \cite{MR0340371}, we need our objects to inhabit a bicategory, as opposed to a 2-category, due to the fact that crossed modules equipped with \bfls as morphisms form a genuine bicategory equivalent to the 2-category of group-like stacks.

We define a \strong{bi-multicategory} $\bC$ the structure defined by the following data:\footnote{Without the extra structure given by permutations, the name ``bi-multicategory'' appears in \cite{pisani2014}, where it denotes the bicategory-analog of a multicategory: a $\Cat$-enriched multicategory with weakly associative composition. For enrichment over simplicial sets, see also \cite{2011arXiv1111.4146R}, which also incorporates permutations.}
\begin{itemize}
\item A class of objects $x,y,\dots$
\item For each tuple $(y_1,\dots,y_n;x)$ of objects, a groupoid of arrows $\grdHom_{\bC}(y_1,\dots,y_n;x)$. The cells are denoted
  \begin{equation*}
    \xymatrix@1{ **[l](y_1,\dots,y_n) \rtwocell^f_g{\;\alpha} & x}.
  \end{equation*}
\item For each object $x$, a functor $\imath_x\colon \mathbf{1}\to \grdHom_{\bC}(x;x)$, where $\mathbf{1}$ is the singleton category. The resulting distinguished object is the identity arrow $\id_x\colon (x) \to x$. 
\item Compositions functors
  \begin{gather*}
    \grdHom_{\bC}(y_1, \dots, y_n; x) \times \grdHom_{\bC}(z_{1,1},\dots,z_{1,m_1};y_1)\times \dots
    \times \grdHom_{\bC}(z_{n,1}\dots,z_{n,m_n};y_n) \lto \grdHom_{\bC}(z_{1,1},\dots,z_{n,m_n};x) \\
    (f; g_1,\dots,g_n) \lto f(g_1,\dots,g_n).
  \end{gather*}
\item Associativity data for the composition and the identities as in a bicategory.
\item For each $n$ and for each tuple $(y_1,\dots,y_n;x)$, an action by $\Sigma_n$, that is, a functor
  \begin{equation*}
    \xi^*\colon \grdHom_{\bC}(y_{\xi(1)},\dots,y_{\xi(n)};x) \lto \grdHom_{\bC} (y_1,\dots,y_n;x)
  \end{equation*}
  such that $(\xi\eta)^* \iso \eta^*\xi^*$ and the composition functors are equivariant for this action. Moreover, these isomorphisms are subject to appropriate coherence conditions.
\end{itemize}
The definition of pseudo- (or lax-)functor $F\colon \bC\to \bB$ between multi-bicategories is \emph{mutatis mutandis} the same as for bicategories.  For the symmetric structure, we add the condition that the functors
\begin{equation*}
  F_{y_1,\dots,y_n;x}\colon \grdHom_{\bC}(y_1,\dots,y_n;x)\to \grdHom_{\bB}(F(y_1),\dots,F(y_n);F(x)) 
\end{equation*}
preserve the $\Sigma_n$-action, for all $n$.

The last item in the data list intuitively affords for cells of the form
\begin{equation}
  \label{eq:67}
  \vcenter{%
    \xymatrix@R-1.5pc{%
      **[l](y_1,\dots,y_n) \ar[dd] \ar@/^/[dr] \ddrtwocell<\omit>{} \\
      & x\\
      **[l](y_{\xi(1)},\dots y_{\xi(n)}) \ar@/_/[ur] & {}
    }
  }
\end{equation}
Dropping all $\grdHom_{\bC}$-groupoids except those of the form $\grdHom_{\bC}(y;x)$, i.e.\ those of valence (or arity) one, we obtain an ordinary bicategory, the \emph{underlying bicategory} of $\bC$. Also, a \strong{2-multicategory} is a bi-multicategory in which all the associativity and identity data are strict. If all the groupoids $\grdHom_{\bC}(-;-)$ are discrete we obtain an ordinary multicategory.

The following is the analog of the ``generalized functor category'' in \textcites[\S 2]{MR0340371}. 

\begin{definition}
  \label{def:10}
  Let $x,y$ be two objects of $\bC$. Then $\lbrace y,x\rbrace$ denotes the category whose objects are pairs $(n,f)$, where $n$ is a natural number and $f$ is an object of $\grdHom_\bC(y,\dots,y;x)$ ($y$ is repeated $n$ times); morphisms $(n,f)\to (m,g)$ only exist if $n=m$ and consist of a permutation $\xi\in\Sigma_n$ and a morphism $\alpha\colon f \to \xi^*(g)$. The composition of morphisms is dictated by the composition in $\bC$. Explicitly, if $(\eta,\beta)\colon (n,g)\to (n,h)$, then $(\eta,\beta)\circ (\xi,\alpha) = (\eta\xi,\xi^*(\beta)\circ\alpha)$ modulo the (unnamed) coherence morphism $\xi^*\eta^* \iso (\eta\xi)^*$.
\end{definition}
\begin{remark}
  \label{rem:18}
  Let $\SS$ be the skeletal category of finite sets with permutations as morphisms. ($\SS$ can be obtained as the core, i.e.\ the largest subgroupoid, of the Segal category $\Gamma$ of finite sets.) Thus  an object of $\SS$ can be identified with a natural number. For any two objects $x,y$ of $\bC$ there results an obvious projection functor $p\colon \lbrace y,x\rbrace\to \SS$. In fact, $\lbrace y,x\rbrace$ can be obtained as the Grothendieck construction applied to the functor $F\colon \SS\to \Cat$ defined by $F(n) = \grdHom_{\bC}(y,\dots,y;x)$ and $F(\xi)=\xi^*$.
\end{remark}
For the following definition is the analog of the operations defined in \textcites[\S 2.1, \S 2.2]{MR0340371}. 
\begin{definition}
  \label{def:14}
  For all objects $x,y,z$ of $\bC$:
  \begin{enumerate}
  \item
    $\lbrace y,x\rbrace \wr \lbrace z,y\rbrace$ is the category whose objects are lists $(f;g_1,\dots,g_n)$, where $n=p(f)$; morphisms $(f;g_1,\dots,g_n)\to (f';g'_1,\dots,g'_n)$ (there are no morphisms if $p(f)\neq p(f')$) consist of a morphism $\alpha\colon f\to \xi^*(f')$ and for $1\leq i \leq n$, morphisms $\beta_i \colon g_{\xi(i)}\to \eta_{\xi(i)}^*(g'_i)$, where $\xi=p(f)\in \Sigma_n$, $p(g_i)$ $\eta_i=p(g_i)\in \Sigma_{n_i}$. Such a morphism is denoted $(\alpha; \beta_1,\dots,\beta_n)$.
  \item
    There is functor, called the \strong{assembly map}
    \begin{equation*}
      \medcirc \colon
      \lbrace y,x\rbrace \wr \lbrace z,y\rbrace\lto
      \lbrace z,x\rbrace,
    \end{equation*}
    sending the object $(f;g_1,\dots,g_n)$ to the composition $f(g_1,\dots,g_n)$ and the morphism $(\alpha; \beta_1,\dots,\beta_n)$ to the pasting of $\alpha$ and the $n$-tuple $(\eta_{\xi(1)},\dots,\eta_{\xi(n)})$, which is denoted $\alpha(\beta_1,\dots,\beta_n)$ for simplicity, following Kelly, again.
  \end{enumerate}
\end{definition}
Note that the details of the constructions stated in a direct fashion in Definitions~\ref{def:10} and~\ref{def:14} can be deduced from the combinatorics of the composition in the bi-multicategory $\bC$. Similarly, the associativity for the composition given by the assembly map is also inherited from the ambient bi-multicategory, hence, for objects $x,y,z,w$ of $\bC$ we have:
\begin{equation}
  \label{eq:68}
  \vcenter{\xymatrix{%
    \bigl(\lbrace y,x\rbrace \wr \lbrace z,y\rbrace\bigr) \wr \lbrace w,z\rbrace
    \ar[rr]^{\mathit{ass.}} \ar[d]_{\medcirc\wr\id}  &&
    \lbrace y,x\rbrace \wr \bigl(\lbrace z,y\rbrace \wr \lbrace w,z\rbrace\bigr)
    \ar[d]^{\id\wr\medcirc} \\
    \lbrace z,x\rbrace \wr \lbrace w,z\rbrace \ar[r] &
    \lbrace w,x\rbrace \ultwocell<\omit>  &
    \lbrace y,x\rbrace \wr \lbrace w,y\rbrace \ar[l] 
  }}
\end{equation}

\begin{definition}[\protect{%
    \textcites{MR0340371,MR0364393}}]
  \label{def:9}
  A \strong{Club} is a bi-multicategory with one object.
\end{definition}
Given a club with unique object $*$, set $\cat{k} = \lbrace *,*\rbrace$. By abuse of notation we refer to $\cat{k}$ itself as a club. Then the assembly map and~\eqref{eq:68} give a monoidal structure $\medcirc\colon \cat{k}\wr \cat{k}\to \cat{k}$.

Examples of clubs are:
\begin{itemize}
\item $\SS$ itself, with projection $\id\colon \SS\to \SS$.
\item The natural numbers identified with a discrete subcategory $\NN$ of $\SS$, with projection $\iota\colon \NN\hookrightarrow \SS$.
\end{itemize}
In both cases $\cat{k}=\SS$ and $\cat{k}=\NN$, the composition morphism $\cat{k}\wr \cat{k}\to \cat{k}$ sends the object $(n;m_1,\dots,m_n)$ to $m_1+\dots+m_n$. In the case of $\NN$, however, we dispense with the symmetric group action. 
\begin{remark}[Notation]
  \label{rem:15}
  \textcites{MR0340371} uses the notations $\medcirc$ (resp.\ $\mu$) where we use $\wr$ (resp.\ $\medcirc$). The symbol $\wr$ usually denotes the so-called wreath product.
  
  The wreath product $F\wr\cat{D}$ (or simply $\cat{C}\wr \cat{D}$) of two categories $\cat{C}$ and $\cat{D}$, the first equipped with a functor $F\colon \cat{C}\to \Gamma$,\footnote{$\Gamma$ is the Segal category of finite sets.} is defined in a way somewhat similar to Definition~\ref{def:14} \cite[see e.g.{}][]{Berger2007230}. In particular, objects are tuples of the form $(c;d_1,\dots,d_n)$, where $F(c)$ has cardinality $n$ and $d_1,\dots,d_n$ are objects of $\cat{D}$. Morphisms, however are slightly different.
\end{remark}

\section{(Unbiased) Monoids and monoidal structures}
\label{sec:formalism-monoid}

We keep the notations and environment of section~\ref{sec:formalism}. $\bC$ is a bicategory identified with the underlying bicategory of a bi-multicategory $\M\bC$.  We add another piece of notation: for an object $x$, let $\M x=\lbrace x,x\rbrace$. Having observed that this is a monoid for the composition, a monoid object in $\bC$ is defined in the expected way. Recall that $\NN$ is a club, as discussed in the previous section. Recall that a lax monoidal functor is one that respects the monoidal structures up to coherent natural transformation.\footnote{The distinction between lax and op-lax, i.e.\ the direction of the 2-arrows in the natural transformations is immaterial as we work with bicategories whose 2-arrows are isomorphisms.}
\begin{definition}
  \label{def:11}
  A \strong{(pseudo-)monoid} in $\bC$ is an object $x$ of $\M\bC$ equipped with a lax monoidal functor $(t,\theta) \colon \NN\to \M x$. A \emph{symmetric} (pseudo-)monoid would is the same thing with the club $\SS$, instead.
\end{definition}
In the next diagrams we spell out the conditions for $(t,\theta)$ to be a lax monoidal functor. Explicitly, the lax monoidal functor is given by a diagram
\begin{equation}
  \label{eq:69}
  \vcenter{\xymatrix{%
      \NN\wr\NN \ar[r]^\medcirc \ar[d]_{t\wr t} \drtwocell<\omit>{\theta}
      & \NN \ar[d]^t \\
      \M x\wr \M x \ar[r]_\medcirc & \M x
    }}
\end{equation}
From it we have the coordinate expression for the natural transformation $\theta$, which assigns to the object $(n;m_1,\dots,m_n)\in \NN\wr\NN$ the morphism
\begin{equation}
  \label{eq:70}
  \theta_{n;m_1,\dots,m_n}\colon t(m_1+\dots +m_k) \lto t(n)( t(m_1),\dots, t(m_n) )
\end{equation}
in $\M x$.  These data are subject to be compatible with the associativity conditions for the assembly maps of both $\NN$ and $\M x$, namely they must satisfy the following commutative diagram of natural transformations:
\begin{equation}
  \label{eq:71}
  \vcenter{\xymatrix{%
      (\M x\wr \M x)\wr \M x \ar[rrrr]^{\mathit{ass}}
      \ar[ddd]_{\medcirc\wr \Id}
      &&&& \M x\wr (\M x\wr \M x) \ar[ddd]^{\Id\wr\medcirc} \\
      & (\NN \wr \NN)\wr \NN \ddltwocell<\omit>{\theta\wr t}
      \ar[ul]^{(t\wr t)\wr t}
      \ar[d]_{\medcirc\wr\Id} \ar[rr]^{\mathit{ass}} &&
      \NN\wr (\NN \wr \NN) \ar[ur]_{t\wr (t\wr t)} 
      \ar[d]^{\Id\wr \medcirc}\\
      & \NN\wr \NN \ar[r]_\medcirc \drtwocell<\omit>{\theta}
      \ar[dl]^{t\wr t} &
      \NN \ar[d]^t &
      \NN\wr\NN \ar[dr]_{t\wr t} \ar[l]^\medcirc \\
      \M x \wr \M x \ar[rr]_\medcirc  & & \M x \urtwocell<\omit>{\;\theta} & &
      \M x\wr \M x \ar[ll]^\medcirc \uultwocell<\omit>{t\wr\theta}
    }}
\end{equation}
In~\eqref{eq:71} the top quadrangle and the small pentagon are strictly commutative. For the quadrangle, it follows from the functoriality of the associator, whereas for the small pentagon the associativity morphism reads
\begin{equation*}
  ((n;m_1,\dots,m_n);l_1,\dots,l_k) \longmapsto (n;(m_1;l_1,\dotsc),\dots,(m_n;\dotsc,l_k)),
\end{equation*}
where $k=m_1+\dots +m_n$;  both paths evaluate to $h=l_1+\dots +l_k = j_1 + \dots + j_n$, where for $i=1,\dots,n$ we have $j_i = l_{m_{i-1}+1} + \dots + l_{m_{i-1}+m_i}$. Lastly, the back face of the diagram, that is the large pentagon, is just~\eqref{eq:68} specialized to the same object.  Thus, the commutativity in diagram~\eqref{eq:71}, when expressed in coordinates, reads as follows:
\begin{equation}
  \label{eq:72}
  \vcenter{\xymatrix{%
      t(k) (t(l_1),\dots,t(l_k)) \ar[d] & t(h) \ar[r] \ar[l] & t(n)(t(j_1),\dots,t(j_n)) \ar[d] \\
      t(n) (t(m_1),\dots,t(m_n)) (t(l_1),\dots,t(l_k)) \ar[rr] && t(n)(t(m_1)(t(l_1),\dotsc),\dots,t(m_n)(\dots,t(l_k))
    }}
\end{equation}
\begin{remark}
  \label{rem:16}
  By construction, the pseudo-monoid structure on $x$ given by the club morphism $t\colon \NN\to \M x$ possesses $n$-ary operations $t(n)$ for arbitrary values of $n$. If $x$ has an internal structure (namely objects, as in the main text), then this implements a monoidal structure in ``unbiased'' form \cite[see e.g.][]{MR2094071}, with the $n$-ary unprivileged operations of all degrees. In this case diagram~\eqref{eq:71} expresses the coherence condition in unbiased form, namely, that the two possible ways to remove parentheses from an expression must coincide as it is apparent from the coordinate version~\eqref{eq:72}.
\end{remark}
\begin{remark}
  \label{rem:19}
  Consider the singletons $\lbrace x \rbrace$ as a copy of the singleton category $\one$. Any category with products gives rise to a multicategory in a standard way \cite[Example 2.16]{MR2094071}. Applied to $\lbrace x\rbrace$, there is singleton worth of morphisms from $x$ to $x$ of arbitrary arity. The definition of the $\wr$-operation introduced in Appendix~\ref{sec:formalism} can be copied in order to define a category $\NN\wr x$ (dropping the braces from the notation for convenience). One could show (for example, by formally checking the adjunction properties of $\wr$ and $\lbrace -,-\rbrace$ as in\cite{MR0340371}) that $(t,\theta)$ in Definition~\ref{def:11} correspond to a functor
  \begin{equation*}
    \tilde t \colon \NN\wr x \lto x
  \end{equation*}
  and a natural transformation
  \begin{equation*}
    \xymatrix{%
      (\NN\wr\NN)\wr x \ar[rr] \ar[d]_{\medcirc \wr x}
      \drtwocell<\omit>{\tilde\theta}
      &&  \NN\wr (\NN\wr x) \ar[d]^{\NN\wr \tilde t} \\
      \NN\wr x \ar[r]_{\tilde t} & x 
      & \NN\wr x \ar[l]^{\tilde t}
    }
  \end{equation*}
  By definition, $(\tilde t,\tilde \theta)$ equip $x$ with a structure of $\NN$-\strong{(pseudo-)algebra}. We will not work with this variant.
\end{remark}
One can consider morphisms of monoids as follows. Denote $\NN\wr x$ (cf.\ Remark~\ref{rem:19}) by $\cat{T}x$. Let $f\colon y\to x$ be a (unary) arrow. Then $f$ determines a functor $\cat{T}f\colon \cat{T}y\to \cat{T}x$ and two functors
\begin{equation*}
  f_* \lbrace y,y\rbrace \lto \lbrace y,x \rbrace \qquad
  \cat{T}f^* \lbrace x,x\rbrace \lto \lbrace y,x\rbrace
\end{equation*}
by post-composing with $f$ or pre-composing with $\cat{T}f$, respectively.
\begin{definition}
  \label{def:12}
  Let $x,y$ be (pseudo-)monoids in $\bC$. A \strong{morphism of monoids} is a pair $(f\colon y\to x,\lambda)$ fitting the diagram
  \begin{equation*}
    \xymatrix{%
      \NN \ar[r]^{t_y} \ar[d]_{t_x} \drtwocell<\omit>{\lambda}
      & \lbrace y,y\rbrace  \ar[d]^{f_*} \\
      \lbrace x,x\rbrace \ar[r]_{\cat{T}f^*} & \lbrace y,x\rbrace
    }
  \end{equation*}
  and compatible with the pseudo-algebra condition~\eqref{eq:69}.
\end{definition}
The compatibility between the above diagram and~\eqref{eq:69} means that in the coordinates of~\eqref{eq:70} we have
\begin{equation*}
  \xymatrix{%
    f\circ t_y(m_1+\dots +m_k) \ar[r]^{f\theta_y} \ar[d]_{\lambda_{m_1+\dots+m_n}} &
    f\circ ( t_y(n)( t_y(m_1),\dots, t_y(m_n) )) \ar[r]^*+[Fo]{\scriptstyle 1}&
    (f\circ  t_y(n))( t_y(m_1),\dots, t_y(m_n) ) \ar[d]^{\lambda_n\circ\Id} \\
    t_x(m_1+\dots +m_k) \circ \cat{T}f \ar[d]_{\theta_x\cat{T}f} & &
    (t_x(n)\circ \cat{T}f)( t_y(m_1),\dots, t_y(m_n) )  \ar[d]^*+[Fo]{\scriptstyle 2} \\
    (t_x(n)( t_x(m_1),\dots, t_x(m_n) )) \circ \cat{T}f \ar[r]_*+[Fo]{\scriptstyle 3} &
    t_x(n)(( t_x(m_1),\dots, t_x(m_n) ) \circ \cat{T}f) &
    t_x(n)(f\circ t_y(m_1),\dots, f\circ t_y(m_n) )
    \ar[l]^(0,45)*+{\scriptstyle{\Id\circ\lambda_{m_1,\dots,m_n}}}
    }
\end{equation*}
where the numbered arrows result from associativity isomorphism for composition.

Pseudo-monoids behave in the expected manner with respect to pseudo-functors. Specifically, let $F\colon \M\bC\to \M\bB$ be a pseudo-functor, which by virtue of our identification, we think of as coming from an pseudo-functor $F\colon \bC\to \bB$ between bicategories. It is clear that for any objects $x,y$ of $\bC$, $F$ induces a functor
\begin{equation*}
  F_{(y);x} \colon \lbrace y,x\rbrace \lto \lbrace F(y),F(x)\rbrace.
\end{equation*}
\begin{definition}
  \label{def:13}
  Let $x$ be a pseudo-monoid in $\bC$. Then $F(x)$ acquires the structure of a pseudo-monoid in $\bB$ by virtue of the composition:
  \begin{equation*}
    \xymatrix@1{\NN \ar[r]^t & \M x \ar[r]^{F_\bullet}& \M F(x),}
  \end{equation*}
  with the full pseudo-monoid structure for $\M F(x)$ results from the composition
\begin{equation*}
  \xymatrix{%
    \NN\wr \NN \ar[r]^{t\wr t} \ar[d]_\medcirc &
    \M x\wr \M x \ar[r]^(0.4){F_\bullet \wr F_\bullet} \ar[d]^\medcirc &
    \M F(x) \wr \M F(x) \ar[d]^\medcirc \\
    \NN \ar[r]^t & \M x \ar[r]^{F_\bullet} \ultwocell<\omit>{\theta} & \M F(x)
    \ultwocell<\omit>{\epsilon}
  }
\end{equation*}
\end{definition}

\section{Pentagons}
\label{sec:pentagons}

Pentagon diagrams express the coherence condition in a monoidal category. This condition is replaced by a diagram of the form~\eqref{eq:71} or~\eqref{eq:72} if the monoidal structure is given in unbiased form \cite{MR2094071}.  The actual pentagon can be recovered if these diagrams are specialized to arities equal to $4$ and the monoidal category comes from a monoid inhabiting a 2-category. This shows the equivalence of the biased and unbiased definitions.  A slight generalization of the pentagon occurs if the monoid inhabits a bicategory, and both these situations arise in the main text. In addition,  conditions arising from pentagons are cocycle conditions in appropriate cohomology theories.  It is useful to compute them once and for all in the general setting of a monoid object in a bicategory.

To begin with, observe that by applying~\eqref{eq:69} and \eqref{eq:70} to the objects $(2;2,1)$ and $(2;1,2)\in \NN\wr \NN$ we obtain morphisms $\theta_{2;2,1}\colon t(3)\to t(2)(t(2),t(1))$, $\theta_{2;1,2}\colon t(3)\to t(2)(t(1),t(2))$, and combining these two we obtain
\begin{equation}
  \label{eq:73}
  \mu \colon t(2)(t(2),t(1)) \lto t(2)(t(1),t(2)).
\end{equation}
The pentagons (or the diagrams related to them) arise from the decomposition of the operation $t(4)$, by way of~\eqref{eq:72}, down to terms only involving the binary and unary operations $t(2)$ and $t(1)$, respectively. We can assume that the unary operation $t(1)\colon x\to x$ coincides with the identity $\id_x$. These decompositions can be encoded by working our way along the small pentagon in the diagram~\eqref{eq:71}, so for instance one of them corresponds to the sequences
\begin{equation*}
  \xymatrix{%
    ((2;2,1);2,1,1)  \ar[rr] \ar[d] & & (2; (2;2,1), 1) \ar[d] \\
    (3; 2, 1, 1) \ar[r] & 4 & (2; 3, 1) \ar[l]
  }
\end{equation*}
(We have simply written $1$ in place of the more accurate but cumbersome expression $(1;1)$.)

There are six distinct sequences including the one above. Their starting points, counted from the upper left corner of the small pentagon in the diagram~\eqref{eq:71}, are:
\begin{enumerate}
\item \label{item:8}  $((2;2,1);2,1,1)$,
\item \label{item:9}  $((2;2,1);1,2,1)$,
\item \label{item:10} $((2;1,2);1,2,1)$,
\item \label{item:11} $((2;1,2);1,1,2)$,
\item \label{item:12} $((2;1,2);2,1,1)$,
\item \label{item:13} $((2;2,1);1,1,2)$.
\end{enumerate}
Let us also use the notations $m_i = t(i)$, $i=1,\dots,4$, with the special provision $m_1=\id_x=\id$.  With these, applying~\eqref{eq:72} and~\eqref{eq:71}, we obtain the diagram on page~\pageref{eq:74}:
\begin{landscape}
  \begin{equation}
    \label{eq:74}
    \vcenter{%
  \xymatrix@C-2pc@R+1pc{%
    & & m_2(\id(m_2),m_2(\id,\id)) \ar[rr] & & m_2(\id,m_2)(m_2,\id,\id) \\
    & m_2(m_2(\id,\id),\id(m2)) \ar[ddl] & & *+[Fo]{\ref{item:12}} & &  m_2(m_2,\id)(m_2,\id,\id) \ar@{.>}@[magenta][ul]_\mu \\
    & & m_2(m_2,m_2) \ar[uu] \ar[ul] \ar@{.>}@[blue][uurr] \ar@{.>}@[blue][dll] & & m_3(m_2,\id,\id) \ar@{.>}@[red][ll] \ar[uu] \ar[ur] \\
    m_2(m_2,\id)(\id,\id,m_2) \ar@{.>}@[magenta][dd]_\mu & *+[Fo]{\ref{item:13}} & & & & *+[Fo]{\ref{item:8}} & m_2(m_2(m_2,\id),\id) \ar[uul] \ar@{.>}@[magenta][dd]^\mu \\
    & m_3(\id,\id,m_2) \ar[dl] \ar[ul] \ar@{.>}@[red][uur] & & m_4 \ar[uul] \ar[uur] \ar[rr] \ar[ll] \ar[ddl] \ar[ddr] & & m_2(m_3,\id) \ar[ur] \ar[dr] \\
    m_2(\id, m_2)(\id,\id,m_2) & *+[Fo]{\ref{item:11}} & & & & *+[Fo]{\ref{item:9}} & m_2(m_2(\id,m_2),\id) \ar[ddl] \\
    & & m_2(\id,m_3) \ar[dd] \ar[dl] & & m_3(\id,m_2,\id) \ar[dd] \ar[dr] \\
    & m_2(\id,m_2(\id,m_2)) \ar[uul] & & *+[Fo]{\ref{item:10}} & &  m_2(m_2,\id)(\id,m_2,\id) \ar@{.>}@[magenta][dl]^\mu \\
    & & m_2(\id,m_2(m_2,\id)) \ar[rr] \ar@{.>}@[magenta][ul]^\mu  & & m_2(\id,m_2)(\id,m_2,\id)
  }}
\end{equation}
\end{landscape}
The six petals correspond to the indicated sequences and are numbered accordingly. We have marked the (magenta) arrows resulting from the morphism~\eqref{eq:73}.  Note that there are five of them.  In particular:
\begin{itemize}
\item If $\M\bC$ is \emph{unital,} the petals~\ref{item:12} and~\ref{item:13} collapse to the red arrows.
\item If furthermore $\M\bC$ is in fact a 2-multicategory, i.e.\ the 2-categorical analog of a multicategory, so that $\bC$ itself is a genuine 2-category, the associativity morphisms at the tops of all petals reduce to identities. As a result, the outer perimeter reduces to a standard pentagon
\end{itemize}
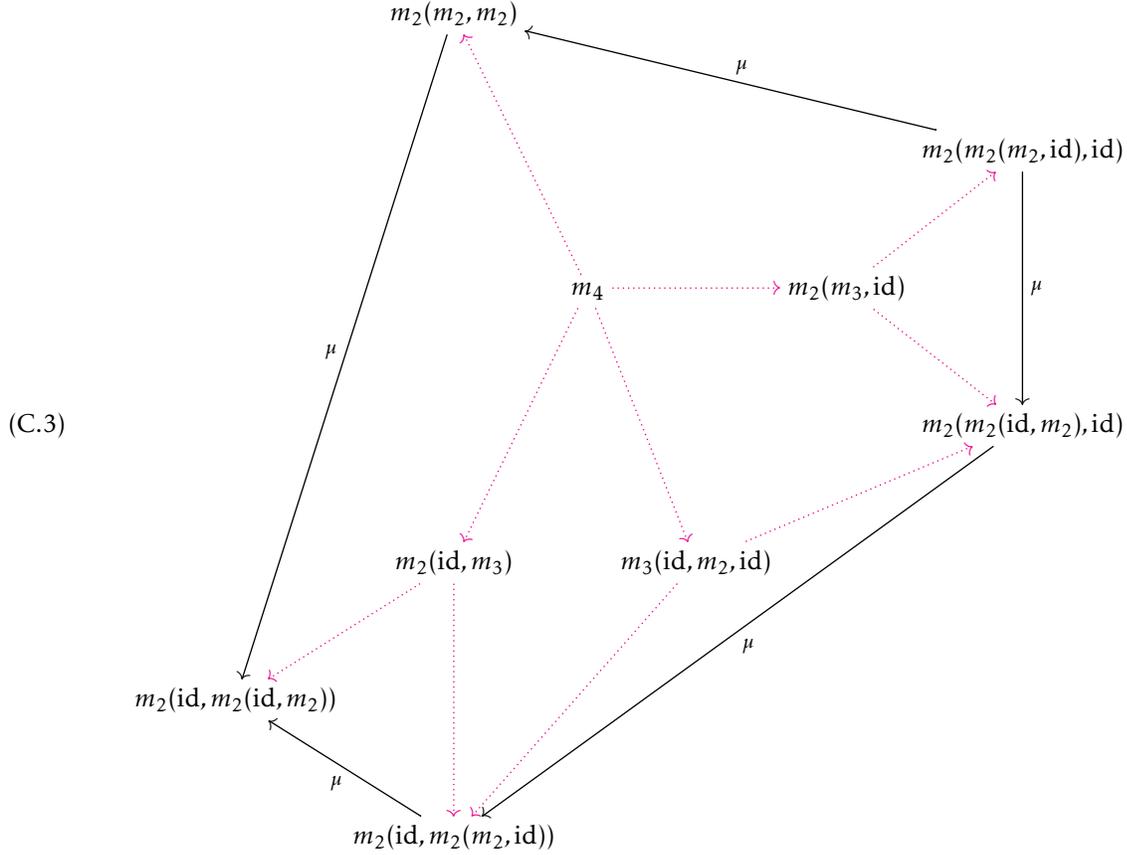
\begin{equation}
  \label{eq:75}
  \vcenter{%
  \xymatrix@C-2pc@R+1pc{%
    & m_2(m_2,m_2) \ar[dddddl]_\mu\\
    & & & & & m_2(m_2(m_2,\id),\id) \ar[dd]^\mu  \ar[llllu]_\mu \\
    & & m_4 \ar@{.>}@[magenta][uul] \ar@{.>}@[magenta][rr]  \ar@{.>}@[magenta][ddl] \ar@{.>}@[magenta][ddr] & & m_2(m_3,\id) \ar@{.>}@[magenta][ur] \ar@{.>}@[magenta][dr] \\
    & & & & & m_2(m_2(\id,m_2),\id) \ar[dddllll]^\mu\\
    & m_2(\id,m_3) \ar@{.>}@[magenta][dd] \ar@{.>}@[magenta][dl] & & m_3(\id,m_2,\id) \ar@{.>}@[magenta][ddll] \ar@{.>}@[magenta][urr] \\
    m_2(\id,m_2(\id,m_2)) \\ 
    & m_2(\id,m_2(m_2,\id)) \ar[ul]^\mu 
  }}
\end{equation}

\section{Some lemmas on symmetric braidings}
\label{sec:lemmas}

\begin{lemma}
  \label{lem:10}
  Let $(\cat{C},\otimes,c)$ be a strictly associative braided monoidal category. For every $x,y,z,w$, the diagram
  \begin{equation*}
    \xymatrix@C+1.5pc{%
      x\otimes y\otimes z\otimes w
      \ar[r]^{1\otimes c_{y,z}\otimes 1}
      \ar[d]_{c_{x,y}\otimes c_{z,w}} &
      x\otimes z\otimes y\otimes w
      \ar[d]^{c_{x\otimes z,y\otimes w}} \\
      y\otimes x\otimes w\otimes z
      \ar[r]_{1\otimes c_{x,w}\otimes 1} &
      y\otimes w\otimes x\otimes z
    }
  \end{equation*}
  commutes if and only if the braiding $c$ is symmetric.
\end{lemma}
\begin{proof}
  The naturally occurring diagram relating $c_{x\otimes z,y\otimes w}$ to $c_{x,y}\otimes c_{z,w}$ is
  \begin{equation*}
    \xymatrix@C+1.5pc{%
      x\otimes y\otimes z\otimes w
      \ar@{<-}[r]^{1\otimes c_{z,y}\otimes 1}
      \ar[d]_{c_{x,y}\otimes c_{z,w}} &
      x\otimes z\otimes y\otimes w
      \ar[d]^{c_{x\otimes z,y\otimes w}} \\
      y\otimes x\otimes w\otimes z
      \ar[r]_{1\otimes c_{x,w}\otimes 1} &
      y\otimes w\otimes x\otimes z
    }
  \end{equation*}
  which is valid with respect to any braiding. This is equal to the diagram in the statement if and only if $c_{y,z}$ is the inverse of $c_{z,y}$.
\end{proof}
\begin{lemma}
  \label{lem:11}
  Let $(\cat{C},\otimes,c)$ be a strictly associative braided monoidal category. If the braiding $c$ is symmetric, the following diagram
  \begin{equation*}
    \xymatrix@C+1.5pc{%
      x\otimes  y\otimes  z\otimes  u\otimes  v\otimes  w
      \ar[rr]^{1\otimes c_{y  z,u  v}\otimes 1}
      \ar[d]_{c_{x,y}\otimes 1\otimes 1\otimes c_{v,w}}
      & & x\otimes  u\otimes  v\otimes  y\otimes  z\otimes  w
      \ar[d]^{c_{x.u}\otimes 1\otimes 1\otimes c_{z,w}}\\
      y\otimes x\otimes  z\otimes  u\otimes  w\otimes  v
      \ar[d]_{1\otimes c_{x  z,u  w}\otimes 1}
      & &
      u\otimes  x\otimes  v\otimes  y\otimes  w\otimes  z
      \ar[d]^{1\otimes c_{x  v,y  w}\otimes 1}\\
      y\otimes  u\otimes  w\otimes  x\otimes  z\otimes  v
      \ar[rr]_{c_{y.u}\otimes 1\otimes 1\otimes c_{z,v}}
      & &
      u\otimes  y\otimes  w\otimes  x\otimes  v\otimes  z
    }
  \end{equation*}
  commutes.
\end{lemma}
\begin{proof}
  Use Lemma~\ref{lem:10} to replace $c_{yz,uv}$ with $c_{y,u}\otimes c_{z,v}$ and similarly for $c_{xz,uw}$ and $c_{xv,yw}$. This converts the diagram in the statement into the following one
  \begin{equation*}
    \xymatrix@C+1.5pc{%
      x\otimes  y\otimes  u\otimes  z\otimes  v\otimes  w
      \ar[rr]^{1\otimes c_{y,u}\otimes c_{z,v}\otimes 1}
      \ar[d]_{c_{x,y}\otimes 1\otimes 1\otimes c_{v,w}}
      & & x\otimes  u\otimes  y\otimes  v\otimes  z\otimes  w
      \ar[d]^{c_{x.u}\otimes 1\otimes 1\otimes c_{z,w}}\\
      y\otimes  x\otimes  u\otimes  z\otimes  w\otimes  v
      \ar[d]_{1\otimes c_{x,u}\otimes c_{z,w}\otimes 1}
      & &
      u\otimes  x\otimes  y\otimes  v\otimes  w\otimes  z
      \ar[d]^{1\otimes c_{x,y}\otimes c_{v,w}\otimes 1}\\
      y\otimes  u\otimes  x\otimes  w\otimes  z\otimes  v
      \ar[rr]_{c_{y.u}\otimes 1\otimes 1\otimes c_{z,v}}
      & &
      u\otimes  y\otimes  x\otimes  w\otimes  v\otimes  z
    }
  \end{equation*}
  which consists of two juxtaposed copies of the same kind of hexagon:
  \begin{equation*}
    \xymatrix{%
      x\otimes  y\otimes  u
      \ar[rr]^{1\otimes c_{y,u}}
      \ar[d]_{c_{x,y}\otimes 1}
      & & x\otimes  u\otimes  y
      \ar[d]^{c_{x.u}\otimes 1} \\
      y\otimes  x\otimes  u
      \ar[d]_{1\otimes c_{x,u}}
      & &
      u\otimes  x\otimes  y
      \ar[d]^{1\otimes c_{x,y}}\\
      y\otimes  u\otimes  x
      \ar[rr]_{c_{y.u}\otimes 1}
      & &
      u\otimes  y\otimes  x
    }
  \end{equation*}
  The latter commutes in any braided monoidal (strictly associative) category.
\end{proof}
A \strong{permutative} category is a symmetric monoidal category $(\cat{C},\otimes,c)$ in which associativity and unitality hold strictly.
 Thus the diagrams in Lemmas~\ref{lem:10} and \ref{lem:11} commute if and only if the category $(\cat{C},\otimes,c)$ is permutative. On the other hand, the statements of Lemma~\ref{lem:10} and Lemma~\ref{lem:11} hold in general. Requiring that $\cat{C}$ be strictly associative merely simplifies the diagrams and overall legibility. For example, in the general case the diagram in the statement of Lemma~\ref{lem:10} would need replacing with:
  \begin{equation*}
    \xymatrix@C+1.5pc{%
      x\otimes (z\otimes (y\otimes w)) &
      \ar[l]
      (x\otimes z)\otimes (y\otimes w)
      \ar[r]^{c_{x\otimes z,y\otimes w}} &
      (y\otimes w)\otimes (x\otimes z)
      \ar[r] &
      y\otimes (w\otimes (x\otimes z))\\
      x\otimes ((z\otimes y)\otimes w) \ar[u] &&&
      y\otimes ((w\otimes x)\otimes z) \ar[u] \\
      x\otimes ((y\otimes z)\otimes w)
      \ar[u]^{1\otimes c_{y,z}\otimes 1} \ar[d] &&&
      y\otimes ((x\otimes w)\otimes z)
      \ar[u]_{1\otimes c_{x,w}\otimes 1} \ar[d] \\
      x\otimes (y\otimes (z\otimes w)) &
      \ar[l]
      (x\otimes y)\otimes (z\otimes w) 
      \ar[r]_{c_{x,y}\otimes c_{z,w}}&
      (y\otimes x)\otimes (w\otimes z)
      \ar[r] &
      y\otimes (x\otimes (w\otimes z))           
    }
  \end{equation*}
  with a similarly modified proof. A similar, but more complicated modifications apply to the diagram in the proof of Lemma~\ref{lem:11}.

\section{Hypercohomology computations}
\label{sec:hyper}

In order to carry out a full, explicit cohomology computation for the class of a ring-like stack $\stR$, it is necessary to resolve various simplicial objects related to the ring $A$, notably the Eilenberg-Mac~Lane's $K(A,n)$ for $n=2,3$, by way of certain bisimplicial objects.

In the main text, and specifically in sect.~\ref{sec:decomp}, 
we simply referred to the spectral sequence for the hypercohomology of the simplicial object $K(A,1)$\footnote{As done in ref. \cite[\S 6]{MR1702420}, discussing the passage from a monoidal category to a monoidal stack} and focus on its $E^{0,3}$-term, showing that it is isomorphic to $\HML^3(A,M)$ proper. While this is appropriate to show the agreement with Mac~Lane cohomology, it is important to provide the main outline of the full calculation.

It is easier to illustrate the full computation in the (known) cases of $H^3(K(A,1),M)$, or $H^4(K(A,2),M)$ first, that is, by just considering the underlying structure of braided (symmetric) stack of $\stR$. The full case just calls for the addition of more data following the same pattern.

As it is well known, $\stR$ is only a gerbe over $A$ \cite[see e.g.\ ][]{MR1771927}, and  the decomposition of $\varpi\colon \stR\to A$ requires choosing local objects. Thus we need to choose covers $\xi_n\colon U_{n,0}\to A^n$, which we complete to an hypercover of $K(A,1)$. Following ref.\ \cite{MR676809}, this means an augmented bisimplicial object
  \begin{equation*}
    U_{\bullet,\bullet} \lto K(A,1)_\bullet,
  \end{equation*}
  such that each $U_{n,\bullet}\to A^n$ is a hypercover. (In fact, and it is convenient to do so, $U$ can be arranged so that its horizontal zero level $U_{\bullet,0}\to K(A,1)$ is an acyclic fibration: this ensures it becomes finer in the horizontal direction so as to guarantee the existence of local choices.) We denote by $d_{i,h}$ (resp.\ $d_{i,v}$) the face maps of $U$ in the horizontal (resp.\ vertical) direction, where the ``horizontal direction'' is the direction of $K(A,1)$.

  Let us use the conventions:
  \begin{enumerate}
  \item $a = \xi_1 \colon U_{1,0}\to A$,
  \item $(a,b) = \xi_2\colon U_{2,0} \to A\times A$,
  \item $(a,b,c) = \xi_3 \colon U_{3,0}\to A\times A\times A$,
  \end{enumerate}
  etcetera, where of course the meaning of the letters $a, b, c, \dotsc$ depends on the level $n$ and it is determined, in effect, by the simplicial face maps. Thus, at level $n=2$ we have
  \begin{align*}
    a &= d_2 \xi_2 = \xi_1 d_{h,2} \\
    b &= d_0 \xi_2 = \xi_1 d_{h,0} \\
    a+b  &= d_1 \xi_2 = \xi_1 d_{h,1},
  \end{align*}
  whereas at level $n=3$ we have, with the same conventions:
  \begin{align*}
    a &= d_2 d_3\xi_3 = \xi_1 d_{h,2} d_{h,3}\\
    b &= d_0 d_3\xi_3 = \xi_1 d_{h,0} d_{h,3}\\
    c &= d_0 d_1 \xi_3 = \xi_1 d_{h,0}d_{h,1} \\
    a + b &= d_1 d_3\xi_3 = \xi_1 d_{h,1} d_{h,3}\\
    b + c &= d_0 d_2\xi_3 = \xi_1 d_{h,0} d_{h,2}\\
    a + b +c &= d_1 d_2\xi_3 = \xi_1 d_{h,1} d_{h,2}.
  \end{align*}
  
  The choice of an object $X$ of the pullback $\stR_a = \xi_1^*\stR$ determines the class of $\stR$, as a gerbe over $A$ with band $M$, in $H^2(A,M)$. This class is obtained by computing it in the standard way along the hypercover $U_{1,\bullet}\to A$. As it is well known, the class itself is represented by an object of $\alpha\in M(U_{1,2})$, which is simply the defect in the commutativity of the triangle formed by the three possible pullbacks to $U_{1,2}$ of the isomorphism $\varphi\colon d_{1,v}^*X\isoto d_{0,v}^*X$ defined over $U_{1,1}$ \cite[see][]{MR95b:18009,MR95m:18006}. As a result, $\alpha$ is a Čech cocycle relative to the hypercover $U_{1,\bullet}$: 
  \begin{equation}
    \label{eq:26}
    d_{0,v}^*\alpha +d_{1,v}^*\alpha -d_{2,v}^*\alpha + d_{3,v}^*\alpha = 0\,.
  \end{equation}
  In the horizontal direction, by successively pulling back $X$ along the horizontal face maps at vertical level $n=0$ we obtain objects $X_a$, $X_{a+b}, \dotsc$ (recall the convention above) providing decompositions for the various pullbacks $d_{i,h}^*\dots d_{j,h}^*\stR$ over $A^n$.

  At level $n=2$, over $U_{2,0}$, we get an isomorphism
  \begin{equation*}
    \sigma_{a,b} \colon X_a + X_b \lisoto X_{a+b},
  \end{equation*}
  which is only compatible with the various isomorphisms $\varphi_a\colon d_{1,v}^*X_a\isoto d_{0,v}^*X_a$ up to an automorphism, say $\delta\in M(U_{2,1})$, of $d_{0,v}^*X_{a+b}$. A calculation of the pullbacks to $U_{2,2,}$ yields
  \begin{equation}
    \label{eq:77}
    d_{0,v}^*\delta -d_{1,v}^*\delta +d_{2,v}^*\delta =
    -d_{0,h}^*\alpha +d_{1,h}^*\alpha -d_{2,h}^*\alpha\,.
  \end{equation}
  Observe that \emph{modulo $\alpha$,} the element $\delta$ gives the class in $H^1(A\times A,M)$ of the sheaf $\Hom(X_a+X_b,X_{a+b})$.

  Also, still over $U_{2,0}$, the braiding of $\stR$ gives the diagram
  \begin{equation*}
    \xymatrix{
      X_a + X_b \ar[d]_{\sigma_{a,b}} \ar[r] & X_b + X_a \ar[d]^{\sigma_{b,a}} \\
      X_{a + b} \ar[r]_{g_+(a,b)} & X_{a + b} 
    }
  \end{equation*}
  defining $g_+(a,b)\in M(U_{2,0})$. If the braiding is in addition symmetric, then the previous diagram yields
  \begin{equation*}
    g_+(a,b) + g_+(b,a) = 0,
  \end{equation*}
  namely $g_+$ is antisymmetric under the map determined by the pullback of $U_{2,\bullet}$ by the swap map of $A^2$. If we denote this operation by $\tau$, then the above relation would be written $g_+ + \tau^*g_+=0$

  Similarly, an analysis of the associativity of the monoidal structure of $\stR$ yields $f_+(a,b,c)\in M(U_{3,0})$, as an automorphism of the object $X_{a+b+c}$ of $d_{2,h}^*d_{1,h}^*\stR$ over $A^3$. Under pullback of both $f_+$ and $g_+$ to $U_{3,1}$ we obtain:
  \begin{equation}
    \label{eq:78}
    \begin{aligned}
      d_{0,v}^*f_+ -d_{1,v}^*f_+ &= d_{0,h}^*\delta -d_{1,h}^*\delta +d_{2,h}^*\delta -d_{3,h}^*\delta \\
      d_{0,v}^*g_+ -d_{1,v}^*g_+ &= -\delta + \tau^*\delta.
    \end{aligned}
  \end{equation}
  Finally, the familiar relations describing a (symmetric) braiding arise by analyzing the behavior of $f_+$ when pulled back to $U_{4,0}$ and that of $g_+$ when pulled back to $U_{3,0}$, yielding the standard cocycle relation
  \begin{equation*}
    d_{0,h}^*f_+ -d_{1,h}^*f_+ +d_{2,h}^*f_+ -d_{3,h}^*f_+ + d_{4,h}^*f_+ = 0,
  \end{equation*}
  plus the two equations arising from Mac~Lane's hexagonal diagrams (and the antisymmetry condition above, if $\stR$ is symmetric).

  Again, observe that \emph{modulo $\delta$,} $f_+$ and $g_+$ define objects in $H^0(A^4,M)$ and $H^0(A^3,M)$ satisfying the standard cocycle relations for a class in $H^4(K(A,2),M)$ (or $H^5(K(A,3),M)$). This is the relevant part in the discussion of the class arising from a braided (or symmetric) monoidal stack as above. 

  In summary, the quadruplet $(f_+,g_+,\delta,\alpha)$ defines a cocycle of degree 3 with respect to the total complex defined as follows. Let $B(A,k)$ be the iterated bar construction on $A$, where $k=2,3$. If $k=1$, then $B(A,1)$ is simply the complex $\ZZ[K(A,1)]\sptilde$.\footnote{A down-shifted version is the complex denoted $L^2(A)$ or $L^3(A)$ in the paper.} With a mild abuse of language, for each (vertical) level $n$ let $B(U_{\bullet,n},k)$ denote the complex obtained in a way analogous to $B(A,k)$ from the abelian sheaves $\ZZ[U_{\bullet,n}]$.  Thus we get an augmented simplicial object
  \begin{equation*}
    \xymatrix{%
      \cdots \ar@<1ex>[r] \ar[r] \ar@<-1ex>[r]&
      B(U_{\bullet,1},k) \ar@<0.5ex>[r] \ar@<-0.5ex>[r] & B(U_{\bullet,0},k) \ar[r]^\xi  & B(A,k).
    }
  \end{equation*}
  Then we form (neglecting $B(A,k)$) the total complex with respect to the vertical direction to obtain the necessary complex.

  A glance at the complex used in the main text, section~\ref{sec:decomp}, namely the bar complex built on top of a DGA structure on the complexes $L^2(A)$ (or  $L^3(A)$),  reveals that the same construction with the hypercover
  \begin{equation*}
    U_{\bullet,\bullet} \to K(A,1)
  \end{equation*}
  would also work in this case, with the difference that in total degree 3 we have the quintuplet used in the paper, namely $(f,\alpha_1,\alpha_2,f_+,g_+)$ \emph{plus} the automorphisms $\alpha$ and $\delta$ as above arising from the decomposition of $\stR$ over $A$ and $A\times A$, supplemented by an additional one, $\epsilon\in M(U_{2,1})$, arising from the descent condition on the pullback of the various biextensions $E_{a,b}$ and their higher analogs.

  To accommodate for these changes, we must form simplicial objects corresponding to the bar complexes $\Bar B_{k,\bullet}(A)$, $k=2,3$, in section~\ref{sec:product-l2a}. That is, for each complex $B(U_{\bullet,n},k)$ above we apply the (normalized) bar construction $\Bar B_k(U_{\bullet,n}) \coloneqq \Bar B(B(U_{\bullet,n},k),\eta)$ to form the augmented simplicial object
  \begin{equation}
    \label{eq:81}
    \xymatrix{%
      \cdots \ar@<1ex>[r] \ar[r] \ar@<-1ex>[r]&
      \Bar B_k(U_{\bullet,1}) \ar@<0.5ex>[r] \ar@<-0.5ex>[r] & \Bar B_k(U_{\bullet,0}) \ar[r]^\xi  & \Bar B_{k,\bullet}(A)\,,
    }
  \end{equation}
  for $k=2,3$. In addition to our notational conventions, we must consider appropriate ones for the multiplicative structures: we consider products $ab=(d_2\xi_2)(d_0\xi_2)$, $abc=(d_2d_3\xi_3) (d_0d_3\xi_3) (d_0d_1\xi_3)$, etc.\ and use the same notation for the corresponding maps at any level in the vertical direction.

  To explain the required steps to obtain the full class in some more detail, consider the biexact bifunctors
  \begin{equation*}
    m_{a,b} \colon \stR_a\times \stR_b \lto \stR_{ab}
  \end{equation*}
  defined over $U_{2,0}$ (again, recall the above convention about naming the components of $\xi_1,\xi_2,\dots$) corresponding to the biextension $E_{a,b}$ introduced in sect.~\ref{sec:decomposition-str}, and their higher arity analogs. In terms of local objects, we have an isomorphism
  \begin{equation*}
    m_{a,b} \colon X_aX_b \lisoto X_{ab}\,,
  \end{equation*}
  which we can safely indicate with the same name, between objects of $\stR_{ab}$. Analogously to the case of the sum operation, $m_{a,b}$ is only compatible with the (horizontal) pullbacks of $\phi\colon d_{1,v}^*X\isoto d_{0,v}^*X$ up to an automorphism $\epsilon\in M(U_{2,1})$. For added clarity, let us specify the labels as $\epsilon(a,b)$. Comparing the three possible pullbacks to $U_{2,2}$ yields the relation
  \begin{equation}
    \label{eq:79}
    d_{0,v}^*\epsilon(a,b) -d_{1,v}^*\epsilon(a,b) +d_{2,v}^*\epsilon(a,b)
    = -a\alpha(b) + \alpha(ab) -\alpha(a)b\,,
  \end{equation}
  where a notation like $\alpha(ab)$ means the pullback of $\alpha\in M(U_{1,2})$ along the analog $\Bar B_k(U_{2,2})\to \Bar B_k(U_{1,2})$ of the map $ab$, as mentioned above.

  The remaining three relations are found by analyzing the behavior of the diagrams~\eqref{eq:65},~\eqref{eq:57} which define the quantities $f(a,b,c)$, $\alpha_1(a,b;c)$, and $\alpha_2(a;b,c)$, under pullback from $U_{3,0}$ to $U_{3,1}$. They can be readily computed by interpreting $f(a,b,c)$, $\alpha_1(a,b;c)$, and $\alpha_2(a;b,c)$ as 2-arrows in the diagrams
  
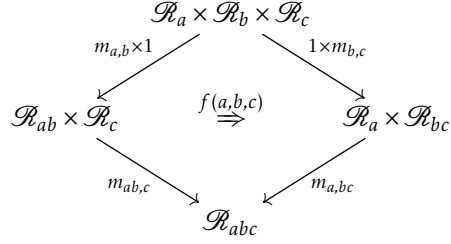
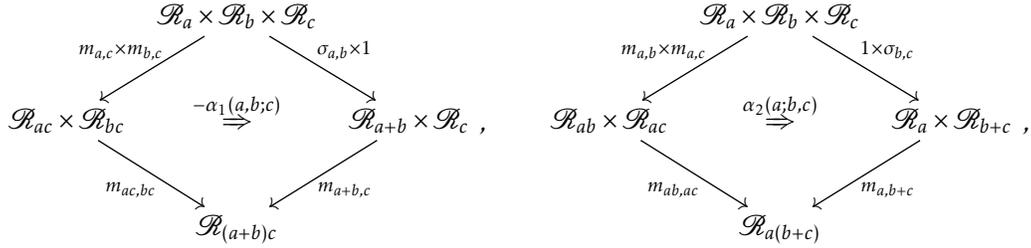
\begin{equation*}
    \xymatrix@C-1.5pc{%
      & \stR_a \times \stR_b\times \stR_c 
      \ar[dl]_{m_{a,b}\times 1} \ar[dr]^{1\times m_{b,c}} \\
      \stR_{ab} \times \stR_c \ar[dr]_{m_{ab,c}} &&
      \stR_a \times \stR_{bc} \ar[dl]^{m_{a,bc}} \\
      & \stR_{abc} \uutwocell<\omit>{f(a,b,c)}
    }
  \end{equation*}
  and
  \begin{equation*}
    \vcenter{%
      \xymatrix@C-1.5pc{%
        & \stR_a \times \stR_b\times \stR_c 
        \ar[dl]_{m_{a,c}\times m_{b,c}} \ar[dr]^{\sigma_{a,b}\times 1} \\
        \stR_{ac} \times \stR_{bc} \ar[dr]_{m_{ac,bc}} &&
        \stR_{a+b} \times \stR_{c} \ar[dl]^{m_{a+b,c}} \\
        & \stR_{(a+b)c} \uutwocell<\omit>{-\alpha_1(a,b;c)}
      }}\,,\qquad
    \vcenter{%
      \xymatrix@C-1.5pc{%
        & \stR_a \times \stR_b\times \stR_c 
        \ar[dl]_{m_{a,b}\times m_{a,c}} \ar[dr]^{1\times \sigma_{b,c}} \\
        \stR_{ab} \times \stR_{ac} \ar[dr]_{m_{ab,ac}} &&
        \stR_a \times \stR_{b+c} \ar[dl]^{m_{a,b+c}} \\
        & \stR_{a(b+c)} \uutwocell<\omit>{\alpha_2(a;b,c)}
      }}\,,
  \end{equation*}
  and then comparing, for each diagram, the two possible pullbacks to $U_{3,1}$. We obtain the following relations:
  \begin{equation}
    \label{eq:80}
    \begin{aligned}
      d_{0,v}^*f(a,b,c) -d_{1,v}^* f(a,b,c)
      &= -a\epsilon(b,c) + \epsilon(ab,c) -\epsilon(a,bc) +\epsilon(a,b)c \\
      d_{0,v}^*\alpha_1(a,b;c) -d_{1,v}^* \alpha_1(a,b;c)
      &= -\delta(ac,bc) + \delta(a,b)c -\epsilon(a,c) + \epsilon(a+b,c) -\epsilon(b,c) \\
      d_{0,v}^*\alpha_2(a,b;c) -d_{1,v}^* \alpha_2(a,b;c)
      &= \delta(ab,ac) - a\delta(b,c) +\epsilon(a,b) -\epsilon(a,b+c) +\epsilon(a,c)\,.
    \end{aligned}
  \end{equation}
  In summary, the relations~\eqref{eq:26}, \eqref{eq:77} and~\eqref{eq:79}, and \eqref{eq:78} and~\eqref{eq:80}, together with the ones found in the main text, sect.~\ref{sec:cohom-class}, over $U_{4,0}$, show the entire collection
  \begin{equation*}
    ((f,\alpha_1,\alpha_2,f_+,g_+);(\delta,-\epsilon),\alpha)
  \end{equation*}
  forms a cocycle of total degree 3 in the double complex obtained from~\eqref{eq:81} for $k=3$ (minus the term in degree $-1$, of course). In particular, the relations~\eqref{eq:78} and~\eqref{eq:80} feature the horizontal differential of $(\delta,-\epsilon)$, which has horizontal degree 2, according to sect.~\ref{sec:bar-construction}. Similarly, the relations~\eqref{eq:77}, \eqref{eq:79} feature the horizontal differential of $\alpha$, with horizontal degree 1.


\phantomsection  
\addcontentsline{toc}{section}{References}  
\printbibliography
\end{document}